\documentclass[11pt,reqno,twoside]{amsart}
\usepackage{amssymb,amsmath,amsthm,soul,color,paralist}
\usepackage{t1enc}
\usepackage[cp1250]{inputenc}
\usepackage{a4,indentfirst,latexsym}
\usepackage{graphics}
\usepackage{mathrsfs}
\usepackage{cite,enumitem,graphicx}
\usepackage[colorlinks=true,urlcolor=blue,
citecolor=red,linkcolor=blue,linktocpage,pdfpagelabels,
bookmarksnumbered,bookmarksopen]{hyperref}
\usepackage[english]{babel}
\usepackage[left=2.50cm,right=2.50cm,top=2.72cm,bottom=2.72cm]{geometry}
\usepackage[metapost]{mfpic}
\usepackage[hyperpageref]{backref}
\usepackage[colorinlistoftodos]{todonotes}
\usepackage[normalem]{ulem}

\makeatletter
\providecommand\@dotsep{5}
\def\listtodoname{List of Todos}
\def\listoftodos{\@starttoc{tdo}\listtodoname}
\makeatother

\numberwithin{equation}{section}

\newcommand{\e}{\varepsilon}
\newcommand{\R}{\mathbb{R}}
\newcommand{\N}{\mathbb{N}}

\DeclareMathOperator{\dive}{div}
\DeclareMathOperator{\supp}{supp}

\newtheorem{prop}{Proposition}[section]
\newtheorem{lem}{Lemma}[section]
\newtheorem{thm}{Theorem}[section]

\newtheorem{cor}{Corollary}[section]

\newtheorem{remark}{Remark}[section]

\begin{document}
\title[Concentrating solutions for fractional Schr\"odinger equations]{Concentrating solutions for 
a class of nonlinear fractional Schr\"odinger equations in $\R^{N}$}

\author[V. Ambrosio]{Vincenzo Ambrosio}
\address{Dipartimento di Ingegneria Industriale e Scienze Matematiche, Universit\`a Politecnica delle Marche, via Brecce Bianche n.12, 60131 Ancona (Italy)}
%\address{Dipartimento di Scienze Pure e Applicate (DiSPeA),
%Universit\`a degli Studi di Urbino `Carlo Bo'
%Piazza della Repubblica, 13
%61029 Urbino (Pesaro e Urbino, Italy)}
\email{v.ambrosio@univpm.it}


\begin{abstract}
We deal with the existence of positive solutions for the following fractional Schr\"odinger equation:
$$
\e ^{2s} (-\Delta)^{s} u + V(x) u = f(u) \mbox{ in } \R^{N},
$$
where $\varepsilon>0$ is a parameter, $s\in (0, 1)$, $N\geq 2$, $(-\Delta)^{s}$ is the fractional Laplacian operator, and $V:\R^{N}\rightarrow \R$ is a positive continuous function. Under the assumptions that the nonlinearity $f$ is either asymptotically linear or superlinear at infinity, we prove the existence of a family of positive solutions which concentrates at a local minimum
of $V$ as $\e$ tends to zero. 
\end{abstract}

\maketitle
\section{Introduction}

In this paper we investigate the existence and the concentration phenomenon of positive solutions for the following fractional equation:
\begin{equation}\label{P}
\varepsilon ^{2s} (-\Delta)^{s} u + V(x) u = f(u) \mbox{ in } \R^{N},
\end{equation}
where $\e>0$ is a parameter, $s\in (0, 1)$ and $N\geq 2$. \\
The external potential $V: \R^{N}\rightarrow \R$ is a locally H{\"o}lder continuous function and bounded below away from zero, that is, there exists $V_{0} >0$ such that
\begin{equation}\label{2.2}
V(x)\geq V_{0}>0 \quad \mbox{ for all } x\in \R^{N}.
\end{equation}
Concerning the nonlinearity $f:\R\rightarrow \R$, we assume that it satisfies the following basic assumptions:
\begin{compactenum}[($f1$)]
\item $f\in C^{1}(\R, \R)$; 
\item  $\lim_{t\rightarrow 0} \frac{f(t)}{t}=0$; 
\item there exists $p \in (1, \frac{N+2s}{N-2s})$ such that $\lim_{t\rightarrow \infty}\frac{f(t)}{t^p}=0$. 
\end{compactenum}
The nonlocal operator $(-\Delta)^{s}$ appearing in (\ref{P}) is the so-called fractional Laplacian, which can be defined, for any $u: \R^{N}\rightarrow \R$ smooth enough, by setting
$$
(-\Delta)^{s}u(x)=-\frac{C(N,s)}{2} \int_{\R^{N}} \frac{u(x+y)+u(x-y)-2u(x)}{|y|^{N+2s}} dy  \quad (x\in \R^{N}),
$$
where $C(N, s)$ is a dimensional constant depending only on $N$ and $s$; see  \cite{DPV}.\\
In the last decade, great attention has been devoted to the study of nonlinear elliptic problems involving fractional operators, due to their intriguing analytic structure and specially in view of several applications in many areas of the research such as crystal dislocation, finance, phase transitions, material sciences, chemical reactions, minimal surfaces, etc. For more details and applications on this subject we refer the interested reader to \cite{DPV, MBRS}.\\
One of the main reasons of studying (\ref{P}) is the search of standing wave solutions $\psi(t, x)=u(x) e^{-\frac{\imath c t}{\hbar}}$ for the following time-dependent fractional Schr\"odinger equation
\begin{equation}\label{2.3}
i \hbar \frac{\partial \Phi}{\partial t} = \frac{\hbar^{2}}{2m} (-\Delta)^{s} \Phi +
W(x) \Phi - g(|\Phi|)\Phi \mbox{ for } (t, x)\in \R\times \R^{N}. 
\end{equation}
Equation (\ref{2.3}) has been derived by Laskin in \cite{Laskin1, Laskin2}, and plays a fundamental role in quantum mechanics in the study of particles on stochastic fields modeled by L\'evy processes.

When $s=1$, equation (\ref{P}) becomes the classical Schr\"odinger equation
\begin{align}\label{Pe}
-\varepsilon ^{2} \Delta u + V(x) u = f(u) \mbox{ in } \R^{N},
\end{align}
for which the existence and the multiplicity of solutions has been extensively studied in the last thirty years by many authors; see \cite{AF, AMS, BW, BL1, FW, Oh, Rab, Wang}. \\
Rabinowitz in \cite{Rab} investigated the existence of positive solutions to (\ref{Pe}) for $\e>0$ small enough, under the assumption that $f$ satisfies the well-known Ambrosetti-Rabinowitz condition \cite{ambrosetti}, that is,

\begin{compactenum}[$(f4)$]
\item there exists $\mu >2$ such that $0< \mu F(t) \leq f(t)t $ for any $t>0$,
\end{compactenum}
where $\displaystyle{F(t)= \int_{0}^{t} f(\tau) d\tau}$, and the potential $V(x)$ satisfies the following global condition:
$$
\liminf_{|x|\rightarrow \infty} V(x)>\inf_{x\in \R^{N}} V(x).
$$
Wang \cite{Wang} showed that these solutions concentrate at global minimum points of $V(x)$.
Using a local mountain pass approach, Del Pino and Felmer in \cite{DF}, proved the existence of a single spike solution to (\ref{Pe}) which concentrates around a local minimum of $V$, by assuming that 
there exists a bounded open set $\Lambda$ in $\R^{N}$ such that
$$
\inf_{x\in \Lambda} V(x)<\min_{x\in \partial \Lambda} V(x),
$$
and considering nonlinearities $f$ satisfying $(f4)$ and the monotonicity assumption on $t \mapsto \frac{f(t)}{t}$.\\
Subsequently, Jeanjean and Tanaka \cite{JT2} introduced new variational methods to extend the results obtained in \cite{DF}, to a wider class of nonlinearities.

In the non-local setting, there are only few results concerning the existence and the concentration phenomena of solutions for the fractional equation (\ref{P}), maybe because many important techniques developed in the local framework cannot be adapted so easily to the fractional case.\\
Next, we recall some fundamental results related to the concentration phenomenon of solutions for the nonlinear fractional Schr\"odinger  equation (\ref{P}), obtained in recent years.

Chen and Zheng \cite{CZ}  studied, via the Liapunov-Schmidt reduction method, the concentration phenomenon for solutions of (\ref{P}) with $f(t)=|t|^{\alpha}t$, and under suitable limitations on the dimension $N$ of the space and the fractional powers $s$. Davila et al. \cite{DDPW} showed that if the potential $V$ satisfies
$$
V\in C^{1, \alpha}(\R^{N})\cap L^{\infty}(\R^{N}) \mbox{ and } \inf_{x\in \R^{N}} V(x)>0,
$$
then (\ref{P}) has multi-peak  solutions.
Fall et al. \cite{FFV} established necessary and sufficient conditions on the smooth potential $V$ in order to produce concentration of solutions of (\ref{P}) when the parameter $\e$ converges to zero.  In particular, when $V$ is coercive and has a unique global minimum, then ground-states concentrate at this point.
Alves and Miyagaki \cite{AM} investigated the existence and the concentration of positive solutions to (\ref{P}), via a penalization approach, under condition $(f4)$ and the assumption $f(t)/t$ is increasing in $(0, \infty)$. He and Zou \cite{HZ} used variational methods and the Ljusternik-Schnirelmann theory to  study \eqref{P} when $f(t)=g(t)+t^{2^{*}_{s}-1}$ and $g$ satisfies $(f4)$ and the monotonicity assumption on $g(t)/t$.
In \cite{A3} the author extended the results in \cite{AM} and \cite{HZ} obtaining the existence and the multiplicity of solutions to \eqref{P} when $f$ has subcritical or supercritical growth.
Finally, we would like also to mention to the papers \cite{A, A1, A2, A4, CW, CN, DMV, DPPV, FQT, FLS, Secchi1, Secchi2, SZ2} in which the existence and the multiplicity of solutions for different nonlinear fractional Schr\"odinger equations has been investigated by using several variational approaches.

Motivated by the above papers, in this work we aim to study the existence of positive solutions to (\ref{P}) concentrating around local minima of the potential $V(x)$, under the assumptions that the nonlinearity $f$ is asymptotically linear or superlinear at infinity, and without supposing the monotonicity of $f(t)/t$. We recall that the hypothesis $(f4)$ and the assumption $f(t)/t$ is increasing have a fundamental role in \cite{AM, A, HZ} to verify the boundedness of Palais-Smale sequences and to apply Nehari manifold arguments, respectively.

Now, we state our main result:
\begin{thm}\label{teorema principale}

Let us assume that $f(t)$ satisfies $(f1)$-$(f3)$ and either $(f4)$ 
or the following condition $(f5)$: 
\begin{compactenum}[(i)]
\item There exists $a\in(0, \infty]$ such that  $\lim_{t\rightarrow \infty}\frac{f(t)}{t}=a$.
\item There exists a constant $D\geq 1$ such that 
\begin{equation} \label{2.4}
\hat{F}(t)\leq D \hat{F}(\bar{t}) \quad 0\leq t\leq \bar{t},
\end{equation}
          where $\hat{F}(t)= \frac{1}{2}f(t)t - F(t)$.
\end{compactenum}
Let $\Lambda \subset \R^{N}$ be a bounded open set such that 
\begin{equation}\label{2.5}
\inf_{\Lambda} V < \min_{\partial \Lambda} V
\end{equation} 
and, when $a< \infty$ in $(f5)$, 
\begin{equation}\label{2.6}
\inf_{\Lambda}V< a.
\end{equation}
Then, there exists $\varepsilon_{0}>0$ such that, for any $\varepsilon \in (0, \varepsilon_{0}]$,
 equation (\ref{P}) admits a positive solution $u_{\varepsilon}(x)$. Moreover, if $x_{\e}$ denotes a global maximum point of $u_{\e}$, then we have
\begin{compactenum}[$(1)$]
\item $V(x_{\varepsilon})\rightarrow \inf_{x\in \Lambda} V(x)$;
\item there exists $C >0$ such that
      $$
      u_{\varepsilon}(x)\leq \frac{C \varepsilon^{N+2s}}{\varepsilon^{N+2s}+ |x-x_{\varepsilon}|^{N+2s}} 
      \quad \mbox{ for all } x\in \R^{N}. 
      $$
\end{compactenum}
\end{thm}

\medskip

\noindent
A common approach to tackle fractional nonlocal problems, is to make use of the extension method due to Caffarelli and Silvestre \cite{CS}, which allows us to transform a given nonlocal equation into a degenerate elliptic problem in the half-space with a nonlinear Neumann boundary condition. In this work, we prefer to investigate \eqref{P} directly in $H^{s}(\R^{N})$ in order to adapt to our framework some ideas used in \cite{JT2}. 
Anyway, the presence of the fractional Laplacian $(-\Delta)^{s}$, which is a nonlocal operator, induces several technical difficulties that will be overcome by developing some clever and appropriate arguments.
%Anyway, the presence of the fractional Laplacian $(-\Delta)^{s}$, induces several technical difficulties which make our analysis rather tough. Therefore, some careful arguments will be developed to overcome the nonlocal character of $(-\Delta)^{s}$.
% that will be overcame by using some accurate arguments which take care of the nonlocal character of $(-\Delta)^{s}$.
%A common approach to tackle fractional nonlocal problems, is to make use of the extension method due to Caffarelli and Silvestre \cite{CS}, which allows us to transform a given nonlocal equation into a degenerate elliptic problem in the half-space with a nonlinear Neumann boundary condition. In this work, we do not follow this approach, and we prefer to investigate the problem directly in $H^{s}(\R^{N})$, in order to adapt to our framework some ideas developed in \cite{JT2}. 
%Clearly, due to the presence of the fractional Laplacian $(-\Delta)^{s}$, which is a nonlocal operator, a more accurate analysis is needed.
%Anyway, due to the presence of the fractional Laplacian $(-\Delta)^{s}$, which is a nonlocal operator, a more careful analysis will be needed and some refined estimates will be done.

We would like to note that Theorem \ref{teorema principale} extends and improves the  result in \cite{AM}, because we do not require any monotonicity assumption on $f(t)/t$, and we are able to deal with a more general class of nonlinearities, including the asymptotically linear case (see condition $(f5)$). Moreover, our result is in clear accordance with that for the classical local counterpart, that is Theorem $1.1$ in \cite{JT2}. \\
We also point out that in contrast with the case $s=1$, the decay at infinity of solutions of (\ref{P}) is of power-type and not exponential; see \cite{FQT}.\\
Now, we give the main ideas for the proof of Theorem \ref{teorema principale}.
After rescaling equation (\ref{P}) with the change of variable $v(x)=u(\e x)$, we introduce a modified functional $J_{\e}$ and we prove that it satisfies a mountain pass geometry \cite{ambrosetti}. Then, we investigate the boundedness of Cerami sequences for $J_{\e}$, and we give two types of boundedness results: one when $\e$ is fixed, the other one to deduce uniform boundedness when $\e\rightarrow 0$. Through a careful study of the behavior as $\e\rightarrow 0$ of bounded Cerami sequences $(v_{\e})$, we prove that  there exists a  subsequence $(v_{\e_{j}})$ which converges, in a suitable sense, to a sum of translated critical points of certain autonomous functionals. This concentration-compactness type result will be useful to show that an appropriate translated sequence $v_{\e_{j}}(\cdot+y_{\e_{j}})$ converges to a least energy solution $\omega^{1}$. Thus, we exploit some results obtained in \cite{FQT} to deduce $L^{\infty}$-estimates (uniformly in $j\in \N$) and some information about the behavior at infinity of the translated sequence, which permit to obtain a positive solution of the rescaled equation.

\medskip

\noindent
The outline of the paper is the following: in Section $2$ we collect some preliminary results concerning the fractional Sobolev spaces and we introduce the variational setting. Moreover, we study the modified functionals $J_{\e}$. In Section $3$ we present some fundamental properties related to autonomous functionals. In Section $4$ we give a concentration-compactness type result. In the last section we provide the proof of Theorem \ref{teorema principale}.

\section{Preliminaries and functional setting}

\subsection{Fractional Sobolev spaces and some useful Lemmas}
\noindent
In this section we briefly recall some properties of the fractional Sobolev spaces, and we introduce some notations which we will use along the paper.\\
For any $s\in (0,1)$, we denote by $\mathcal{D}^{s, 2}(\R^{N})$ the completion of the set $C^{\infty}_{0}(\R^{N})$ consisting of the infinitely differentiable functions $u: \R^{N}\rightarrow \R$ with compact support, with respect to the following norm
$$
[u]^{2}=\iint_{\R^{2N}} \frac{|u(x)-u(y)|^{2}}{|x-y|^{N+2s}} \, dx \, dy =\|(-\Delta)^{\frac{s}{2}} u\|^{2}_{L^{2}(\R^{N})},
$$
where the second identity holds up to a positive constant. Equivalently,
$$
\mathcal{D}^{s, 2}(\R^{N})=\left\{u\in L^{2^{*}_{s}}(\R^{N}): [u]<\infty\right\}.
$$
Let us also define the fractional Sobolev space
$$
H^{s}(\R^{N})= \left\{u\in L^{2}(\R^{N}) : \frac{|u(x)-u(y)|}{|x-y|^{\frac{N+2s}{2}}} \in L^{2}(\R^{2N}) \right \}
$$
endowed with the natural norm 
$$
\|u\|_{H^{s}(\R^{N})}:= \sqrt{[u]^{2} + \|u\|_{L^{2}(\R^{N})}^{2}}.
$$

\noindent
For the convenience of the reader we recall the following fundamental embeddings:
\begin{thm}\cite{DPV}\label{Sembedding}
Let $s\in (0,1)$ and $N>2s$. Then there exists a sharp constant $S_{*}=S(N, s)>0$
such that for any $u\in \mathcal{D}^{s, 2}(\R^{N})$
\begin{equation}\label{FSI}
\|u\|^{2}_{L^{2^{*}_{s}}(\R^{N})} \leq S^{-1}_{*} [u]^{2}. 
\end{equation}
Moreover $H^{s}(\R^{N})$ is continuously embedded in $L^{q}(\R^{N})$ for any $q\in [2, 2^{*}_{s}]$ and compactly in $L^{q}_{loc}(\R^{N})$ for any $q\in [2, 2^{*}_{s})$. 
\end{thm}

\noindent
Now, we prove the following technical result which will be useful in the sequel.
\begin{lem}\label{funlemma}
Let $(w_{j})\subset H^{s}(\R^{N})$ be a bounded sequence in $H^{s}(\R^{N})$, and let $\eta\in C^{\infty}(\R^{N})$ be a function such that $0\leq \eta\leq 1$, $\eta=0$ in $B_{1}$, $\eta=1$ in $\R^{N}\setminus B_{2}$. Set $\eta_{R}(x)=\eta(\frac{x}{R})$.
Then we get
$$
\lim_{R\rightarrow \infty}\limsup_{j\rightarrow \infty} \iint_{\R^{2N}} |w_{j}(x)|^{2} \frac{|\eta_{R}(x)-\eta_{R}(y)|^{2}}{|x-y|^{N+2s}} \, dx dy=0.
$$
\end{lem}
\begin{proof}
Let us note that $\R^{2N}$ can be written as 
\begin{align*}
\R^{2N}&=((\R^{N}\setminus B_{2R})\times (\R^{N}\setminus B_{2R})) \cup ((\R^{N}\setminus B_{2R})\times B_{2R})\cup (B_{2R}\times \R^{N}) \\
&=: X^{1}_{R}\cup X^{2}_{R} \cup X^{3}_{R}.
\end{align*}
Then
\begin{align}\label{Pa1}
&\iint_{\R^{2N}}\frac{|\eta_{R}(x)-\eta_{R}(y)|^{2}}{|x-y|^{N+2s}} |w_{j}(x)|^{2} dx dy =\iint_{X^{1}_{R}}\frac{|\eta_{R}(x)-\eta_{R}(y)|^{2}}{|x-y|^{N+2s}} |w_{j}(x)|^{2} dx dy \nonumber \\
&+\iint_{X^{2}_{R}}\frac{|\eta_{R}(x)-\eta_{R}(y)|^{2}}{|x-y|^{N+2s}} |w_{j}(x)|^{2} dx dy+
\iint_{X^{3}_{R}}\frac{|\eta_{R}(x)-\eta_{R}(y)|^{2}}{|x-y|^{N+2s}} |w_{j}(x)|^{2} dx dy.
\end{align}
Now, we estimate each integral in (\ref{Pa1}).
Since $\eta_{R}=1$ in $\R^{N}\setminus B_{2R}$, we have
\begin{align}\label{Pa2}
\iint_{X^{1}_{R}}\frac{|w_{j}(x)|^{2} |\eta_{R}(x)-\eta_{R}(y)|^{2}}{|x-y|^{N+2s}} dx dy=0.
\end{align}
Let $k>4$. Clearly, we have
\begin{equation*}
X^{2}_{R}=(\R^{N} \setminus B_{2R})\times B_{2R} \subset ((\R^{N}\setminus B_{kR})\times B_{2R})\cup ((B_{kR}\setminus B_{2R})\times B_{2R}) 
\end{equation*}
Let us observe that, if $(x, y) \in (\R^{N}\setminus B_{kR})\times B_{2R}$, then
\begin{equation*}
|x-y|\geq |x|-|y|\geq |x|-2R>\frac{|x|}{2}. 
\end{equation*}
Therefore, taking into account that $0\leq \eta_{R}\leq 1$, $|\nabla \eta_{R}|\leq \frac{C}{R}$ and applying H\"older's inequality, we can see
\begin{align}\label{Pa3}
&\iint_{X^{2}_{R}}\frac{|w_{j}(x)|^{2} |\eta_{R}(x)-\eta_{R}(y)|^{2}}{|x-y|^{N+2s}} dx dy \nonumber \\
&=\int_{\R^{N}\setminus B_{kR}} \int_{B_{2R}} \frac{|w_{j}(x)|^{2} |\eta_{R}(x)-\eta_{R}(y)|^{2}}{|x-y|^{N+2s}} dx dy \nonumber \\
&+ \int_{B_{kR}\setminus B_{2R}} \int_{B_{2R}} \frac{|w_{j}(x)|^{2} |\eta_{R}(x)-\eta_{R}(y)|^{2}}{|x-y|^{N+2s}} dx dy \nonumber \\
&\leq 2^{2+N+2s} \int_{\R^{N}\setminus B_{kR}} \int_{B_{2R}} \frac{|w_{j}(x)|^{2}}{|x|^{N+2s}}\, dxdy\nonumber \\
&+ \frac{C}{R^{2}} \int_{B_{kR}\setminus B_{2R}} \int_{B_{2R}} \frac{|w_{j}(x)|^{2}}{|x-y|^{N+2(s-1)}}\, dxdy \nonumber \\
&\leq CR^{N} \int_{\R^{N}\setminus B_{kR}} \frac{|w_{j}(x)|^{2}}{|x|^{N+2s}}\, dx + \frac{C}{R^{2}} (kR)^{2(1-s)} \int_{B_{kR}\setminus B_{2R}} |w_{j}(x)|^{2} dx \nonumber \\
&\leq CR^{N} \left( \int_{\R^{N}\setminus B_{kR}} |w_{j}(x)|^{2^{*}_{s}} dx \right)^{\frac{2}{2^{*}_{s}}} \left(\int_{\R^{N}\setminus B_{kR}}\frac{1}{|x|^{\frac{N^{2}}{2s} +N}}\, dx \right)^{\frac{2s}{N}} \nonumber \\
&+ \frac{C k^{2(1-s)}}{R^{2s}} \int_{B_{kR}\setminus B_{2R}} |w_{j}(x)|^{2} dx \nonumber \\
&\leq \frac{C}{k^{N}} \left( \int_{\R^{N}\setminus B_{kR}} |w_{j}(x)|^{2^{*}_{s}} dx \right)^{\frac{2}{2^{*}_{s}}} + \frac{C k^{2(1-s)}}{R^{2s}} \int_{B_{kR}\setminus B_{2R}} |w_{j}(x)|^{2} dx \nonumber \\
&\leq \frac{C}{k^{N}}+ \frac{C k^{2(1-s)}}{R^{2s}} \int_{B_{kR}\setminus B_{2R}} |w_{j}(x)|^{2} dx.
\end{align}

\noindent
Fix $\e\in (0,1)$. Notice that
\begin{align}\label{Ter1}
\iint_{X^{3}_{R}} &\frac{|w_{j}(x)|^{2} |\eta_{R}(x)- \eta_{R}(y)|^{2}}{|x-y|^{N+2s}}\, dxdy \nonumber\\
&\leq \int_{B_{2R}\setminus B_{\varepsilon R}} \int_{\R^{N}} \frac{|w_{j}(x)|^{2} |\eta_{R}(x)- \eta_{R}(y)|^{2}}{|x-y|^{N+2s}}\, dxdy 
\nonumber \\
&\quad + \int_{B_{\varepsilon R}} \int_{\R^{N}} \frac{|w_{j}(x)|^{2} |\eta_{R}(x)- \eta_{R}(y)|^{2}}{|x-y|^{N+2s}}\, dxdy. 
\end{align}
Since
\begin{align*}
\int_{B_{2R}\setminus B_{\varepsilon R}} \int_{\R^{N} \cap \{y: |x-y|<R\}} &\frac{|w_{j}(x)|^{2} |\eta_{R}(x)- \eta_{R}(y)|^{2}}{|x-y|^{N+2s}}\, dxdy \\
&\leq \frac{C}{R^{2s}} \int_{B_{2R}\setminus B_{\varepsilon R}} |w_{j}(x)|^{2} dx,
\end{align*}
and 
\begin{align*}
\int_{B_{2R}\setminus B_{\varepsilon R}} \int_{\R^{N} \cap \{y: |x-y|\geq R\}} &\frac{|w_{j}(x)|^{2} |\eta_{R}(x)- \eta_{R}(y)|^{2}}{|x-y|^{N+2s}}\, dxdy \\
&\leq \frac{C}{R^{2s}} \int_{B_{2R}\setminus B_{\varepsilon R}} |w_{j}(x)|^{2} dx,
\end{align*}
we can infer that
\begin{align}\label{Ter2}
\int_{B_{2R}\setminus B_{\varepsilon R}} \int_{\R^{N}} \frac{|w_{j}(x)|^{2} |\eta_{R}(x)- \eta_{R}(y)|^{2}}{|x-y|^{N+2s}}\, dxdy \leq \frac{C}{R^{2s}} \int_{B_{2R}\setminus B_{\varepsilon R}} |w_{j}(x)|^{2} dx. 
\end{align}
Now, using the definition of $\eta_{R}$, $\e\in (0,1)$, and $0\leq \eta_{R}\leq 1$, we get 
\begin{align}\label{Ter3}
\int_{B_{\varepsilon R}} \int_{\R^{N}} &\frac{|w_{j}(x)|^{2} |\eta_{R}(x)- \eta_{R}(y)|^{2}}{|x-y|^{N+2s}} dxdy \nonumber \\
&= \int_{B_{\varepsilon R}} \int_{\R^{N}\setminus B_{R}} \frac{|w_{j}(x)|^{2} |\eta_{R}(x)- \eta_{R}(y)|^{2}}{|x-y|^{N+2s}}\, dxdy\nonumber \\
&\leq 4 \int_{B_{\varepsilon R}} \int_{\R^{N}\setminus B_{R}} \frac{|w_{j}(x)|^{2}}{|x-y|^{N+2s}} dxdy\nonumber \\
&\leq C \int_{B_{\varepsilon R}} |w_{j}(x)|^{2} dx \int_{(1-\e)R}^{\infty} \frac{1}{r^{1+2s}} dr\nonumber \\
&=\frac{C}{[(1-\e)R]^{2s}} \int_{B_{\varepsilon R}} |w_{j}(x)|^{2} dx,
\end{align}
where we used the fact that $|x-y|>(1-\e)R$ when $(x, y) \in B_{\varepsilon R}\times (\R^{N} \setminus B_{R})$.
Taking into account \eqref{Ter1}, \eqref{Ter2} and \eqref{Ter3} we deduce that
\begin{align}\label{Pa4}
\iint_{X^{3}_{R}} &\frac{|w_{j}(x)|^{2} |\eta_{R}(x)- \eta_{R}(y)|^{2}}{|x-y|^{N+2s}}\, dxdy \nonumber \\
&\leq \frac{C}{R^{2s}} \int_{B_{2R}\setminus B_{\varepsilon R}} |w_{j}(x)|^{2} dx + \frac{C}{[(1-\e)R]^{2s}} \int_{B_{\varepsilon R}} |w_{j}(x)|^{2} dx. 
\end{align}
Putting together \eqref{Pa1}, \eqref{Pa2}, \eqref{Pa3} and \eqref{Pa4} we obtain
\begin{align}\label{Pa5}
\iint_{\R^{2N}} &\frac{|w_{j}(x)|^{2} |\eta_{R}(x)- \eta_{R}(y)|^{2}}{|x-y|^{N+2s}}\, dxdy \nonumber \\
&\leq \frac{C}{k^{N}} + \frac{Ck^{2(1-s)}}{R^{2s}} \int_{B_{kR}\setminus B_{2R}} |w_{j}(x)|^{2} dx + \frac{C}{R^{2s}} \int_{B_{2R}\setminus B_{\varepsilon R}} |w_{j}(x)|^{2} dx \nonumber \\
&+ \frac{C}{[(1-\e)R]^{2s}}\int_{B_{\varepsilon R}} |w_{j}(x)|^{2} dx. 
\end{align}
Since $(w_{j})$ is bounded in $H^{s}(\R^{N})$, by Theorem \ref{Sembedding}, 
we may assume that $w_{j}\rightarrow w$ in $L^{2}_{loc}(\R^{N})$ for some $w\in H^{s}(\R^{N})$. 
Then, taking the limit as $j\rightarrow \infty$ in \eqref{Pa5} and applying H\"older's inequality we have
\begin{align}\label{XZZ}
&\limsup_{j\rightarrow \infty} \iint_{\R^{2N}} \frac{|w_{j}(x)|^{2} |\eta_{R}(x)- \eta_{R}(y)|^{2}}{|x-y|^{N+2s}}\, dxdy \nonumber\\
&\leq \frac{C}{k^{N}} + \frac{Ck^{2(1-s)}}{R^{2s}} \int_{B_{kR}\setminus B_{2R}} |w(x)|^{2} dx + \frac{C}{R^{2s}} \int_{B_{2R}\setminus B_{\varepsilon R}} |w(x)|^{2} dx \nonumber \\
&+ \frac{C}{[(1-\e)R]^{2s}}\int_{B_{\varepsilon R}} |w(x)|^{2} dx \nonumber\\
&\leq \frac{C}{k^{N}} + Ck^{2} \left( \int_{B_{kR}\setminus B_{2R}} |w(x)|^{2^{*}_{s}} dx\right)^{\frac{2}{2^{*}_{s}}} + C\left(\int_{B_{2R}\setminus B_{\varepsilon R}} |w(x)|^{2^{*}_{s}} dx\right)^{\frac{2}{2^{*}_{s}}} \nonumber \\
&+ C\left( \frac{\e}{1-\e}\right)^{2s} \left(\int_{B_{\varepsilon R}} |w(x)|^{2^{*}_{s}} dx\right)^{\frac{2}{2^{*}_{s}}}. 
\end{align}
By $w\in L^{2^{*}_{s}}(\R^{N})$, $k>4$ and $\e \in (0,1)$ we can note that
\begin{align*}
\lim_{R\rightarrow \infty} \int_{B_{kR}\setminus B_{2R}} |w(x)|^{2^{*}_{s}} dx = \lim_{R\rightarrow \infty} \int_{B_{2R}\setminus B_{\varepsilon R}} |w(x)|^{2^{*}_{s}} dx = 0. 
\end{align*}
Choosing $\e= \frac{1}{k}$ in \eqref{XZZ} we get
\begin{align*}
&\limsup_{R\rightarrow \infty} \limsup_{j\rightarrow \infty} \iint_{\R^{2N}} \frac{|w_{j}(x)|^{2} |\eta_{R}(x)- \eta_{R}(y)|^{2}}{|x-y|^{N+2s}}\, dxdy\\
&\quad \leq \lim_{k\rightarrow \infty}  \left[\, \frac{C}{k^{N}}
+ C\left(\frac{1}{k-1}\right)^{2s} \left(\int_{B_{\frac{1}{k}R}} |w(x)|^{2^{*}_{s}} dx\right)^{\frac{2}{2^{*}_{s}}}\, \right]=0.
\end{align*}
\end{proof}

\noindent
Let us introduce the space of radial functions in $H^{s}(\R^{N})$
$$
H^{s}_{r}(\R^{N})=\left \{u\in H^{s}(\R^{N}): u(x)=u(|x|)\right\}. 
$$
Related to this space, we have the following fundamental compactness result due to Lions \cite{Lions}:
\begin{thm}\cite{Lions}\label{Lions}
Let $s\in (0,1)$ and $N\geq 2$. Then $H^{s}_{r}(\R^{N})$ is compactly embedded in $L^{q}(\R^{N})$ for any $q\in (2, 2^{*}_{s})$.
\end{thm}

\noindent
Finally, we recall the following two useful lemmas:
\begin{lem}\cite{Secchi1}\label{lionslemma}
Let $N>2s$ and $r\in [2, 2^{*}_{s})$. If $(u_{j})$ is a bounded sequence in $H^{s}(\R^{N})$ and if
$$
\lim_{j \rightarrow \infty} \sup_{y\in \R^{N}} \int_{B_{R}(y)} |u_{n}|^{r} dx=0
$$
for some $R>0$,
then $u_{j}\rightarrow 0$ in $L^{t}(\R^{N})$ for all $t\in (2, 2^{*}_{s})$.
\end{lem}

\begin{lem}\cite{CW}\label{Strauss}
Let $(X, \|\cdot\|_{X})$ be a Banach space such that $X$ is continuously and compactly embedded into $L^{q}(\R^{N})$ for $q\in [q_{1}, q_{2}]$ and $q\in (q_{1}, q_{2})$, respectively, where $q_{1}, q_{2}\in (0, \infty)$.
Assume that $(u_{j})\subset X$, $u: \R^{N} \rightarrow \R$ is a measurable function and $P\in C(\R, \R)$ is such that
\begin{compactenum}[(i)]
\item $\displaystyle{\lim_{|t|\rightarrow 0} \frac{P(t)}{|t|^{q_{1}}}=0}$, \\
\item $\displaystyle{\lim_{|t|\rightarrow \infty} \frac{P(t)}{|t|^{q_{2}}}=0}$,\\
\item $\displaystyle{\sup_{j\in \N} \|u_{j}\|_{X}<\infty}$,\\
\item $\displaystyle{\lim_{j \rightarrow \infty} P(u_{j}(x))=u(x)} \mbox{ for a.e. } x\in \R^{N}$.
\end{compactenum}
Then, up to a subsequence, we have
$$
\lim_{j\rightarrow \infty} \|P(u_{j})-u\|_{L^{1}(\R^{N})}=0.
$$
\end{lem}

\subsection{Modification of the nonlinearity}
Since we are looking for positive solutions of (\ref{P}), we can suppose that $f(t)= 0$ for any $t \leq 0$.\\
Arguing as in \cite{JT2}, we can prove the following useful properties of the function $f$:
\begin{lem}\label{lemma 2.1}
Assume that $(f1)$-$(f3)$ hold.  Then we have:
\begin{compactenum}[(i)]
\item For all $\delta >0$ there exists $C_{\delta}>0$ such that
      \begin{equation}\label{2.7}
      |f(t)|\leq \delta |t|+ C_{\delta} |t|^{p} \, \mbox{ for all } t\in\R.
      \end{equation}
\item If $(f4)$ holds, then $f(t)\geq 0$ for all $t\geq 0$.
\item If $(f5)$ holds, then $f(t)\geq 0$, $\hat{F}(t)\geq 0$, 
      $\frac{d}{dt}(\frac{F(t)}{t^{2}})\geq 0$  for all $t \geq 0$.
\item If $t\mapsto \frac{f(t)}{t}$ is nondecreasing for $t\in (0, \infty)$, then $(f5)$ is satisfied with $D=1$.
\end{compactenum}
\end{lem}

\noindent
Now, let us suppose that $f(t)$ satisfies $(f1)$-$(f3)$ and that 
$$
V_{0}< a= \lim_{\xi \rightarrow \infty} \frac{f(t)}{t} \in (0, \infty].
$$
Take $\nu \in (0, \frac{V_{0}}{2})$ and we define
$$
\underline{f}(t):=
\begin{cases}
\min \{f(t), \nu t\} & \text{ if $t \geq 0$} \\
0                        & \text{ if $t <0$}.
\end{cases}
$$ 
Using $(f2)$ we can find $r_{\nu} >0$ such that 
$$
\underline{f}(t) = f(t) \quad
\mbox{ for all } |t|\leq r_{\nu}.
$$ 
Moreover it holds that
$$
\underline{f}(t):=
\begin{cases}
\nu t & \text{ for large $t \geq 0$} \\
0      & \text{ for $t \leq 0$}.
\end{cases}
$$ 
For technical reasons, it is convenient to choose $\nu$ as follows:\\
If $(f4)$ holds, then we take $\nu >0$ such that
\begin{equation}\label{2.13}
\frac{\nu}{2V_{0}} < \frac{1}{2}-\frac{1}{\mu}.
\end{equation} 
When $(f5)$ is satisfied, we choose $\nu \in (0, \frac{V_{0}}{2})$ such that $\nu$ is a regular value of $t\in (0, \infty) \mapsto \frac{f(t)}{t}$.
Since $\lim_{t\rightarrow 0} \frac{f(t)}{t}=0$ and $\lim_{t \rightarrow \infty}
\frac{f(t)}{t}= a>V_{0} >\nu$, if $\nu$ is a regular value of $\frac{f(t)}{t}$ we deduce that 
\begin{equation}\label{2.9JT}
k_{\nu}= card\{t \in(0, \infty) : f(t)= \nu t\} <\infty.
\end{equation}

\noindent
Now, let $\Lambda \subset \R^{N}$ be a bounded open set such that $\partial \Lambda\in C^{\infty}$, and we assume that $\Lambda$ satisfies (\ref{2.5}).
We take an open set $\Lambda' \subset \Lambda$ with smooth boundary $\partial \Lambda'$ and we define a function $\chi \in C^{\infty}(\R^{N}, \R)$ such that
\begin{align*}
&\inf_{\Lambda \setminus \Lambda'} V > \inf_{\Lambda} V,\\
&\min_{\partial \Lambda'} V > \inf_{\Lambda'} V = \inf_{\Lambda} V,\\
&\chi (x)= 1 \quad \mbox{ for } x\in \Lambda',\\
&\chi (x) \in (0,1) \quad  \mbox{ for } x\in \Lambda \setminus \overline{\Lambda'},\\
&\chi (x)=0 \quad  \mbox{ for } x\in \R^{N} \setminus \Lambda.
\end{align*}

\noindent
Without loss of generality, we suppose that $0\in \Lambda'$ and $V(0)= \inf_{x\in \Lambda} V(x)$. \\

\noindent
Finally, we introduce the following penalty function 
\begin{equation*}
g(x, t) = \chi(x) f(t) + (1- \chi(x))\underline{f}(t) \quad \mbox{ for } (x, t)\in \R^{N}\times \R,
\end{equation*}
and we set
\begin{align*}
\underline{F}(t)&= \int_{0}^{t} \underline{f}(\tau) d\tau, \\
G(x, t)&= \int_{0}^{t} g(x, \tau) d\tau = \chi(x)F(t)+ (1-\chi(x))\underline{F}(t).
\end{align*}

\noindent
As in \cite{JT2}, it is easy to check that the following properties concerning $\underline{f}(t)$ and $g(x, t)$ hold. 
\begin{lem}\label{lemma 2.2}
\begin{compactenum}[(i)]
\item $\underline{f}(t)=0$, $\underline{F}(t)=0$ for all $t\leq 0$.
\item $\underline{f}(t)\leq \nu t$, $\underline{F}(t)\leq F(t)$ for all $t\geq 0$.
\item  $\underline{f}(t)\leq f(t)$  for all $t\geq 0$.
\item If $f(t)$ satisfies either $(f4)$ or $(f5)$, then $\underline{f}(t)
\geq 0$ for all $t \in \R$.
\item If $f(t)$ satisfies $(f5)$, then $\underline{f}(t)$ also satisfies
$(f5)$. Moreover, $\underline{\hat{F}}(t)\geq 0$ for all $t\geq 0$.  
\end{compactenum}
\end{lem}

\begin{cor}\label{corollario 2.1}
\begin{compactenum}[(i)]
\item $g(x, t)\leq f(t)$, $G(x, t)\leq F(t) $ for all $(x, t)\in \R^{N} \times \R$.
\item  $g(x, t)= f(t) \mbox{ if } |t|< r_{\nu}$.
\item For any $\delta>0$ there exists $C_{\delta}>0$ such that 
$$
|g(x, t)|\leq \delta |t| + C_{\delta} |t|^{p} \quad \mbox{ for all } (x, t)\in \R^{N}\times \R.
$$
\item if $f(t)$ satisfies $(f5)$-$(ii)$, then $g(x, t)$ satisfies 
$$
\hat{G}(x,t)\leq D^{k_{\nu}} \hat{G}(x,\bar{t}) \quad \mbox{ for all } 0 \leq t \leq \bar{t},
$$
where $\hat{G}(x, t)= \frac{1}{2} g(x, t)t- G(x, t)$, $D\geq 1$ is given in $(f5)$-$(ii)$ and $k_{\nu}$
is given in \eqref{2.9JT}.
\end{compactenum}
\end{cor}

\noindent
In what follows, we investigate the existence of positive solutions $u_{\e}$ of the following modified problem
\begin{equation} \label{2.15}
\varepsilon^{2s} (-\Delta)^{s} u + V(x) u= g(x, u) \quad \mbox{in } \R^{N} 
\end{equation}
with the property
$$
|u_{\varepsilon}(x)|\leq r_{\nu} \mbox{ for } x\in \R^{N} \setminus \Lambda'.
$$
In view of the definition of $g$, these functions $u_{\varepsilon}$ are also solutions of  (\ref{P}).

\subsection{Mountain pass argument}
Using the change of variable $v(x)= u(\varepsilon x)$, it is possible to prove that \eqref{2.15} is equivalent to the following problem
\begin{equation}\label{2.17}
(-\Delta)^{s} v + V(\varepsilon x) v = g(\varepsilon x, v) \quad \mbox{ in } \R^{N}.
\end{equation}
The energy functional associated with (\ref{2.17}) is given by 
$$
J_{\varepsilon}(v)= \frac{1}{2} \int_{\R^{N}} |(-\Delta)^{\frac{s}{2}} v|^{2} + V(\varepsilon x) v^{2} dx -
\int_{\R^{N}} G(\varepsilon x, v) dx \quad \forall v\in H^{s}_{\varepsilon}
$$
where the fractional space
$$
H^{s}_{\varepsilon}= \Bigl\{v \in H^{s}(\R^{N}) : \int_{\R^{N}} V(\varepsilon x) v^{2} dx < \infty \Bigr\}
$$
is endowed with the norm 
$$
\|v\|_{H^{s}_{\varepsilon}}^{2} = \int_{\R^{N}} |(-\Delta)^{\frac{s}{2}} v|^{2} + V(\varepsilon x) v^{2} dx .
$$

\noindent
Since $V_{0}>0$, we can endow $H^{s}(\R^{N})$ with the following equivalent norm
$$
\|v\|_{H^{s}}^{2}= \int_{\R^{N}} |(-\Delta)^{\frac{s}{2}} v|^{2} + V_{0} v^{2} dx.
$$
Clearly,
\begin{equation}\label{2.18}
\|v\|_{H^{s}} \leq \|v\|_{H^{s}_{\varepsilon}}
\end{equation}
so we get $H^{s}_{\varepsilon} \subset H^{s}(\R^{N})$
and $H^{s}_{\varepsilon}$ is continuously embedded into $L^{r}(\R^{N})$ 
for any $2 \leq r \leq 2^{*}_{s}$, and there exists $C'_{r}> 0$ such that
\begin{equation}\label{2.19}
\|v\|_{L^{r}(\R^{N})} \leq C'_{r} \|v\|_{H^{s}}.
\end{equation}

\noindent
We start by proving that $J_{\e}$ possesses a mountain pass geometry that is uniform with respect to $\e$. 
\begin{lem} \label{proposizione 2.1}
$J_{\varepsilon} \in C^{1}(H^{s}_{\varepsilon}, \R)$ and satisfies the following properties: 
\begin{compactenum}[(i)]
\item $J_{\varepsilon}(0)=0$;
\item there exist $\rho_{0} >0$ and $\delta_{0}>0 $ independent of $\varepsilon \in (0,1]$
such that
\begin{align*}
&J_{\varepsilon}(v) \geq \delta_{0} \mbox{ for all } \|v\|_{H^{s}} = \rho_{0}\\
&J_{\varepsilon}(v) > 0 \mbox{ for all } 0 < \|v\|_{H^{s}} \leq \rho_{0};
\end{align*}
\item there exist $v_{0} \in C^{\infty}_{0}(\R^{N})$ and $\varepsilon_{0}> 0$
such that $J_{\varepsilon}(v_{0})<0$ for all $\varepsilon \in (0, \varepsilon_{0}]$.
\end{compactenum}
\end{lem}

\begin{proof}
Obviously, $J_{\varepsilon} \in C^{1}(H^{s}_{\varepsilon}, \R)$ and  $J_{\varepsilon}(0)=0$.
Using $\underline{F}\leq F$ and taking $\delta = \frac{V_{0}}{2}$
in \eqref{2.7}, we get
\begin{align*}
J_{\varepsilon}(v) &= \frac{1}{2} \|v\|_{H^{s}_{\varepsilon}}^{2} - \int_{\R^{N}} 
\chi(\varepsilon x)F(v) + (1- \chi(\varepsilon x))\underline{F}(v) \, dx \\
& \geq \frac{1}{2} \|v\|_{H^{s}_{\varepsilon}}^{2} -\int_{\R^{N}} F(v) \, dx \\
& \geq \frac{1}{2} \|v\|_{H^{s}}^{2} - \frac{V_{0}}{4}\|v\|_{L^{2}(\R^{N})}^{2} - 
C_{\frac{V_{0}}{2}} \|v\|_{L^{p+1}(\R^{N})}^{p+1} \\
& \geq \frac{\|v\|_{H^{s}}^{2}}{4} - \tilde{C}_{p+1} C_{\frac{V_{0}}{2}} \|v\|_{H^{s}}^{p+1},
\end{align*}
where we used \eqref{2.18} and (\ref{2.19}) with $r=p+1$. Thus $(ii)$ is satisfied. 

In order to verify that $(iii)$ holds, we take $v_{0}\in C^{\infty}_{0}(\R^{N})$ such that
$$
\frac{1}{2} \int_{\R^{N}} |(-\Delta)^{\frac{s}{2}} v_{0}|^{2} + V(0) v_{0}^{2} \,dx - \int_{\R^{N}} F(v_{0}) \,dx < 0. 
$$
This choice is lawful due to the fact that $V(0)<\lim_{z\rightarrow \infty}\frac{f(z)}{z}$, so the existence of a such $v_{0}$ follows from Theorem $1$ in \cite{A2} (see Lemma \ref{prop 5.2}), where is proved that  
$$
v \mapsto \frac{1}{2} \int_{\R^{N}} |(-\Delta)^{\frac{s}{2}} v|^{2} + V(0) v^{2} \,dx - 
\int_{\R^{N}} F(v) \,dx
$$
has a mountain pass geometry. Since $0\in \Lambda'$, we can observe that 
$$
J_{\varepsilon}(v_{0}) \rightarrow \frac{1}{2}\int_{\R^{N}} |(-\Delta)^{\frac{s}{2}} v_{0}|^{2} + V(0)v_{0}^{2} \,dx -
\int_{\R^{N}} F(v_{0}) \,dx <0 \mbox{  as  } \varepsilon \rightarrow 0,
$$
that is $(iii)$ is verified  for $\e$ sufficiently small.
\end{proof}

\noindent
Since $J_{\e}$ has a mountain pass geometry, for any $\varepsilon \in (0, \varepsilon _{0}]$ we can define the mountain pass value
\begin{equation}\label{2.20}
c_{\varepsilon} = \inf_{\gamma \in \Gamma_{\varepsilon}} \max_{t\in [0, 1]} 
J_{\varepsilon}(\gamma(t))
\end{equation}
where
\begin{equation}\label{2.21}
\Gamma_{\varepsilon}= \left\{\gamma \in C([0,1],H^{s}_{\varepsilon}) : \gamma(0)=0 \mbox{ and } 
J_{\varepsilon} (\gamma (1))< 0\right\}.
\end{equation}
Using Lemma \ref{proposizione 2.1}, we are able to give the following estimate for $c_{\varepsilon}$.

\begin{cor} \label{corollario 2.2}
There exist $m_{1},m_{2} >0$ such that for any  $\varepsilon \in (0, \varepsilon_{0}]$
$$
m_{1} \leq c_{\varepsilon} \leq m_{2}.
$$
\end{cor}

\begin{proof}
For any $\gamma \in \Gamma_{\varepsilon}$ we have
\begin{align*}
\gamma ([0,1]) \cap \{v\in H^{s}_{\varepsilon} : \|v\|_{H^{s}} = \rho\} \neq \emptyset.
\end{align*}
Hence, by using Lemma \ref{proposizione 2.1}, we can deduce that
\begin{align*}
\max_{t\in [0,1]} J_{\varepsilon}(\gamma(t)) \geq \inf_{\|v\|_{H^{s}}= \rho_{0}}
J_{\varepsilon}(v)\geq \delta_{0}.
\end{align*}
Set $\gamma_{0}(t)= tv_{0}$, where $v_{0} \in C_{0}^{\infty}(\R^{N})$ is obtained in Lemma \ref{proposizione 2.1}.
Then we can see that
\begin{align*}
c_{\varepsilon} &= \inf_{\gamma \in \Gamma_{\varepsilon}} \Bigl(\max_{t\in [0,1]} 
J_{\varepsilon}(\gamma(t)) \Bigr) 
\leq \max_{t\in [0,1]} J_{\varepsilon} (\gamma_{0}(t)) 
\leq \sup_{\varepsilon \in (0, \varepsilon_{0}]} \Bigl( \max_{t\in [0,1]} J_{\varepsilon}
(\gamma_{0}(t))\Bigr).
\end{align*}
Therefore, we put $m_{1}= \delta_{0}$ and $m_{2}= \sup_{\varepsilon \in (0, \varepsilon_{0}]} 
\Bigl( \max_{t\in [0,1]} J_{\varepsilon} (\gamma_{0}(t))\Bigr) $.
\end{proof}

\noindent
Next, we investigate the boundedness of Cerami sequences corresponding to the mountain pass value $c_{\e}$. 
We recall that the existence of a Cerami sequence for $J_{\e}$ follows by the following variant version of the mountain pass theorem.
\begin{thm}\cite{ekeland}\label{teorema 3.1}
Let $X$ be a real Banach space with its dual $X^{*}$, and suppose that $I\in C^{1}(X, \R)$ satisfies
$$
\max\{I(0), I(e)\}\leq \mu<\alpha\leq \inf_{\|x\|=\rho} I(x),
$$
for some $\mu<\alpha$, $\rho>0$ and $e\in X$ with $\|e\|>\rho$. Let $c\geq \alpha$ be characterized by
$$
c=\inf_{\gamma\in \Gamma}\max_{t\in [0, 1]} I(\gamma(t)),
$$
where
$$
\Gamma=\{\gamma\in C([0, 1], X): \gamma(0)=0, \gamma(1)=e\}
$$
is the set of continuous paths joining $0$ and $e$.
Then there exists a Cerami sequence $(x_{j})\subset X$ at the level $c$ that is
$$
I(x_{j})\rightarrow c \mbox{ and } (1+\|x_{j}\|)\|I'(x_{j})\|_{*}\rightarrow 0
$$ 
as $j\rightarrow \infty$.
\end{thm}
\noindent

Using Lemma \ref{proposizione 2.1} and Theorem \ref{teorema 3.1}, we can deduce that
for all $\varepsilon \in (0,\varepsilon_{0}]$ there exists a Cerami sequence
$(v_{j}) \subset H^{s}_{\varepsilon}$ such that
\begin{align*}
&J_{\varepsilon} (v_{j}) \rightarrow b_{\varepsilon} \\
&(1+ \|v_{j}\|_{H^{s}_{\varepsilon}})\|J'_{\varepsilon}(v_{j})\|_{H^{-s}_{\varepsilon}} \rightarrow 0 
\mbox{ as } j\rightarrow \infty.
\end{align*}

The next result states that every critical point $v_{\e}$ of $J_{\e}$ at the level $c_{\e}$ is uniformly bounded with respect to $\e$, that is 
\begin{equation}\label{3.1}
\limsup_{\varepsilon\rightarrow 0} \|v_{\varepsilon}\|_{H^{s}_{\varepsilon}} < \infty.
\end{equation}

\begin{lem}\label{proposizione 3.2}
Assume that $f$ satisfies $(f1)$-$(f3)$ and either $(f4)$ or $(f5)$. Suppose that there exists a sequence 
$(v_{\varepsilon})_{\varepsilon\in (0, \varepsilon_{1}]}$, with $\e_{1}\in (0, \e_{0}]$, such that
\begin{align}
&v_{\varepsilon} \in H^{s}_{\varepsilon}, \nonumber\\
&J_{\varepsilon}(v_{\varepsilon})\in [m_{1}, m_{2}] \quad \forall \varepsilon 
\in (0, \varepsilon_{1}], \label{3.4} \\
&(1+ \|v_{\varepsilon}\|_{H^{s}_{\varepsilon}})\|J'_{\varepsilon}(v_{\e})\|_{H^{-s}_{\varepsilon}}
\rightarrow 0 \mbox{ as } \varepsilon\rightarrow 0 \label{3.5}
\end{align}
with $0<m_{1}<m_{2}$. Then \eqref{3.1} holds.
\end{lem}

\begin{proof}
Firstly, we assume that $(f4)$ holds. Let $(v_{\varepsilon})$ be a sequence satisfying \eqref{3.4} and \eqref{3.5}. 
Then we can see that \eqref{3.4} yields
\begin{equation}\label{3.6}
J_{\varepsilon}(v_{\varepsilon})= \frac{1}{2} \|v_{\varepsilon}\|^{2}_{H^{s}_{\varepsilon}}-
\int_{\R^{N}} (1-\chi(\varepsilon x)) \underline{F}(v_{\varepsilon}) + \chi(\varepsilon x) 
F(v_{\varepsilon}) \,dx \leq m_{2}.
\end{equation}
Moreover, by \eqref{3.5}, for any $\e$ sufficiently small we have
$$
|\langle J'_{\varepsilon}(v_{\varepsilon}), v_{\varepsilon}\rangle| \leq \|J'_{\varepsilon} (v_{\varepsilon})\|_{H^{-s}_{\varepsilon}}
\|v_{\varepsilon}\|_{H^{s}_{\varepsilon}} \leq \|J'_{\varepsilon}(v_{\varepsilon})\|_{H^{-s}_{\varepsilon}}
(1+\|v_{\varepsilon}\|_{H^{s}_{\varepsilon}})\leq 1, 
$$
that is
\begin{equation}\label{3.7}
\left|\|v_{\varepsilon}\|^{2}_{H^{s}_{\varepsilon}}-
\int_{\R^{N}} (1-\chi(\varepsilon x))\underline{f}(v_{\varepsilon})v_{\varepsilon} + \chi(\varepsilon x) 
f(v_{\varepsilon})v_{\varepsilon}\, dx \right| \leq 1.
\end{equation}
Taking into account \eqref{3.6}, \eqref{3.7} and $(f4)$  we get
$$
\Bigl(\frac{1}{2}- \frac{1}{\mu} \Bigr) \|v_{\varepsilon}\|^{2}_{H_{\varepsilon}} \leq \int_{\R^{N}} (1- \chi(\varepsilon x)) \Bigl(\underline{F}(v_{\varepsilon})-\frac{1}{\mu}f(v_{\varepsilon})v_{\varepsilon}\Bigr) \,dx+m_{2}+\frac{1}{\mu}.
$$
Using $(i)$ and $(iv)$ of Lemma \ref{lemma 2.2}, we know that  $t \underline{f}(t)\geq 0$ for all $t\in \R$, so we obtain
\begin{equation}\label{verde}
\Bigl(\frac{1}{2}- \frac{1}{\mu} \Bigr) \|v_{\varepsilon}\|^{2}_{H_{\varepsilon}} \leq
\int_{\R^{N}} (1- \chi(\varepsilon x))\underline{F}(v_{\varepsilon}) \,dx + m_{2} + \frac{1}{\mu}.
\end{equation}
On the other hand, by $(ii)$ of Lemma \ref{lemma 2.2} it follows that
$$
\underline{F}(t) \leq \frac{\nu t^{2}}{2} \quad \mbox{ for all } t \in \R. 
$$
Then
\begin{align*}
\int_{\R^{N}} (1- \chi(\varepsilon x))\underline{F}(v_{\varepsilon}) \,dx &\leq \frac{1}{2} \nu 
\|v_{\varepsilon}\|_{L^{2}(\R^{N})}^{2} \leq \frac{\nu}{2V_{0}}\|v_{\e}\|^{2}_{H^{s}_{\e}},
\end{align*}
which together with \eqref{verde} yields
\begin{align*}
\Bigl(\frac{1}{2}- \frac{1}{\mu}\Bigr) \|v_{\varepsilon}\|^{2}_{H^{s}_{\varepsilon}} \leq \frac{\nu}{2V_{0}}
\|v_{\varepsilon}\|^{2}_{H^{s}_{\varepsilon}}+ m_{2} + \frac{1}{\mu}. 
\end{align*}
In view of \eqref{2.13} we get
$$
\|v_{\varepsilon}\|^{2}_{H^{s}_{\varepsilon}}\leq \frac{m_{2}+\frac{1}{\mu}}{\Bigl[\Bigl(
\frac{1}{2}- \frac{1}{\mu}\Bigr)- \frac{\nu}{2V_{0}}\Bigr]},
$$
which implies that $\|v_{\varepsilon}\|_{H^{s}_{\varepsilon}}$ is bounded if $\varepsilon$ is small enough.

\noindent
Now, let us suppose that $(f5)$ holds. Arguing by contradiction, we assume that
$$
\limsup_{\varepsilon \rightarrow 0} \|v_{\varepsilon}\|_{H^{s}_{\varepsilon}}= \infty.
$$
Let $\varepsilon_{j}\rightarrow 0$ be a subsequence such that $\|v_{\varepsilon_{j}}\|_{H^{s}_{\varepsilon_{j}}} 
\rightarrow \infty$. For simplicity, we denote $\varepsilon_{j}$ still by $\varepsilon$. \\
Set
$\displaystyle{w_{\varepsilon}= \frac{v_{\varepsilon}}{\|v_{\varepsilon}\|_{H^{s}_{\varepsilon}}}}$. 
Clearly $\|w_{\varepsilon}\|_{H^{s}} = \frac{\|v_{\varepsilon}\|_{H^{s}}}{\|v_{\varepsilon}\|_{H^{s}_{\varepsilon}}}
\leq \frac{\|v_{\varepsilon}\|_{H^{s}_{\varepsilon}}}{\|v_{\varepsilon}\|_{H^{s}_{\varepsilon}}} =1$.
Moreover, we can see that there exists $C_{1}>0$ independent of $\varepsilon$ such that 
\begin{equation}\label{3.8}
\|\chi_{\varepsilon} w_{\varepsilon}\|_{H^{s}}\leq C_{1},
\end{equation}
where $\chi_{\e}(x)=\chi(\e x)$.\\
Indeed, using $0\leq \chi\leq 1$, $(|a|+|b|)^{2}\leq 2(|a|^{2}+|b|^{2})$, $\e\in (0, \e_{1}]$ and $s\in (0, 1)$, we get
\begin{align*}
&\iint_{\R^{2N}} \frac{|\chi(\e x) w_{\e}(x)-\chi(\e y) w_{\e}(y)|^{2}}{|x-y|^{N+2s}} \, dx dy+\int_{\R^{N}} V_{0} (\chi_{\varepsilon} w_{\varepsilon})^{2} \, dx \\
&\leq 2\iint_{\R^{2N}} \frac{|\chi(\e x)-\chi(\e y)|^{2}}{|x-y|^{N+2s}} w^{2}_{\e}(x)\, dx dy+2\iint_{\R^{N}} \frac{|w_{\e}(x)-w_{\e}(y)|^{2}}{|x-y|^{N+2s}} \, dx dy\\
&\quad +\int_{\R^{N}} V_{0} w^{2}_{\varepsilon} \, dx \\
&\leq 2\e^{2} \|\nabla \chi\|^{2}_{L^{\infty}(\R^{N})} \int_{\R^{N}} w^{2}_{\e}(x) \,dx \int_{|z|\leq 1} \frac{1}{|z|^{N+2s-2}} \, dz\\
&\quad+8 \int_{\R^{N}} w^{2}_{\e}(x) \,dx \int_{|z|>1} \frac{1}{|z|^{N+2s}} \, dz +2 \iint_{\R^{N}} \frac{|w_{\e}(x)-w_{\e}(y)|^{2}}{|x-y|^{N+2s}} \, dx dy\\
&\quad +\int_{\R^{N}} V_{0} w^{2}_{\varepsilon} \, dx \\
&\leq \left((1-s)^{-1}\e^{2}_{1} \|\nabla \chi\|^{2}_{L^{\infty}(\R^{N})}\alpha_{N-1}+4s^{-1}\alpha_{N-1}+V_{0} \right) \|w_{\e}\|^{2}_{L^{2}(\R^{N})}+2[w_{\e}]^{2} \\
&\leq C_{1}\|w_{\varepsilon}\|^{2}_{H^{s}}\leq C_{1}, 
\end{align*}
where $\alpha_{N-1}$ denotes the Lebesgue measure of the unit sphere in $\R^{N}$. \\
Now, (\ref{3.5}) implies that $\langle J'_{\varepsilon}(v_{\varepsilon}), \varphi \rangle =o(1)$  for any 
$\varphi\in H^{s}_{\varepsilon}$, that is
\begin{align*}
\int_{\R^{N}} &(-\Delta)^{\frac{s}{2}} v_{\varepsilon} (-\Delta)^{\frac{s}{2}}\varphi + V(\varepsilon x) v_{\varepsilon}\varphi  \, dx \\
&=\int_{\R^{N}} [\chi_{\varepsilon} f(v_{\varepsilon}) \varphi+(1- \chi_{\varepsilon})\underline{f}(v_{\varepsilon})]\varphi dx 
+o(1),
\end{align*}
or equivalently
\begin{align}\label{3.9}
\int_{\R^{N}} &(-\Delta)^{\frac{s}{2}} w_{\varepsilon} (-\Delta)^{\frac{s}{2}}\varphi + V(\varepsilon x) w_{\varepsilon} \varphi\, dx \nonumber \\ 
&=\int_{\R^{N}}\left[\chi_{\varepsilon}\frac{f(v_{\varepsilon})}{v_{\varepsilon}} w_{\varepsilon}+ 
(1- \chi_{\varepsilon})\frac{\underline{f}(v_{\varepsilon})}{v_{\varepsilon}}w_{\varepsilon}\right] \varphi  dx+o(1).
\end{align}
Taking $\varphi= w_{\varepsilon}^{-}=\min\{w_{\e}, 0\}$ in (\ref{3.9}) and recalling that 
$$
(x-y)(x^{-}-y^{-})\geq |x^{-}-y^{-}|^{2} \mbox{ for any } x, y\in \R,
$$ 
and that $f(t)=\underline{f}(t)=0$ for all $t\leq 0$, 
we have 
\begin{align*}
\int_{\R^{N}} &|(-\Delta)^{\frac{s}{2}} w^{-}_{\varepsilon}|^{s}+ V(\varepsilon x) (w^{-}_{\varepsilon})^{2}  \, dx \nonumber \\
&\leq  \int_{\R^{N}} \left[\chi_{\varepsilon}\frac{f(v_{\varepsilon})}{v_{\varepsilon}} w_{\varepsilon}+ 
(1- \chi_{\varepsilon})\frac{\underline{f}(v_{\varepsilon})}{v_{\varepsilon}}w_{\varepsilon}\right] w_{\varepsilon}^{-} dx
+ o(1)=o(1),
\end{align*}
so we get
\begin{equation}\label{3.10}
\|w_{\varepsilon}^{-}\|^{2}_{H^{s}_{\varepsilon}} \rightarrow 0 \mbox{ as } 
\varepsilon\rightarrow 0.
\end{equation}

Now, we can observe that one of the following two cases must occur.
\begin{compactenum}
\item [Case $1$:]  $\limsup_{\varepsilon\rightarrow 0} \Bigl(\sup_{z\in \R^{N}} \int_{B_{1}(z)} 
|\chi_{\varepsilon}(x) w_{\varepsilon}|^{2} dx \Bigr)>0$; 
\item [Case $2$:]  $\limsup_{\varepsilon\rightarrow 0} \Bigl(\sup_{z\in \R^{N}} \int_{B_{1}(z)}
|\chi_{\varepsilon}(x)w_{\varepsilon}|^{2} dx\Bigr)=0$.
\end{compactenum}

\begin{description}
\item [Step$1$] Case $1$ can not occur under assumption $(f5)$ with $a=\infty$.  
\end{description}

\noindent
We argue by contradiction, and we suppose that Case $1$ occurs. Then, up to a subsequence, there exist $(x_{\varepsilon})\subset \R^{N}$, $d>0$ and $x_{0}\in \overline{\Lambda}$ such that
\begin{align}
&\int_{B_{1}(x_{\varepsilon})} |\chi_{\varepsilon}w_{\varepsilon}|^{2} dx 
\rightarrow d >0, \label{3.11} \\ 
&\varepsilon x_{\varepsilon} \rightarrow x_{0} \in \overline{\Lambda}. \label{3.12}
\end{align}
Indeed, the existence of $(y_{\e})$ satisfying \eqref{3.11} is clear.
Moreover, \eqref{3.11} implies that 
$B_{1}(x_{\varepsilon})\cap \supp(\chi_{\varepsilon})\neq \emptyset$, so there exists $z_{\varepsilon}\in \supp(\chi_{\varepsilon})$ such that
$\chi(\varepsilon z_{\varepsilon})\neq 0$ and $|z_{\varepsilon}-x_{\varepsilon}|<1$. 
Hence 
$|\varepsilon x_{\varepsilon} - \varepsilon z_{\varepsilon}|< \varepsilon$ yields 
$\varepsilon x_{\varepsilon} \in N_{\varepsilon}(\Lambda)=\{z\in \R^{N} : {\rm dist}(z, \Lambda) < \varepsilon\}$,
and we may assume that \eqref{3.12} holds.
Since $\|w_{\varepsilon}\|_{H^{s}} \leq 1$, we may suppose that
\begin{equation}\label{3.13}
w_{\varepsilon}(\cdot+ x_{\varepsilon})\rightharpoonup w_{0} \mbox{ in } H^{s}(\R^{N}).
\end{equation}
Taking into account (\ref{3.12}) and (\ref{3.13}) we have 
\begin{align*}
(\chi_{\varepsilon} w_{\varepsilon})(\cdot+x_{\varepsilon}) \rightharpoonup \chi(x_{0})w_{0} \mbox{ in } H^{s}(\R^{N}).
\end{align*}
To prove this, fix $\varphi\in H^{s}(\R^{N})$, and we note that
\begin{align*}
\iint_{\R^{2N}} &\frac{(\chi_{\e}w_{\e})(x+x_{\e})-(\chi_{\e}w_{\e})(y+x_{\e})}{|x-y|^{N+2s}}(\varphi(x)-\varphi(y))\, dx dy\\
&=\iint_{\R^{2N}} \frac{(h_{\e}w_{\e})(x+x_{\e})-(h_{\e}w_{\e})(y+x_{\e})}{|x-y|^{N+2s}}(\varphi(x)-\varphi(y))\, dx dy\\
&\quad +\iint_{\R^{2N}} \frac{(w_{\e}(x+x_{\e})-w_{\e}(y+x_{\e}))}{|x-y|^{N+2s}}(\varphi(x)-\varphi(y))\chi(x_{0})\, dx dy \\
&=A_{\e}+B_{\e},
\end{align*}
where $h_{\e}(x)=\chi_{\e}(x)-\chi(x_{0})$.
In view of (\ref{3.13}) we know that
\begin{align*}
B_{\e}\rightarrow \iint_{\R^{2N}} \frac{(w_{0}(x)-w_{0}(y))}{|x-y|^{N+2s}}(\varphi(x)-\varphi(y))\chi(x_{0})\, dx dy.
\end{align*}
Now, we observe that 
\begin{align*}
A_{\e}&=\iint_{\R^{2N}} \frac{(h_{\e}(x+x_{\e})-h_{\e}(y+x_{\e}))}{|x-y|^{N+2s}}(\varphi(x)-\varphi(y))w_{\e}(x+x_{\e})\, dx dy\\
&+ \iint_{\R^{2N}} \frac{(w_{\e}(x+x_{\e})-w_{\e}(y+x_{\e}))}{|x-y|^{N+2s}}(\varphi(x)-\varphi(y))h_{\e}(y+x_{\e})\, dx dy\\
&=A_{\e}^{1}+A_{\e}^{2}.
\end{align*}
Using H\"older's inequality, \eqref{3.12}, \eqref{3.13} and the dominated convergence theorem, we can see that
\begin{align*}
|A_{\e}^{2}|\leq C\left(\iint_{\R^{2N}} \frac{|\varphi(x)-\varphi(y)|^{2}}{|x-y|^{N+2s}}|h_{\e}(y+x_{\e})|^{2}\, dx dy  \right)^{\frac{1}{2}}\rightarrow 0.
\end{align*}
On the other hand
\begin{align*}
|A_{\e}^{1}|\leq [\varphi]\left(\iint_{\R^{2N}} \frac{|h_{\e}(x+x_{\e})-h_{\e}(y+x_{\e})|^{2}}{|x-y|^{N+2s}}|w_{\e}(x+x_{\e})|^{2}\, dx dy  \right)^{\frac{1}{2}}\rightarrow 0
\end{align*}
because
\begin{align*}
&\iint_{\R^{2N}} \frac{|h_{\e}(x+x_{\e})-h_{\e}(y+x_{\e})|^{2}}{|x-y|^{N+2s}}|w_{\e}(x+x_{\e})|^{2}\, dx dy \\
&\leq \int_{\R^{N}} |w_{\e}(x+x_{\e})|^{2}\, dx\left[ \int_{|y-x|>\frac{1}{\e}} \frac{4dy}{|x-y|^{N+2s}}+\int_{|y-x|<\frac{1}{\e}} \frac{\e^{2}\|\nabla \chi\|^{2}_{L^{\infty}(\R^{N})} dy}{|x-y|^{N+2s-2}}\right] \\
&\leq C\e^{2s} \int_{\R^{N}} |w_{\e}(x+x_{\e})|^{2}\, dx\leq C\e^{2s}\rightarrow 0.
\end{align*}

Now, let us show that $\chi(x_{0})\neq 0$ and $w_{0}\geq 0$ ($\not\equiv 0$).
If by contradiction $\chi(x_{0})=0$, by the dominated convergence theorem, (\ref{3.11}), \eqref{3.13} and Theorem \ref{Sembedding} we obtain 
\begin{align*}
0<d&= \lim_{\varepsilon\rightarrow 0} \int_{B_{1}(x_{\varepsilon})} |\chi_{\varepsilon}
w_{\varepsilon}|^{2} dx\\
&= \lim_{\varepsilon\rightarrow 0} \int_{B_{1}} |\chi_{\varepsilon} w_{\varepsilon}|^{2} 
(x+x_{\varepsilon}) \,dx \\
&= \int_{B_{1}} |\chi(x_{0})w_{0}(x)|^{2} dx =0, 
\end{align*}
which is impossible.
For the same reason $w_{0}\not\equiv 0$.
Using (\ref{3.10}) and (\ref{3.13}) we can see that  $w_{0} \geq 0$ in $\R^{N}$. 
Thus, there exists a set $K \subset \R^{N}$ such that
\begin{align}
&|K|>0 \label{3.14} \\
&w_{\varepsilon}(x+x_{\varepsilon}) \rightarrow w_{0}(x)>0 \quad \mbox{ for } x\in K. \label{3.15}
\end{align} 
Taking $\varphi= w_{\varepsilon}$ in (\ref{3.9}), we get
\begin{align*}
1=\|w_{\varepsilon}\|_{H^{s}_{\varepsilon}}^{2} = \int_{\R^{N}} \chi_{\varepsilon} \frac{f(v_{\varepsilon})}{v_{\varepsilon}} w_{\varepsilon}^{2} +
(1- \chi_{\varepsilon}) \frac{\underline{f}(v_{\varepsilon})}{v_{\varepsilon}} w_{\varepsilon}^{2} 
dx + o(1),
\end{align*}
and using $(iv)$ of Lemma \ref{lemma 2.2}, we deduce that
\begin{equation}\label{3.16}
\limsup_{\varepsilon \rightarrow 0} \int_{\R^{N}} \chi_{\varepsilon} 
\frac{f(v_{\varepsilon})}{v_{\varepsilon}} w_{\varepsilon}^{2} \, dx \leq 1, 
\end{equation}
that is
$$
\limsup_{\varepsilon \rightarrow 0} \int_{\R^{N}} \chi(\varepsilon x + 
\varepsilon x_{\varepsilon})\frac{f(v_{\varepsilon}(x+x_{\varepsilon}))}{v_{\varepsilon}
(x+x_{\varepsilon})}w_{\varepsilon}(x+x_{\varepsilon})^{2} \, dx \leq 1.
$$
In view of (\ref{3.14}), (\ref{3.15}) and the definition of $w_{\varepsilon}$, we obtain 
$$
v_{\varepsilon}(x+x_{\varepsilon})= w_{\varepsilon}(x+x_{\varepsilon})
\|v_{\varepsilon}\|_{H^{s}_{\varepsilon}} \rightarrow w_{0}(x)\cdot (\infty) = 
\infty \quad \forall x\in K.
$$
This, together with $\lim_{\xi\rightarrow \infty} \frac{f(\xi)}{\xi}=a= \infty$ and Fatou's Lemma yields 
\begin{align*}
\liminf_{\varepsilon \rightarrow 0} &\int_{\R^{N}} \chi_{\varepsilon}(x+x_{\varepsilon})\frac{f(v_{\varepsilon}(x+x_{\varepsilon}))}
{w_{\varepsilon}(x+x_{\varepsilon})} w_{\varepsilon}(x+x_{\varepsilon})^{2} dx \\
&\geq \liminf_{\varepsilon \rightarrow 0}\int_{K} \chi_{\varepsilon}(x+x_{\varepsilon})\frac{f(v_{\varepsilon}(x+x_{\varepsilon}))}
{v_{\varepsilon}(x+x_{\varepsilon})}w_{\varepsilon}(x+x_{\varepsilon})^{2} \, dx= \infty
\end{align*}
which contradicts (\ref{3.16}).

\begin{description}
\item [Step $2$] Case $1$ can not take place under assumption $(f5)$ with $a<\infty$. 
\end{description}

\noindent
As in Step $1$, we extract a  subsequence and we assume that (\ref{3.11}),(\ref{3.12}) and 
(\ref{3.13}) hold with $\chi(x_{0})\neq 0$ and $w_{0}\geq 0$ $(\not\equiv 0)$. We aim to prove that $w_{0}$ is a weak solution to
\begin{equation}\label{3.17}
(-\Delta)^{s} w_{0} + V(x_{0})w_{0}= (\chi(x_{0})a + (1-\chi(x_{0}))\nu)w_{0} \mbox{ in } \R^{N}.
\end{equation}
This provides a contradiction since $(-\Delta)^{s}$ has no eigenvalues in $H^{s}(\R^{N})$ (this fact can be seen by using the Pohozaev Identity for the fractional Laplacian \cite{A1, CW, Secchi2}). \\
Fix $\varphi \in C^{\infty}_{0}(\R^{N})$. Taking into account (\ref{3.12}), (\ref{3.13}) and the continuity of $V$, we can see that
\begin{align}\label{3.18}
&\int_{\R^{N}}  (-\Delta)^{\frac{s}{2}}w_{\varepsilon}(x+x_{\varepsilon})(-\Delta)^{\frac{s}{2}}\varphi(x) + 
V(\varepsilon x+ \varepsilon x_{\varepsilon})w_{\varepsilon} \varphi \, dx \nonumber \\
&\rightarrow \int_{\R^{N}} (-\Delta)^{\frac{s}{2}} w_{0} (-\Delta)^{\frac{s}{2}} \varphi +V(x_{0})w_{0}\varphi \, dx.
\end{align}
Now, we show that
\begin{align}\label{3.19}
\int_{\R^{N}} \frac{g(\varepsilon x+\varepsilon x_{\varepsilon}, v_{\varepsilon}(x+x_{\varepsilon}))}
{v_{\varepsilon}(x+x_{\varepsilon})} w_{\varepsilon}\varphi \,dx
\rightarrow (\chi(x_{0})a + (1-\chi(x_{0}))\nu)\int_{\R^{N}} w_{0} \varphi\, dx. 
\end{align}

Take $R>1$ such that $\supp \varphi \subset B_{R}$.  Then, using the fact that $H^{s}(\R^{N})$ is compactly embedded into $L^{2}_{loc}(\R^{N})$, we get
$\|w_{\varepsilon}- w_{0}\|^{2}_{L^{2}(B_{R})} \rightarrow 0$ .
Hence, there exists $h\in L^{2}(B_{R})$ such that 
\begin{align*}
|w_{\varepsilon}|\leq h \mbox{ a.e. in } B_{R}.
\end{align*}
Since $a<\infty$, there exists $C>0$ such that $|\frac{g(x, t)}{t}|\leq C$ for any
$t>0$. We recall that 
$$
\frac{g(x, t)}{t} \rightarrow \chi(x)a+(1-\chi(x))\nu <\infty \mbox{ as } t
\rightarrow \infty.
$$

Then
\begin{align}\label{3.20}
\Bigl|\frac{g(\varepsilon x+ \varepsilon x_{\varepsilon}, v_{\varepsilon}(x+x_{\varepsilon}))}
{v_{\varepsilon}(x+x_{\varepsilon})} w_{\varepsilon}\varphi\Bigr| &\leq C \|\varphi\|_{L^{\infty}(\R^{N})} 
|w_{\varepsilon}| \nonumber \\
&\leq C \|\varphi\|_{L^{\infty}(\R^{N})} h \in L^{1}(B_{R}),
\end{align}
and
\begin{align}\label{3.21}
\frac{g(\varepsilon x+ \varepsilon x_{\varepsilon}, v_{\varepsilon}(x+x_{\varepsilon}))}
{v_{\varepsilon}(x+x_{\varepsilon})} w_{\varepsilon}(x) 
\rightarrow [\chi(x_{0}) a + (1-\chi(x_{0}))\nu]w_{0}(x) \mbox{ a.e. in } B_{R}.
\end{align}
In fact, if $w_{0}(x)=0$, being $|\frac{g(x,t)}{t}|\leq C$  
for all $t>0$ and $w_{\varepsilon}\rightarrow w_{0}=0$ a.e. in $B_{R}$,  we get
$$
\Bigl| \frac{g(\varepsilon x+ \varepsilon x_{\varepsilon}, v_{\varepsilon}(x+x_{\varepsilon}))}
{v_{\varepsilon}(x+x_{\varepsilon})} w_{\varepsilon} \Bigr|\leq C |w_{\varepsilon}| \rightarrow 0 \, \mbox{ a.e. in } B_{R}.
$$ 
If $w_{0}(x)\neq 0$ then 
$v_{\varepsilon}(x+x_{\varepsilon})= w_{\varepsilon}(x+x_{\varepsilon}) \|v_{\varepsilon}\|
_{H^{s}_{\varepsilon}}\rightarrow \infty$ and being $w_{\varepsilon}\rightarrow w_{0}$ a.e. in $B_{R}$
we have 
$$
\frac{g(\varepsilon x+ \varepsilon x_{\varepsilon}, v_{\varepsilon}(x+x_{\varepsilon}))}
{v_{\varepsilon}(x+x_{\varepsilon})}w_{\e}\rightarrow [\chi(x_{0})a + (1- \chi(x_{0}))\nu]w_{0} \, \mbox{ a.e. in } B_{R}.
$$
Then (\ref{3.21}) holds. Taking into account (\ref{3.20}) and
(\ref{3.21}), we can infer that (\ref{3.19}) is true in view of the dominated convergence theorem.
Putting together $\langle J'_{\e}(v_{\e}), \varphi\rangle=o(1)$, \eqref{3.18} and \eqref{3.19} we obtain \eqref{3.17}.

\begin{description}
\item [Step $3$] Case $2$ can not take place. 
\end{description}
Assume by contradiction that Case $2$ occurs. Since  (\ref{3.8}) holds and  
$$
\lim_{\varepsilon\rightarrow 0} \sup_{z\in  \R^{N}} \int_{B_{1}(z)}|\chi_{\varepsilon}
w_{\varepsilon}|^{2} \, dx =0,
$$
by Lemma \ref{lionslemma} we deduce that $\|\chi_{\varepsilon}w_{\varepsilon}\|_{L^{p+1}(\R^{N})} \rightarrow 0$.\\
Now, for any $L>1$ we can see that
$$
J_{\varepsilon}(Lw_{\e})=
\frac{1}{2} L^{2}-\int_{\R^{N}} \chi_{\varepsilon} F(Lw_{\varepsilon})+ (1- \chi_{\varepsilon})
\underline{F}(Lw_{\varepsilon}) \, dx.
$$
By $(ii)$ of Lemma \ref{lemma 2.2} and $\nu\in (0, \frac{V_{0}}{2})$ we have 
\begin{align*} 
\int_{\R^{N}}(1- \chi_{\varepsilon})\underline{F}(Lw_{\varepsilon}) \, dx &\leq\int_{\R^{N}} \frac{1}{2} \nu L^{2} |w_{\varepsilon}|^{2} \, dx  \\
&\leq \int_{\R^{N}} \frac{V_{0}}{4} L^{2} |w_{\varepsilon}|^{2} \, dx \\
&\leq \frac{L^{2}}{4} \|w_{\varepsilon}\|_{H^{s}} \leq \frac{L^{2}}{4}.
\end{align*}
Accordingly, 
\begin{equation}\label{3.22}
J_{\varepsilon}(Lw_{\varepsilon})\geq \frac{1}{4}L^{2} - \int_{\R^{N}}\chi_{\varepsilon}
F(Lw_{\varepsilon}) \, dx.
\end{equation}
Using (\ref{2.7}), H\"older's inequality and $\|\chi_{\varepsilon}w_{\varepsilon}\|_{L^{p+1}(\R^{N})}\rightarrow 0$, we obtain
\begin{align}\label{3.23}
\int_{\R^{N}} \chi_{\varepsilon}F(Lw_{\varepsilon}) \, dx &\leq \int_{\R^{N}} \Bigl[\frac{\delta}{2} L^{2} |w_{\varepsilon}|^{2} + C_{\delta} 
\frac{|Lw_{\varepsilon}|^{p+1}}{p+1} \chi_{\varepsilon}(x) \Bigr] dx \nonumber\\
& \leq \delta L^{2} \|w_{\varepsilon}\|_{L^{2}(\R^{N})}^{2} + C_{\delta}L^{p+1} \|w_{\varepsilon}\|_{L^{p+1}(\R^{N})}^{p}
\|\chi_{\varepsilon}w_{\varepsilon}\|_{L^{p+1}(\R^{N})} \nonumber\\
& \leq \frac{\delta L^{2}}{V_{0}^{2}} \|w_{\varepsilon}\|^{2}_{H^{s}_{\varepsilon}} + o(1).
\end{align}
Putting together (\ref{3.22}) and (\ref{3.23}) we have
\begin{align*}
J_{\varepsilon}(Lw_{\varepsilon})\geq \frac{1}{4} L^{2} - \frac{\delta L^{2}}{V_{0}^{2}}
\|w_{\varepsilon}\|^{2}_{H^{s}_{\varepsilon}}+o(1) \quad \forall \delta >0 ,
\end{align*}
and by the arbitrariness of $\delta>0$, we get
$$
\liminf_{\varepsilon\rightarrow 0} J_{\varepsilon}(Lw_{\varepsilon})\geq \frac{1}{4}L^{2}.
$$
Since $\|v_{\varepsilon}\|_{H^{s}_{\varepsilon}}\rightarrow \infty$, we can see that $\frac{L}{\|v_{\varepsilon}\|_{H^{s}_{\varepsilon}}}\in (0,1)$ for $\varepsilon$ sufficiently small, and we deduce
$$
\max_{t \in [0,1]} J_{\varepsilon}(t v_{\varepsilon})\geq J_{\varepsilon}
\Bigl(\frac{L}{\|v_{\varepsilon}\|_{H^{s}_{\e}}}v_{\varepsilon} \Bigr)\geq \frac{1}{4}L^{2}.
$$
%Take $L>0$ sufficiently large in order to have $J_{\varepsilon}(v_{\varepsilon})\leq m_{2} <\frac{1}{4}L^{2}$. \\
Take $L>0$ sufficiently large such that $m_{2} <\frac{1}{4}L^{2}$ and we recall that $J_{\varepsilon}(v_{\varepsilon})\leq m_{2}$ by \eqref{3.4}. 
Then we can find $t_{\varepsilon} \in (0,1)$ such that 
$$
J_{\varepsilon}(t_{\varepsilon}v_{\varepsilon})= \max_{t\in [0,1]} J_{\varepsilon}(tv_{\varepsilon}).
$$ 
Hence 
$$
J_{\varepsilon}(t_{\varepsilon}v_{\varepsilon}) = \max_{t \in [0,1]} J_{\varepsilon}(t v_{\varepsilon})
\geq \frac{1}{4} L^{2} \rightarrow \infty \mbox{ as } L\rightarrow \infty,
$$ 
that is
\begin{equation}\label{3.24}
J_{\varepsilon}(t_{\varepsilon}v_{\varepsilon})\rightarrow \infty \mbox{ as } \varepsilon\rightarrow 0.
\end{equation} 
Now, using $\langle J'_{\varepsilon}(t_{\varepsilon}v_{\varepsilon}), (t_{\varepsilon}v_{\varepsilon})\rangle=0$, (\ref{3.5}) and Corollary \ref{corollario 2.1}-$(iv)$, we can see that 
\begin{align}\label{3.25}
J_{\varepsilon}(t_{\varepsilon}v_{\varepsilon}) &= J_{\varepsilon}(t_{\varepsilon}v_{\varepsilon})-
\frac{1}{2} \langle J'_{\varepsilon}(t_{\varepsilon}v_{\varepsilon}), (t_{\varepsilon}v_{\varepsilon})\rangle \nonumber\\
&= \int_{\R^{N}} \hat{G}(\varepsilon x, t_{\varepsilon}v_{\varepsilon})\, dx \nonumber\\
&\leq D^{k_{\nu}} \int_{\R^{N}} \hat{G}(\varepsilon x, v_{\varepsilon}) \, dx \nonumber\\
&= D^{k_{\nu}} \left(J_{\varepsilon}(v_{\varepsilon})-\frac{1}{2} \langle J'_{\varepsilon}(v_{\varepsilon}), v_{\varepsilon}\rangle\right) \nonumber\\
&\leq D^{k_{\nu}} m_{2} +o(1)
\end{align}
which contradicts (\ref{3.24}).
Then the Case $2$ can not take place. 

\begin{description}
\item [Step $4$] Conclusion.
\end{description}

Putting together Step $1$, Step $2$ and Step $3$,  we can deduce that $\|v_{\varepsilon}\|_{H^{s}_{\varepsilon}}$ is bounded as $\varepsilon\rightarrow 0$.
\end{proof}

\noindent
In the next lemma we prove that every Cerami sequence $(v_{j})\subset H_{\e}^{s}$ at level $c_{\e}$ is bounded and admits a convergent subsequence in $H_{\e}^{s}$.   
\begin{lem}\label{proposizione 3.1}
Assume that $f$ satisfies $(f1)$-$(f3)$ and either $(f4)$ or $(f5)$. Then there exists $\varepsilon_{1} \in (0, \varepsilon_{0}]$ such that for any $\varepsilon \in (0, \varepsilon_{1}]$ and for any $(v_{j})\subset H^{s}_{\varepsilon}$ satisfying 
\begin{align}
&J_{\varepsilon}(v_{j})\rightarrow c>0, \label{3.2} \\
&(1+ \|v_{j}\|_{H^{s}_{\varepsilon}})\|J'_{\varepsilon}(v_{j})\|_{H^{-s}_{\varepsilon}}\rightarrow 0 \label{3.3}
\mbox{ as } j\rightarrow \infty,
\end{align}
for some $c>0$, we get
\begin{compactenum}[(i)]
\item $\|v_{j}\|_{H^{s}_{\varepsilon}}$ is bounded as $j\rightarrow \infty$;
\item  there exists $(j_{k})$ and $v_{0} \in H^{s}_{\varepsilon}$ such that $v_{j_{k}} 
\rightarrow v_{0}$ strongly in $H^{s}_{\varepsilon}$. 
\end{compactenum}
\end{lem}
\begin{proof}
The proof of $(i)$ can be done in similar way to the one of Lemma \ref{proposizione 3.2}, after suitable modifications.  
More precisely, 
in Step $1$ of Lemma \ref{proposizione 3.2}, for a given sequence $(v_{j})$, there exists $(x_{j})\subset \R^{N}$ 
such that
$
\int_{B_{1}(x_{j})} |\chi_{\varepsilon}w_{j}|^{2} \, dx \rightarrow d>0.
$ 
The sequence $(x_{j})$ satisfies $\varepsilon x_{j} \in N_{\varepsilon}(\Lambda)$, and we may assume that $\varepsilon x_{j} \rightarrow x_{0} \in \overline
{N_{\varepsilon}(\Lambda)}$, where $x_{0}$ is such that $\chi(\varepsilon x+x_{0})\neq 0 \mbox{ in } B_{1}$.\\
In Step $2$ we replace (\ref{3.17}) by
\begin{equation}\label{3.26}
(-\Delta)^{s} w_{0} + V(\varepsilon x+ x_{0})w_{0} = (\chi(\varepsilon x + x_{0})a + (1- \chi(\varepsilon
x +x_{0}))\nu)w_{0} \mbox{ in } \R^{N}
\end{equation}
where $w_{0}\in H^{s}(\R^{N})$ is nonnegative and not identically zero. Indeed, using the maximum principle \cite{CabS}, we can see that $w_{0}>0$ in $\R^{N}$. Now we set $\tilde{w}(x)= w_{0}(\frac{x-x_{0}}{\varepsilon})$. Then $\tilde{w}$ satisfies
\begin{equation}\label{3.27}
\varepsilon^{2s} (-\Delta)^{s} \tilde{w} + V(x) \tilde{w}= (\chi(x)a + (1-\chi(x))\nu)\tilde{w} \, \mbox{ in } \R^{N}.
\end{equation}
We aim to prove that this is impossible for $\varepsilon>0$ sufficiently small.
Using the extension technique \cite{CS}, we can see that $\widetilde{W}:=Ext(\tilde{w})$ is a solution to the following problem
\begin{equation}
\left\{
\begin{array}{ll}
\varepsilon^{2s} \dive(y^{1-2s} \nabla \widetilde{W})=0 & \mbox{ in } \R^{N+1}_{+} \\
\frac{\partial \widetilde{W}}{\partial \nu^{1-2s}}=-V(x) \tilde{w}+ (\chi(x)a + (1-\chi(x))\nu)\tilde{w} & \mbox{ on } \partial \R^{N+1}_{+},  \\
\end{array}
\right.
\end{equation}
where we used the notation $w(x)=W(x, 0)$ to denote the trace of $W(x,y)$.\\
Take $r>0$ sufficiently small such that 
$$
\chi(x)=1 \mbox{ and } V(x)<a \quad \mbox{ for } x\in B_{r}.
$$ 
 
Let use introduce the following notations: 
\begin{align*}
B_{r}^{+}&=\{(x, y)\in \R^{N+1}_{+}: y>0, |(x, y)|< r\},\\
\Gamma_{r}^{+}&=\{(x, y)\in \R^{N+1}_{+}: y\geq 0, |(x, y)|= r\},\\
\Gamma_{r}^{0}&=\{(x, 0)\in \partial \R^{N+1}_{+}: |x|< r\},
\end{align*}
and
$$
H^{1}_{0, \Gamma^{+}_{r}}(B^{+}_{r})=\{V\in H^{1}(B_{r}^{+}, y^{1-2s}): V\equiv 0 \mbox{ on } \Gamma_{r}^{+}\}.
$$
Let
$$
\mu_{r}:=\inf\left\{\iint_{B^{+}_{r}} y^{1-2s} |\nabla U|^{2}\, dx dy: U\in H^{1}_{0, \Gamma^{+}_{r}}(B^{+}_{r}), \int_{\Gamma_{r}^{0}} u^{2} \, dx=1 \right\}.
$$
By the compactness of the embedding $H^{1}_{0, \Gamma_{r}^{+}}(B^{+}_{r})\Subset L^{2}(\Gamma^{0}_{r})$, it is not difficult to see that such infimum is achieved by a function $U_{r}\in H^{1}_{\Gamma_{r}^{+}}(B^{+}_{r})\setminus\{0\}$. Moreover, we may assume that  $U_{r}\geq 0$. Then $U_{r}$ is a solution, not identically zero, of 
\begin{equation}\label{new2.21}
\left\{
\begin{array}{ll}
\dive(y^{1-2s} \nabla U_{r})=0 & \mbox{ in } B^{+}_{r} \\
\frac{\partial U_{r}}{\partial \nu^{1-2s}}=\mu_{r} u_{r}  & \mbox{ on } \Gamma_{r}^{0}  \\
U_{r}=0 & \mbox{ on } \Gamma_{r}^{+}.
\end{array}
\right. 
\end{equation}
It follows from the strong maximum principle \cite{CabS} that $U_{r}>0$ on $B_{r}^{+}\cup \Gamma_{r}^{0}$. 
Let us note $\mu_{r}\geq 0$ and $\mu_{r}$ is a nonincreasing function of $r$. 
Indeed, $\mu_{r}$ is decreasing in $r$. In fact, if by contradiction we assume that $r_{1}<r_{2}$ and $\mu_{r_{1}}= \mu_{r_{2}}$, we can multiply the equation $\dive(y^{1-2s} \nabla U_{r_{1}})=0$ by $U_{r_{2}}$, and after an integration by parts, we can use the equalities satisfied by $U_{r_{1}}$ and $U_{r_{2}}$, and the assumption $\mu_{r_{1}}= \mu_{r_{2}}$, to deduce that 
$$
\int_{\Gamma_{r_{1}}^{+}} \frac{\partial U_{r_{1}}}{\partial \nu^{1-2s}} U_{r_{2}} d\sigma=0. 
$$
This gives a contradiction because of $U_{r_{2}}>0$ and $\frac{\partial U_{r_{1}}}{\partial \nu^{1-2s}}<0$ on $\Gamma_{r_{1}}^{+}$.  \\
Now we extend $U_{r}=0$ in $\R^{N+1}_{+}\setminus B_{r}^{+}$, so that $U_{r}\in H^{1}(\R^{N+1}_{+}, y^{1-2s})$.
Therefore, 
\begin{align*}
\varepsilon^{2s} \mu_{r} \int_{\Gamma^{0}_{r}} u_{r} \tilde{w} \, dx &= \iint_{B^{+}_{r}} y^{1-2s} \varepsilon^{2s} \nabla \widetilde{W} \nabla U_{r} \, dx dy \\
&= -\int_{\Gamma^{0}_{r}} (V(x)-a)  \tilde{w} u_{r} \, dx   
\end{align*}
that is
\begin{align}\label{new2.22}
\int_{\Gamma^{0}_{r}} (V(x)-a+\varepsilon^{2s} \mu_{r})  \tilde{w} u_{r} \, dx=0.
\end{align}
But this is impossible because of $V(x)-a+\mu_{r} \varepsilon^{2s}<0$ in $\Gamma^{0}_{r}$ for $\varepsilon>0$ small and $u_{r} \tilde{w}>0$ in $\Gamma^{0}_{r}$.

\noindent
In order to verify $(ii)$, we fix $\varepsilon\in (0, \varepsilon_{1}]$ and $(v_{j})$ satisfying 
 (\ref{3.2}) and (\ref{3.3}). Using $(i)$, we can see that $(v_{j})$ is bounded in $H^{s}_{\varepsilon}$. 
Up to a subsequence, we may assume that  
$$
v_{j} \rightharpoonup v_{0} \mbox{ in } H^{s}_{\varepsilon}.
$$ 
Our claim is to prove that  $v_{j} \rightarrow v_{0}$ in $H^{s}_{\varepsilon}$.
To do this, it is suffices to show that 
\begin{equation}\label{3.30}
\lim_{R\rightarrow \infty}\limsup_{j\rightarrow \infty} \int_{|x|\geq R} |(-\Delta)^{\frac{s}{2}} v_{j}|^{2} +V( \varepsilon x)v_{j}^{2} 
\, dx=0.
\end{equation}
Let us assume that (\ref{3.30}) is true.
Then, for any $\delta>0$ there exists $R>0$ sufficiently large such that  
\begin{equation}\label{Alv1}
\limsup_{j\rightarrow \infty} \int_{|x|\geq R} |(-\Delta)^{\frac{s}{2}} v_{j}|^{2} +V( \varepsilon x)v_{j}^{2} 
\, dx<\delta,
\end{equation}
\begin{equation}\label{Alv2}
\limsup_{j\rightarrow \infty} \int_{|x|\geq R} |(-\Delta)^{\frac{s}{2}} v_{0}|^{2} +V( \varepsilon x)v_{0}^{2} 
\, dx<\delta,
\end{equation}
and
\begin{equation}\label{Alv2}
\int_{|x|\geq R}  g(\e x, v_{0}) v_{0}\, dx<\frac{\delta}{3}.
\end{equation}
Taking into account \eqref{Alv1}, $(iii)$ of Corollary \ref{corollario 2.1} and Theorem \ref{Sembedding}, there exists $j_{0}\in \N$ such that  
\begin{equation}\label{Alv3}
\left|\int_{|x|\geq R}  g(\e x, v_{j}) v_{j}\, dx\right|<\frac{\delta}{3} \mbox{ for all } j\geq j_{0}.
\end{equation}
On the other hand, using $v_{j}\rightarrow v_{0}$ in $L^{q}(B_{R})$ for any $q\in [2, 2^{*}_{s})$, we can see that
\begin{equation}\label{Alv4}
\lim_{j\rightarrow \infty}\int_{B_{R}} g(\varepsilon x, v_{j})v_{j}\, dx= \int_{B_{R}} g(\varepsilon x, v_{0})v_{0} \, dx.
\end{equation}
From \eqref{Alv2}, \eqref{Alv3} and \eqref{Alv4}, there exists $j_{1}\geq j_{0}$ such that
\begin{equation*}
\left|\int_{\R^{N}} g(\e x, v_{j}) v_{j} \, dx-\int_{\R^{N}} g(\e x, v_{0}) v_{0} \, dx\right|<\delta \mbox{ for any } j\geq j_{1}
\end{equation*}
which implies that
\begin{align}\label{Alv5}
\lim_{j\rightarrow \infty}\int_{\R^{N}} g(\varepsilon x, v_{j})v_{j}\, dx= \int_{\R^{N}} g(\varepsilon x, v_{0})v_{0} \, dx.
\end{align}
Since $\langle J'_{\e}(v_{j}), v_{j}\rangle=o_{j}(1)$, by \eqref{Alv5} we deduce that
\begin{align}\label{Alv6}
\lim_{j\rightarrow \infty}\int_{\R^{N}} |(-\Delta)^{\frac{s}{2}}v_{j}|^{2}+V(\e x) v_{j}^{2} \, dx=\int_{\R^{N}} g(\varepsilon x, v_{0}) v_{0} \,dx,
\end{align}
and using $\langle J'_{\e}(v_{j}), v_{0}\rangle=o_{j}(1)$, we also have 
\begin{align}\label{Alv7}
\int_{\R^{N}} |(-\Delta)^{\frac{s}{2}}v_{0}|^{2}+V(\e x) v_{0}^{2} \, dx=\int_{\R^{N}} g(\varepsilon x, v_{0}) v_{0} \,dx.
\end{align}
Putting together  \eqref{Alv6} and \eqref{Alv7} we can infer that
\begin{align*} 
\lim_{j\rightarrow \infty} \|v_{j}\|_{H^{s}_{\e}}^{2}= \|v_{0}\|_{H^{s}_{\e}}^{2}.
\end{align*} 
Recalling that $H^{s}_{\varepsilon}$ is a Hilbert space we obtain that $v_{j}\rightarrow v_{0}$ in $H^{s}_{\varepsilon}$.

\noindent
Now, we show that (\ref{3.30}) holds. 
Let $\eta_{R} \in C^{\infty}(\R^{N}, \R)$ be a cut-off function such that 
\begin{align*}
\eta_{R}(x)=0 &\mbox{ for } |x|\leq \frac{R}{2}, \\
\eta_{R}(x)=1 &\mbox{ for } |x|\geq R, \\
0\leq \eta_{R}(x) \leq 1  &\mbox{ for } x\in \R^{N}, \\
|\nabla \eta_{R}(x)|\leq \frac{C}{R}  &\mbox{ for } x\in \R^{N}.
\end{align*}

Take $R>0$ such that $\frac{\Lambda}{\e}\subset B_{\frac{R}{2}}$.
Since $(v_{j} \eta_{R})$ is bounded in $H^{s}_{\e}$, we can see that $\langle J'_{\varepsilon}(v_{j}),\eta_{R}v_{j} \rangle=o_{j}(1)$.
Hence we get
\begin{align*}
\int_{\R^{N}} &(-\Delta)^{\frac{s}{2}} v_{j} (-\Delta)^{\frac{s}{2}} (v_{j} \eta_{R}) \, dx+ \int_{\R^{N}} V(\varepsilon x) v_{j}^{2}\eta_{R} \, dx \\
&= \int_{\R^{N}} \underline{f}(v_{j})v_{j}\eta_{R} \, dx +o_{j}(1) \\
&\leq \nu \int_{\R^{N}} v_{j}^{2} \eta_{R} \, dx +o_{j}(1).  
\end{align*} 
By our choice of $\nu$, we can see that there exists $\alpha \in (0,1)$ such that
\begin{align}\label{JT0}
&\iint_{\R^{N}} (-\Delta)^{\frac{s}{2}} v_{j} (-\Delta)^{\frac{s}{2}} (v_{j} \eta_{R}) \, dx +\alpha \int_{\R^{N}} V(\varepsilon x) v_{j}^{2}\eta_{R} \, dx \leq o_{j}(1).  
\end{align} 
Now we observe that 
\begin{align}\label{JT1}
&\iint_{\R^{N}} (-\Delta)^{\frac{s}{2}} v_{j} (-\Delta)^{\frac{s}{2}} (v_{j} \eta_{R}) \, dx \nonumber\\
&=\iint_{\R^{2N}} \frac{(v_{j}(x)-v_{j}(y))(v_{j}(x)\eta_{R}(x)-v_{j}(y)\eta_{R}(y))}{|x-y|^{N+2s}} \, dx dy \nonumber\\
&=\iint_{\R^{2N}} \eta_{R}(x)\frac{|v_{j}(x)-v_{j}(y)|^{2}}{|x-y|^{N+2s}} \, dx dy\nonumber \\
&+ \iint_{\R^{2N}} \frac{(v_{j}(x)-v_{j}(y))(\eta_{R}(x)-\eta_{R}(y))}{|x-y|^{N+2s}} v_{j}(y) \, dx dy \nonumber\\
&=:A_{R, j}+B_{R, j}.
\end{align}
Clearly
\begin{align}\label{JT}
A_{R, j}\geq \int_{|x|\geq R} \int_{\R^{N}}\frac{|v_{j}(x)-v_{j}(y)|^{2}}{|x-y|^{N+2s}} \, dx dy.
\end{align}
Using Lemma \ref{funlemma} and the fact that $(v_{j})$ is bounded in $H^{s}(\R^{N})$ we get 
\begin{align}\label{JT2}
\limsup_{R\rightarrow \infty} &\limsup_{j\rightarrow \infty}|B_{R, j}| \nonumber\\
&\leq \limsup_{R\rightarrow \infty}\limsup_{j\rightarrow \infty} \left(\iint_{\R^{2N}} \frac{|v_{j}(x)-v_{j}(y)|^{2}}{|x-y|^{N+2s}} \, dx dy \right)^{\frac{1}{2}} \times \nonumber \\ 
& \quad \times \left(\iint_{\R^{2N}} \frac{|\eta_{R}(x)-\eta_{R}(y)|^{2}}{|x-y|^{N+2s}} |v_{j}(y)|^{2} \, dx dy\right)^{\frac{1}{2}} \nonumber \\
&\leq C \limsup_{R\rightarrow \infty}\limsup_{j\rightarrow \infty}\left(\iint_{\R^{2N}} \frac{|\eta_{R}(x)-\eta_{R}(y)|^{2}}{|x-y|^{N+2s}} |v_{j}(y)|^{2} \, dx dy\right)^{\frac{1}{2}}=0.
\end{align}
Putting together (\ref{JT0})-(\ref{JT2}) we can infer that (\ref{3.30}) holds.

\end{proof}

\noindent
Taking into account  Lemma \ref{proposizione 3.2} and Lemma \ref{proposizione 3.1} we deduce the following result:
\begin{cor}\label{corollario 3.1}
There exists $\varepsilon_{1} \in (0, \varepsilon_{0}]$ such that for any  $\varepsilon \in (0, \varepsilon_{1}]$
there exists a critical point $v_{\varepsilon}\in H^{s}_{\varepsilon}$ of $J_{\varepsilon}(v)$ satisfying 
$J_{\varepsilon}(v_{\varepsilon})=c_{\varepsilon}$, where $c_{\varepsilon}\in [m_{1}, m_{2}]$ is defined
as in $(\ref{2.20})$-$(\ref{2.21})$. 
Moreover there exits a constant $M>0$ independent of 
$\varepsilon \in (0, \varepsilon_{1}]$ such that $\|v_{\varepsilon}\|_{H^{s}_{\varepsilon}} \leq M$ 
for any $\varepsilon\in (0, \varepsilon_{1}]$.
\end{cor}

\section{Limit equations}
In the next section we will see that the sequence of critical points obtained in Corollary \ref{corollario 3.1} converges, in some sense, to a sum of translated critical points of certain autonomous functionals.
% associated with autonomous nonlinear scalar field equations.
As proved in \cite{A2}, least energy solutions for autonomous nonlinear scalar field equations have a mountain pass characterization.  This property will be fundamental to prove Theorem \ref{teorema principale}.  
For this reason, in this section we collect some important results on autonomous functionals associated with "limit equations".
%For this reason, in this section we collect some useful results on suitable limit functionals.

Firstly, we introduce some notations and definitions which will be useful later.
For $x_{0}\in \R^{N}$ we define the autonomous functional $\Phi_{x_{0}}:H^{s}(\R^{N})\rightarrow \R$ by setting
$$
\Phi_{x_{0}}(v)= \frac{1}{2} \int_{\R^{N}} |(-\Delta)^{\frac{s}{2}} v|^{2} +V(x_{0})v^{2} \, dx - 
\int_{\R^{N}} G(x_{0}, v)\, dx.
$$
It is standard to check that $\Phi_{x_{0}}\in C^{1}(H^{s}(\R^{N}), \R)$ and critical points of $\Phi_{x_{0}}$ are weak solutions 
to the equation
\begin{align}\label{LEJT1}
(-\Delta)^{s}u+V(x_{0})u=g(x_{0}, u) \mbox{ in } \R^{N}.
\end{align}
We note that, if $u$ is a solution to \eqref{2.15}, then $v(x)=u(\e x+x_{0})$ satisfies
\begin{align}\label{LEJT2}
(-\Delta)^{s}v+V(\e x+x_{0})v=g(\e x+x_{0}, v) \mbox{ in } \R^{N},
\end{align}
that is \eqref{LEJT1} can be seen as the limit equation of \eqref{LEJT2} as $\e\rightarrow 0$.
%Since $\Phi_{x_{0}}$ is autonomous, we can look for critical points of $\Phi_{x_{0}}$ restricted to the fractional space $H^{s}_{r}(\R^{N})$.
%work on the fractional space $H^{s}_{r}(\R^{N})$.

%\noindent
For any $x_{0}\in \R^{N}$ and $u, v\in H^{s}(\R^{N})$ we use the following notations 
\begin{align*}
\langle u, v \rangle_{H^{s}_{\varepsilon}}&= \int_{\R^{N}} (-\Delta)^{\frac{s}{2}} u (-\Delta)^{\frac{s}{2}} v + V(\varepsilon x)uv \, dx \\
\langle u, v \rangle_{x_{0}} &= \int_{\R^{N}} (-\Delta)^{\frac{s}{2}} u (-\Delta)^{\frac{s}{2}} v+ V(x_{0})uv \, dx \\
\bigl\bracevert v \bigr\bracevert _{x_{0}}^{2} &= \int_{\R^{N}} |(-\Delta)^{\frac{s}{2}} v|^{2} + V(x_{0})v^{2} \, dx. 
\end{align*}

\noindent
Finally we define
%$$
%H(x, t)= -\frac{1}{2} V(x)t^{2} + G(x,t) \,\mbox{ and } \,\Omega= \Bigl\{x\in \R^{N} : \sup_{t>0} H(x, t)>0\Bigr\}.
%$$
$$
H(x, t)= -\frac{1}{2} V(x)t^{2} + G(x,t) 
$$
and
$$
\Omega= \Bigl\{x\in \R^{N} : \sup_{t>0} H(x, t)>0\Bigr\}.
$$

\begin{remark}\label{osservazione 4}
\noindent
\begin{compactenum}[(i)]
\item $\Omega \subset \Lambda$ and $0\in \{x\in \Lambda' : V(x)= \inf_{y\in \Lambda} V(y)\}\subset \Omega$.\\ 
\item If $(f3)$ or $(f5)$ with $a=\infty$ holds, then $\Omega=\Lambda$.
\end{compactenum}
\end{remark}

Now, we state the following Jeanjean-Tanaka type result \cite{JT1} proved in 
\cite{A2} (see Theorem $1$ in \cite{A2}) related to the study of the autonomous problem
\begin{equation}\label{5.1}
(-\Delta)^{s} u=h(u) \mbox{ in } \R^{N},
\end{equation}
where $h\in C^{1}(\R, \R)$ is an odd function satisfying the following Berestycki-Lions type assumptions \cite{BL1}:
\begin{compactenum}[$(h1)$]  
\item $-\infty<\liminf_{t\rightarrow 0} h(t)/t \leq \limsup_{t\rightarrow 0} h(t)/t<0$;
\item$\lim_{|t|\rightarrow \infty} \frac{h(t)}{|t|^{2^{*}_{s}-1}}=0$;
\item there exists $\bar{t}>0$ such that $H(\bar{t})>0$.
\end{compactenum}
We recall that the existence of a solution to \eqref{5.1} has been established in \cite{A2, CW}.
\begin{lem}\cite{A2}\label{prop 5.2}
Assume that $h\in C^{1}(\R, \R)$ is an odd function satisfying the Berestycki-Lions type assumptions $(h1)$-$(h3)$.
Let $\tilde{I}: H^{s}(\R^{N})\rightarrow \R$ be the functional defined by
$$
\tilde{I}(u)= \int_{\R^{N}} \frac{1}{2} |(-\Delta)^{\frac{s}{2}} u|^{2} - H(u) \,dx. 
$$
Then $\tilde{I}$ has a mountain pass geometry and
$c=m$, where $m$ is defined as 
\begin{equation}\label{5.2}
m =\inf \{\tilde{I}(u): u\in H^{s}(\R^{N})\setminus\{0\} \mbox{ is a solution to (\ref{5.1})} \},
\end{equation}
and
\begin{align*}
&c= \inf_{\gamma \in \Gamma} \max_{t\in [0, 1]} \tilde{I}(\gamma(t)) \\
& \mbox{ where } \Gamma=\{\gamma \in C([0,1], H^{s}(\R^{N})) : \gamma(0)=0, \tilde{I}(\gamma(1))<0\}.
\end{align*}
Moreover, for any least energy solution $\omega(x)$ of $(\ref{5.1})$ there exists a path 
$\gamma \in \Gamma$ such that
\begin{align}
& \tilde{I}(\gamma(t))\leq m= \tilde{I}(\omega) \quad \mbox{ for all } t\in [0,1] \label{5.3} \\
& \omega\in \gamma ([0,1]). \label{5.4}
\end{align}
\end{lem}

\noindent
At this point, we give the proof of the following lemma which we will use in the next section to obtain a concentration-compactness type result.
\begin{lem}\label{lemma 4.1}
Assume that $f$ satisfies $(f1)$-$(f3)$. Then we have
\begin{compactenum}[(i)]
\item $\Phi_{x_{0}}(v)$ has non-zero critical points if and only if $x_{0}\in \Omega$.
\item  There exists $\delta_{1}>0$, independent of $x_{0}\in \R^{N}$, such that $\bigl\bracevert v \bigr\bracevert _{x_{0}}\geq \delta_{1}$ for any non-zero critical point $v$ of $\Phi_{x_{0}}$.
\end{compactenum}
\end{lem}
\begin{proof}
Firstly, we extend $f(t)$ to an odd function on $\R$. Let us consider the function
$$
h(t)= -V(x_{0}) t + g(x_{0}, t), 
$$
that is $h(t)=H'(x_{0}, t)$. Clearly $h$ is odd. Now we show that $h$ satisfies assumptions $(h1)$-$(h3)$.
From $(f2)$ and $(f3)$ it follows that $(h1)$ and $(h2)$ hold.\\
Since $\Omega=\{x\in \R^{N} : \sup_{t >0} H(x, t)>0\}$, we can see that $(h3)$ is true if and only if $x_{0}\in \Omega$. 
Then $(i)$ follows by Theorem $1$ in \cite{A2} (see also Theorem $1.1$ in \cite{CW}).\\
Now let $v$ be a non-zero critical point of $\Phi_{x_{0}}$. Then 
$$
\langle \Phi'_{x_{0}}(v),v\rangle=0 \Rightarrow \int_{\R^{N}} |(-\Delta)^{\frac{s}{2}} v|^{2} + V(x_{0}) v^{2} \, dx- 
\int_{\R^{N}} g(x_{0}, v)v \, dx=0.
$$
Using $(i)$ of Corollary \ref{corollario 2.1} we get
$$
\|v\|_{H^{s}}^{2} - \int_{\R^{N}} f(v)v \, dx\leq 0,
$$
so by (\ref{2.7}) it follows that for any $\delta \in (0, V_{0})$
\begin{align*}
\|v\|_{H^{s}}^{2} &\leq \delta \|v\|_{L^{2}(\R^{N})}^{2} + C_{\delta}\|v\|_{L^{p+1}(\R^{N})}^{p+1} \\
&\leq \frac{\delta}{V_{0}} \|v\|_{H^{s}}^{2}+ C_{\delta}C'_{p+1}\|v\|_{H^{s}}^{p+1}. 
\end{align*}
Then
$$
\Bigl(1- \frac{\delta}{V_{0}}\Bigr) \|v\|_{H^{s}}^{2} \leq C_{\delta}C'_{p+1}\|v\|_{H^{s}}^{p+1}, 
$$
and we can find $\delta_{1}>0$ such that $\|v\|_{H^{s}}\geq \delta_{1}$ for any $x_{0}\in \R^{N}$
and for any non-zero critical point $v$. 
Since $\bigl \bracevert v \bigr \bracevert_{x_{0}} \geq \|v\|_{H^{s}}$ we can infer that $(ii)$ is verified.
\end{proof}

\noindent
For any $x\in \R^{N}$, we set 
$$
m(x):=
\left\{
\begin{array}{ll}
\mbox{least energy level of } \Phi_{x}(v) & \mbox{ if } x\in \Omega \\
\infty &  \mbox{ if }  x\in \R^{N}\setminus \Omega.
\end{array}
\right.
$$
By Lemma \ref{prop 5.2}, we can see that  $m(x)$ is equal to the mountain pass value for 
 $\Phi_{x}(v)$ if $x\in \Omega$, that is
$$
m(x)= \inf_{\gamma \in \Gamma} \Bigl( \max_{t\in [0,1]} \Phi_{x}(\gamma(t)) \Bigr)
$$
where $\Gamma=\{\gamma\in C([0, 1], H^{s}(\R^{N})): \gamma(0)=0 \mbox{ and } \Phi_{x}(\gamma(1))<0\}$.

\noindent
Now we prove the following result
\begin{lem}\label{prop 5.3}
$$
m(x_{0})= \inf_{x\in \R^{N}} m(x) \mbox{ if and only if } x_{0}\in \Lambda \mbox{ e } 
V(x_{0})= \inf_{x\in \Lambda} V(x).
$$
In particular, $m(0)= \inf_{x\in \R^{N}} m(x)$.
\end{lem}

\begin{proof}
Fix $x_{0}\in \Lambda $ such that $V(x_{0})= \inf_{x\in \Lambda} V(x)$. 
We note that $x_{0}\in \Lambda'$. Otherwise, if $x_{0}\in \Lambda\setminus\Lambda'$, then 
$$
V(x_{0})\geq \inf_{x\in \Lambda \setminus \Lambda'} V(x)> 
\inf_{x\in \Lambda} V(x)
$$ 
which is impossible. 
Hence $x_{0}\in \Lambda'$ and $\chi(x_{0})=1$.
Moreover, $x_{0}\in \Omega$ by Remark \ref{osservazione 4}. Now, using the fact that $V(x)\geq V(x_{0})$ in $\Lambda$ and $G(x, t )\leq F(t)$ 
for any $(x, t)\in \R^{N}\times \R$, we get for any $\bar{x}\in \Omega$
\begin{align*}
\Phi_{\bar{x}}(v)&= \frac{1}{2} \|(-\Delta)^{\frac{s}{2}} v\|_{L^{2}(\R^{N})}^{2} + \frac{1}{2} V(\bar{x})\|v\|_{L^{2}(\R^{N})}^{2}-
\int_{\R^{N}} G(\bar{x}, v) \,dx \\
&\geq \frac{1}{2} \|(-\Delta)^{\frac{s}{2}}  v\|_{L^{2}(\R^{N})}^{2} +\frac{1}{2} V(x_{0})\|v\|_{L^{2}(\R^{N})}^{2}-
\int_{\R^{N}} F(v) \,dx \\
&= \Phi_{x_{0}}(v)  \mbox{ for any } v\in H^{s}(\R^{N}).
\end{align*}
This implies that $m(x_{0})\leq m(x)$ for all $x\in \R^{N}$, so we have 
$
m(x_{0})\leq \inf_{x\in \R^{N}} m(x) \leq m(x_{0})
$ 
that is $m(x_{0})=\inf_{x\in \R^{N}} m(x)$.\\
Now we fix $x'\in\Lambda$ such that $V(x')>V(x_{0})$. Take $\gamma \in \Gamma$ such that
 (\ref{5.3}) and (\ref{5.4}) hold with $\tilde{I}(v)= \Phi_{x'}(v)$.
Then we deduce that 
$$
m(x_{0})\leq \max_{t\in [0,1]} \Phi_{x_{0}}(\gamma(t))< \max_{t\in [0,1]} \Phi_{x'}(\gamma(t))=m(x').
$$
\end{proof}

\noindent
Finally, we show the continuity of  $m(x)$.
\begin{prop}
The function $m(x): \R^{N} \mapsto (-\infty, \infty]$ is continuous in the following sense:
\begin{align*}
&m(x_{j})\rightarrow m(x_{0}) \quad \mbox{ if } x_{j}\rightarrow x_{0}\in \Omega \\
&m(x_{j})\rightarrow \infty        \quad \mbox{ if } x_{j}\rightarrow x_{0}\in \R^{N}\setminus\Omega.
\end{align*}
\end{prop} 

\begin{proof}
Firstly, we fix $x_{0}\in \Omega$ and we take $(x_{j})\subset \Omega$ such that $x_{j} \rightarrow x_{0}$. 
We aim to prove that $m(x)$ is upper semicontinuous, that is
$$\limsup_{j\rightarrow \infty} m(x_{j})\leq m(x_{0}).$$
In order to prove it, we show that for any fixed $\gamma \in \Gamma$, the map
$$
L_{\gamma} : x\in \Omega \mapsto \max_{t\in [0,1]} \Phi_{x}(\gamma(t))
$$
is continuous.
For any $t\in [0,1]$, we have
\begin{align*}
\Phi_{x_{j}}(\gamma(t))- \Phi_{x_{0}}(\gamma(t)) &= \frac{1}{2} \int_{\R^{N}} [V(x_{j}) - V(x_{0})]|\gamma(t) (x)|^{2} \,dx \\
&-\int_{\R^{N}} [G(x_{j}, \gamma(t) (x)) - G(x_{0}, \gamma(t) (x))] \,dx.
\end{align*}
Then, the continuity of $V$ and the definition of $G$ yield  
$$
\Bigl|\max_{t\in [0,1]} \Phi_{x_{j}}(\gamma(t)) - \max_{t\in [0,1]} \Phi_{x_{0}}(\gamma(t)) \Bigr| 
\leq \max_{t\in [0,1]} |\Phi_{x_{j}}(\gamma(t)) - \Phi_{x_{0}}(\gamma(t))|\rightarrow 0.
$$
Hence, being $m(x_{0})= \inf_{\gamma \in \Gamma} L_{\gamma}(x_{0})$, we deduce that
$m(x)$ is upper semicontinuous.
Now we show that $m(x)$ is lower semicontinuous.  
In order to achieve our aim, we prove that for any least energy solution $u_{j}(x)$ of $\Phi_{x_{j}}(v)$ we have 
\begin{compactenum}[(i)]
\item  $\|u_{j}\|_{H^{s}(\R^{N})}$ is bounded as $j\rightarrow \infty$;
\item  after extracting a subsequence, $u_{j}$ has a non-zero weak limit $u_{0}$ and 
$$
\liminf_{j\rightarrow \infty} \Phi_{x_{j}} (u_{j})\geq \Phi_{x_{0}} (u_{0}).
$$
\end{compactenum}
Indeed, it is clear that one can see that $u_{0}$ is a non-zero critical point of $\Phi_{x_{0}}(v)$, and then we have 
$$
\liminf_{j\rightarrow \infty} m(x_{j}) = \liminf_{j\rightarrow \infty} \Phi_{x_{j}}(u_{j}) 
\geq \Phi_{x_{0}}(u_{0}) \geq m(x_{0}).
$$
Assume that $u_{j}\in H^{s}_{r}(\R^{N})$. We know that  
$u_{j}(x)$ satisfies the Pohozaev Identity \cite{A1, CW, Secchi2}:
\begin{equation}\label{5.5}
\frac{N-2s}{2}\|(-\Delta)^{\frac{s}{2}} u_{j}\|_{L^{2}(\R^{N})}^{2} = N\int_{\R^{N}} H(x_{j}, u_{j}(x)) \,dx.
\end{equation}
Now, we divide the proof in several steps.
\begin{description}
\item [Step $1$] There exists $m_{0}, m_{1}>0$ (independent of $j$) such that
$$
m_{0}\leq m(x_{j})\leq m_{1} \quad \forall j\in \N. 
$$
\end{description}

\noindent
The existence of $m_{1}$ follows by the fact that $m(x)$ is upper semicontinuous. 
Concerning $m_{0}$, we note that
$$
\Phi_{x_{j}}(v)\geq \frac{1}{2} \|(-\Delta)^{\frac{s}{2}} v\|_{L^{2}(\R^{N})}^{2} +
\frac{1}{2} V_{0} \|v\|_{L^{2}(\R^{N})}^{2} - \int_{\R^{N}} F(v) \,dx.
$$
Then, denoted by $m_{0}$ the mountain pass value of
$$
v\mapsto \frac{1}{2} \|(-\Delta)^{\frac{s}{2}} v\|_{L^{2}(\R^{N})}^{2} +
\frac{1}{2} V_{0} \|v\|_{L^{2}(\R^{N})}^{2} - \int_{\R^{N}} F(v) \,dx,
$$
we get the thesis.
\begin{description}
\item [Step $2$] $\frac{N}{s}m_{0} \leq \|(-\Delta)^{\frac{s}{2}} u_{j}\|_{L^{2}(\R^{N})}^{2} \leq \frac{N}{s}m_{1}$ for any $j\in \N$.
\end{description}

\noindent
In view of (\ref{5.5}) we obtain
\begin{align*}
m(x_{j})&= \Phi_{x_{j}}(u_{j}) \\
&=\frac{1}{2} \|(-\Delta)^{\frac{s}{2}} u_{j}\|_{L^{2}(\R^{N})}^{2} - \int_{\R^{N}} H(x_{j}, u_{j}(x))\,dx \\
&= \frac{s}{N}\|(-\Delta)^{\frac{s}{2}} u_{j}\|_{L^{2}(\R^{N})}^{2}
\end{align*}
and using Step $1$ we deduce that 
$$\frac{N}{s}m_{0}\leq \|(-\Delta)^{\frac{s}{2}} u_{j}\|_{L^{2}(\R^{N})}^{2} \leq \frac{N}{s}m_{1}.$$

\begin{description}
\item [Step $3$] Boundedness of $\|u_{j}\|_{L^{2}(\R^{N})}$. 
\end{description}

\noindent
Taking into account \eqref{5.5}, the definition of $H(x,t)$, \eqref{2.2}, $|g(x,t)|\leq \delta|t|+C_{\delta}|t|^{2^{*}_{s}-1}$, Theorem \ref{Sembedding} and Step $2$, we have for any $\delta\in (0, V_{0})$
\begin{align*}
N\frac{V_{0}}{2}\|u_{j}\|^{2}_{L^{2}(\R^{N})}&\leq N\frac{\delta}{2} \|u_{j}\|^{2}_{L^{2}(\R^{N})}+N\frac{C_{\delta}}{2^{*}_{s}}\|u_{j}\|^{2^{*}_{s}}_{L^{2^{*}_{s}}(\R^{N})}\\
&\leq N\frac{\delta}{2} \|u_{j}\|^{2}_{L^{2}(\R^{N})}+N \frac{C_{\delta}}{2^{*}_{s}}S_{*}^{-\frac{2^{*}_{s}}{2}}\left(\frac{N}{s}m_{1}\right)^{\frac{2^{*}_{s}}{2}},
\end{align*}
which implies that $(u_{j})$ is a bounded sequence in $L^{2}(\R^{N})$.\\

\begin{description}
\item [Step $4$] After extracting a subsequence, $u_{j}$ has a non-zero weak limit $u_{0}$. 
\end{description}

\noindent
Gathering Step $2$ and Step $3$, we know that $(u_{j})$ is bounded in $H^{s}_{r}(\R^{N})$, and we denote by $u_{0}$ its weak limit.
Assume by contradiction that $u_{0}\equiv 0$. \\
Then, in view of Theorem \ref{Lions}, we have
\begin{align*}
&u_{j}\rightharpoonup 0 \mbox{ in } H^{s}(\R^{N}),\\
&u_{j}\rightarrow 0 \mbox{ in } L^{q}(\R^{N}) \quad \forall q\in (2, 2^{*}_{s}).
\end{align*}
Taking into account that $\langle \Phi'_{x_{j}}(u_{j}), u_{j}\rangle=0$ and Step $2$, we can deduce that 
\begin{align}\label{mora}
0<\frac{N}{s}m_{0}\leq \|(-\Delta)^{\frac{s}{2}} u_{j}\|_{L^{2}(\R^{N})}^{2}+ V(x_{j}) \|u_{j}\|^{2}_{L^{2}(\R^{N})}= \int_{\R^{N}} g(x_{j}, u_{j}) u_{j}\, dx.
\end{align}
Applying Lemma \ref{Strauss} twice (with $P(t)=f(t)t$ and $P(t)=\underline{f}(t)t$, $q_{1}=2$ and $q_{2}=p+1$) and using $\chi(x)\in[0,1]$, we can see that
\begin{equation*}
\int_{\R^{N}} g(x_{j}, u_{j}) u_{j}\, dx=\chi(x_{j})\int_{\R^{N}} f(u_{j})u_{j}\, dx+(1-\chi(x_{j}))\int_{\R^{N}} \underline{f}(u_{j})u_{j}\, dx \rightarrow 0,
\end{equation*}
which is incompatible with (\ref{mora}).

\begin{description}
\item [Step $5$] $\liminf_{j\rightarrow \infty} \Phi_{x_{j}}(u_{j})\geq \Phi_{x_{0}}(u_{0})$.
\end{description}
Let us note that
\begin{align*}
\Phi_{x_{j}}(u_{j}) 
&= \frac{1}{2}\|(-\Delta)^{\frac{s}{2}} u_{j}\|_{L^{2}(\R^{N})}^{2} 
+\frac{1}{2}V(x_{j})\|u_{j}\|_{L^{2}(\R^{N})}^{2} - 
\int_{\R^{N}} G(x_{j}, u_{j}) \,dx, 
\end{align*}
and
\begin{align*}
\|u_{0}\|_{L^{2}(\R^{N})}^{2} &\leq \liminf_{j\rightarrow \infty} \|u_{j}\|_{L^{2}(\R^{N})}^{2} \\
\|(-\Delta)^{\frac{s}{2}} u_{0}\|_{L^{2}(\R^{N})}^{2} &\leq \liminf_{j\rightarrow \infty} \|(-\Delta)^{\frac{s}{2}} u_{j}\|_{L^{2}(\R^{N})}^{2}
\end{align*} 
by the weak lower semicontinuity of the $H^{s}(\R^{N})$-norm.
On the other hand, using Theorem \ref{Lions}, Lemma \ref{Strauss} (applied to $F(t)$ and $\underline{F}(t)$) and the continuity of $\chi$, we have
\begin{align*}
\int_{\R^{N}} G(x_{j}, u_{j}) \,dx \rightarrow \int_{\R^{N}} G(x_{0}, u_{0}) \,dx 
\mbox{ as } j\rightarrow \infty.
\end{align*}
Therefore, the above facts and $V(x_{j})\rightarrow V(x_{0})$ as $j\rightarrow \infty$, yield
\begin{align*}
\liminf_{j\rightarrow \infty} \Phi_{x_{j}}(u_{j}) 
&\geq \Phi_{x_{0}}(u_{0}).
\end{align*}
Finally, we deal with the case $x_{0}\notin \Omega$.

\begin{description}
\item [Step $6$] Let $x_{0}\notin \Omega $ and $(x_{j})$ such that $x_{j}\rightarrow x_{0}$.
Then $m(x_{j})\rightarrow \infty$. 
\end{description}
We argue by contradiction, and we assume that $m(x_{j}) \not\rightarrow \infty$.
Then, there exists a subsequence, which we denote again by $(x_{j})$, 
such that $m(x_{j})$ stays bounded as $j\rightarrow \infty$. 
Arguing as in Steps $1$-$5$, we can find a non-zero critical point  
of $\Phi_{x_{0}}(v)$, which contradicts $(i)$ of Lemma \ref{lemma 4.1}.
\end{proof}

\section{$\e$-dependent concentration-compactness result}

This section is devoted to the study  of the behavior as  $\varepsilon \rightarrow 0$ of critical points $(v_{\varepsilon})$ obtained in Corollary \ref{corollario 3.1}. More generally we consider $(v_{\varepsilon})$ such that
\begin{align}
& v_{\varepsilon}\in H^{s}_{\varepsilon}, \label{4.1} \\ 
& J_{\varepsilon}(v_{\varepsilon}) \rightarrow c\in \R, \label{4.2}\\
&(1+ \|v_{\varepsilon}\|_{H^{s}_{\varepsilon}})\|J'_{\varepsilon}(v_{\varepsilon})\|_{H^{-s}_{\varepsilon}}
\rightarrow 0, \label{4.3} \\
& \|v_{\varepsilon}\|_{H^{s}_{\varepsilon}}\leq m, \label{4.4}
\end{align} 
where $c$ and $m$ are independent of $\varepsilon$.

\noindent
We begin by proving the following concentration-compactness type result.
\begin{lem}\label{prop 4.1}
Assume that $f$ satisfies $(f1)$-$(f3)$ and $(v_{\varepsilon})_{\varepsilon \in(0, \varepsilon_{1}]}$
satisfies $(\ref{4.1})$-$(\ref{4.4})$.
Then there exists a subsequence $\varepsilon_{j}\rightarrow 0$, 
$l\in \N\cup\{0\}$, sequences $(y_{\varepsilon_{j}}^{k})\subset \R^{N}$, $x^{k}\in \Omega$, 
$\omega^{k}\in H^{s}(\R^{N})\setminus\{0\}$ $(k=1,\cdots, l)$ such that
\begin{align}
& |y_{\varepsilon_{j}}^{k} - y_{\varepsilon_{j}}^{k'}|\rightarrow \infty \mbox{ as }
j\rightarrow \infty, \mbox{ for } k\neq k',  \label{4.5} \\
& \varepsilon_{j} y_{\varepsilon_{j}}^{k} \rightarrow x^{k} \in \Omega \mbox{ as }
j\rightarrow \infty, \label{4.6} \\
& \omega^{k} \not\equiv 0 \mbox{ and } \Phi'_{x^{k}}(\omega^{k})=0, \label{4.7} \\
& \left\|v_{\varepsilon_{j}} - \psi_{\varepsilon_{j}}\left(\sum_{k=1}^{l} \omega^{k}(\cdot-y_{\varepsilon_{j}}^{k})\right)
\right\|_{H^{s}_{\varepsilon_{j}}}\rightarrow 0 \mbox{ as } j\rightarrow \infty, \label{4.8} \\
& J_{\varepsilon_{j}}(v_{\varepsilon_{j}})\rightarrow \sum_{k=1}^{l} \Phi_{x^{k}}(\omega^{k}). \label{4.9}
\end{align}
Here $\psi_{\varepsilon}(x)= \psi (\varepsilon x)$, and $\psi \in C_{0}^{\infty}(\R^{N}, \R)$ is such that 
$\psi (x) =1$ for $ x\in \Lambda$ and $0\leq \psi \leq 1$ on $\R^{N}$.
When $l=0$, we have $\|v_{\varepsilon_{j}}\|_{H^{s}_{\varepsilon_{j}}}\rightarrow 0$ and $J_{\varepsilon_{j}}(v_{\varepsilon_{j}})\rightarrow 0$.
\end{lem}

\begin{remark}
Let us note that $\sup \psi(\varepsilon x)V(\varepsilon x)<\infty$. Moreover, for all $w\in H^{s}(\R^{N})$, $\psi_{\varepsilon} w\in H^{s}_{\varepsilon}$ and there exists
a constant $C>0$, independent of $\varepsilon$, such that
\begin{equation}\label{4.10}
\|\psi_{\varepsilon} w\|_{H^{s}_{\varepsilon}}\leq C\|w\|_{H^{s}}.
\end{equation}
\end{remark}
\begin{remark}
For any $\omega\in H^{s}(\R^{N})$ and for any sequence $(y_{\varepsilon})\subset \R^{N}$
such that $\varepsilon y_{\varepsilon} \rightarrow x_{0}\in \Lambda$,
we have
\begin{align}\label{psiomega}
& \|\psi_{\varepsilon} \omega(\cdot-y_{\varepsilon})\|_{H^{s}_{\varepsilon}}^{2} \nonumber\\
& =\int_{\R^{N}} |(-\Delta)^{\frac{s}{2}}(\psi(\varepsilon x + \varepsilon y_{\varepsilon})\omega(x))|^{2} +
V(\varepsilon x + \varepsilon y_{\varepsilon}) \psi(\varepsilon x + \varepsilon y_{\varepsilon})^{2}
\omega(x)^{2} \, dx \nonumber\\
& \rightarrow \int_{\R^{N}} |(-\Delta)^{\frac{s}{2}} \omega|^{2} + V(x_{0}) \omega^{2} \, dx = \bigl\bracevert 
\omega\bigr\bracevert_{x_{0}}^{2} \quad
\mbox{ as } \varepsilon\rightarrow 0.
\end{align}
We first prove that
\begin{align}\label{terry1}
\int_{\R^{N}} |(-\Delta)^{\frac{s}{2}}(\psi(\varepsilon x + \varepsilon y_{\varepsilon})\omega(x))|^{2} \, dx\rightarrow \int_{\R^{N}} |(-\Delta)^{\frac{s}{2}} \omega|^{2}\, dx.
\end{align}
Thus
\begin{align*}
&\iint_{\R^{2N}} \frac{|\psi(\varepsilon x + \varepsilon y_{\varepsilon})\omega(x)-\psi(\varepsilon y + \varepsilon y_{\varepsilon})\omega(y)|^{2}}{|x-y|^{N+2s}}\, dx dy \\
&=\iint_{\R^{2N}} \frac{|\psi(\varepsilon x + \varepsilon y_{\varepsilon})-\psi(\varepsilon y + \varepsilon y_{\varepsilon})|^{2}}{|x-y|^{N+2s}} (\omega(x))^{2}\, dx dy \\
&+\iint_{\R^{2N}} \frac{|\omega(x)-\omega(y)|^{2}}{|x-y|^{N+2s}} (\psi(\e y+\e y_{\e}))^{2}\, dx dy \\
&+2\iint_{\R^{2N}} \frac{(\psi(\varepsilon x + \varepsilon y_{\varepsilon})-\psi(\varepsilon y + \varepsilon y_{\varepsilon}))(\omega(x)-\omega(y))}{|x-y|^{N+2s}} \omega(x) \psi(\e y+\e y_{\e})\, dx dy \\
&=:A_{\e}+B_{\e}+2C_{\e}.
\end{align*}
Now, by the dominated convergence theorem and $\psi(\e \cdot+\e y_{\e})\rightarrow 1$, we get $B_{\e}\rightarrow [\omega]^{2}$. 
On the other hand
\begin{align}\label{terry}
A_{\e}&=\int_{\R^{N}} \, dx \int_{|x-y|\leq \frac{1}{\e}} \frac{|\psi(\varepsilon x + \varepsilon y_{\varepsilon})-\psi(\varepsilon y + \varepsilon y_{\varepsilon})|^{2}}{|x-y|^{N+2s}} (\omega(x))^{2}\, dy \nonumber \\
&+\int_{\R^{N}} \, dx\int_{|x-y|>\frac{1}{\e}} \frac{|\psi(\varepsilon x + \varepsilon y_{\varepsilon})-\psi(\varepsilon y + \varepsilon y_{\varepsilon})|^{2}}{|x-y|^{N+2s}} (\omega(x))^{2}\, dy \nonumber  \\
&\leq \e^{2}\|\nabla \psi\|^{2}_{L^{\infty}(\R^{N})}\alpha_{N-1}\int_{\R^{N}} \omega^{2} \, dx \int_{0}^{\frac{1}{\e}} \frac{1}{z^{2s-1}} \, dz \nonumber\\
&+4 \alpha_{N-1} \int_{\R^{N}} \omega^{2} \, dx \int_{\frac{1}{\e}}^{\infty} \frac{1}{z^{2s+1}} \, dz \nonumber \\
&=\e^{2s}\alpha_{N-1} \left(\frac{\|\nabla \psi\|^{2}_{L^{\infty}(\R^{N})}}{2-2s}+\frac{2}{s}\right)\int_{\R^{N}} \omega^{2} \, dx\rightarrow 0 \quad \mbox{ as } \e\rightarrow 0,
\end{align}
and using
\begin{align*}
|C_{\e}|\leq [\omega] \sqrt{A_{\e}}\rightarrow 0,
\end{align*}
we can infer that $(\ref{terry1})$ holds.
Since it is clear that 
\begin{align}\label{terry2}
\int_{\R^{N}}V(\varepsilon x + \varepsilon y_{\varepsilon}) \psi(\varepsilon x + \varepsilon y_{\varepsilon})^{2}
\omega(x)^{2} \, dx \rightarrow \int_{\R^{N}} V(x_{0}) \omega^{2} \, dx,
\end{align}
we deduce that $(\ref{terry1})$ and $(\ref{terry2})$ imply $(\ref{psiomega})$.
\end{remark}

\begin{proof}
We divide the proof in several steps. In what follows, we write $\varepsilon$ instead of $\varepsilon_{j}$.

\begin{description}
\item [Step $1$] Up to subsequence, $v_{\varepsilon} \rightharpoonup v_{0}$ in $H^{s}(\R^{N})$
and  $v_{0}$ is a critical point of $\Phi_{0}(v)$. 
\end{description}

\noindent
Using (\ref{4.4}) and (\ref{2.18}), we can see that $\|v_{\varepsilon}\|_{H^{s}}\leq m$. Then $(v_{\varepsilon})$ is bounded in $H^{s}(\R^{N})$ and  we can suppose that $v_{\varepsilon} \rightharpoonup v_{0}$ in 
$H^{s}(\R^{N})$. 

Let us show that $v_{0}$ is a critical point of $\Phi_{0}(v)$, that is $\langle \Phi'_{0}(v_{0}), \varphi\rangle=0$ for any $\varphi\in H^{s}(\R^{N})$.
Since $C^{\infty}_{0}(\R^{N})$ is dense in $H^{s}(\R^{N})$, it is enough to prove it for any $\varphi \in C^{\infty}_{0}(\R^{N})$.
Fix $\varphi \in C^{\infty}_{0}(\R^{N})$. From (\ref{4.3}) it follows that 
$$
\int_{\R^{N}} [(-\Delta)^{\frac{s}{2}} v_{\varepsilon} (-\Delta)^{\frac{s}{2}} \varphi + V(\varepsilon x) v_{\varepsilon} 
\varphi - g(\varepsilon x, v_{\varepsilon})
\varphi] \, dx \rightarrow 0 .
$$
Now we show that
\begin{align*}
\langle J'_{\varepsilon}(v_{\varepsilon}),\varphi \rangle &= \langle v_{\varepsilon}, \varphi\rangle_{H^{s}_{\varepsilon}} - 
\int_{\R^{N}} g(\varepsilon x, v_{\varepsilon}) \varphi \,dx  \rightarrow \langle v_{0}, \varphi \rangle_{0} - \int_{\R^{N}} g(0, v_{0}) \varphi \,dx. 
\end{align*}
Let us note that
\begin{align*}
&\langle v_{\varepsilon}, \varphi \rangle_{H^{s}_{\varepsilon}} - \langle v_{0}, \varphi \rangle_{0} \\
&= \int_{\R^{N}} (-\Delta)^{\frac{s}{2}} (v_{\varepsilon} - v_{0}) (-\Delta)^{\frac{s}{2}} \varphi \, dx + \int_{\R^{N}} [V(\varepsilon x)- V(0)] v_{\varepsilon} \varphi \, dx \\
&\quad + V(0)\int_{\R^{N}} (v_{\varepsilon}-v_{0}) \varphi \, dx \\
&=: (I)+(II)+(III).  
\end{align*}
Then $(I),(III)\rightarrow 0$ because of $v_{\varepsilon}\rightharpoonup v_{0}$ in $H^{s}(\R^{N})$,
and
\begin{align*}
|(II)|&\leq C \|V_{\varepsilon} - V(0)\|_{L^{\infty}(\supp \varphi)} \|v_{\varepsilon}\|_{H^{s}}
\|\varphi\|_{L^{2}(\R^{N})} \\
&\leq C' \|V_{\varepsilon} - V(0)\|_{L^{\infty}(\supp \varphi)}\rightarrow 0.
\end{align*}
On the other hand, using $(iii)$ of Corollary \ref{corollario 2.1} and $H^{s}(\R^{N})\Subset L^{q}_{loc}(\R^{N})$ for any $q\in [2, 2^{*}_{s})$, we have
$$
\int_{\R^{N}} g(\varepsilon x, v_{\varepsilon}) \varphi \, dx \rightarrow \int_{\R^{N}} g(0, v_{0}) \varphi \, dx. 
$$
Hence 
$$
\langle \Phi'_{0}(v_{0}),\varphi \rangle= \int_{\R^{N}} (-\Delta)^{\frac{s}{2}} v_{0} (-\Delta)^{\frac{s}{2}}  \varphi + V(0) v_{0} \varphi - g(0, v_{0})\varphi \, dx =0. 
$$

If $v_{0}\not\equiv 0$, we set $y^{1}_{\varepsilon}=0$ and $\omega^{1}=v_{0}$.

\begin{description}
\item [Step $2$] Suppose that there exist $ n\in \N\cup \{0\}, (y_{\varepsilon}^{k})\subset \R^{N},
x^{k}\in \Omega, \omega^{k}\in H^{s}(\R^{N})$ $(k=1, \ldots , n)$ such that (\ref{4.5}), (\ref{4.6}),
(\ref{4.7}) of Lemma \ref{prop 4.1} hold for $k=1, \ldots, n$ and
\begin{equation}\label{4.11}
v_{\varepsilon}(\cdot+y_{\varepsilon}^{k})\rightharpoonup \omega^{k} \mbox{ in } H^{s}(\R^{N}) \mbox{ for } k=1, \ldots, n.
\end{equation}
Moreover, we assume that 
\begin{equation}\label{4.12}
\sup_{y\in \R^{N}} \int_{B_{1}(y)} \Bigl|v_{\varepsilon} - \psi_{\varepsilon}\sum_{k=1}^{n} \omega^{k}
(x-y_{\varepsilon}^{k})\Bigr|^{2} \, dx \rightarrow 0.
\end{equation}
Then
\begin{equation}\label{4.13}
\Bigl \|v_{\varepsilon}- \psi_{\varepsilon}\sum_{k=1}^{n} \omega^{k}(\cdot-y_{\varepsilon}^{k}) \Bigr \|
_{H^{s}_{\varepsilon}}^{2} \rightarrow 0.
\end{equation}
\end{description}

\noindent
Set 
$$
\xi_{\varepsilon}(x)=v_{\varepsilon}(x)-  \psi_{\varepsilon}(x)\sum_{k=1}^{n} 
\omega^{k}(x-y_{\varepsilon}^{k}).
$$
From (\ref{4.10}) it follows that
\begin{align*}
\|\xi_{\varepsilon}\|_{H^{s}_{\varepsilon}} &\leq \|v_{\varepsilon}\|_{H^{s}_{\varepsilon}}+ 
\|\psi_{\varepsilon}\sum_{k=1}^{n} \omega^{k}(\cdot-y_{\varepsilon}^{k})\|_{H^{s}_{\varepsilon}} \\
&\leq m + C \sum_{k=1}^{n} \|\omega^{k}\|_{H^{s}},
\end{align*}
and being $\|\xi_{\varepsilon}\|_{H^{s}}\leq \|\xi_{\varepsilon}\|_{H^{s}_{\varepsilon}}$,
we deduce that $(\xi_{\varepsilon})$ is bounded in $H^{s}(\R^{N})$. 

By (\ref{4.12}) and Lemma \ref{lionslemma} we have $\|\xi_{\varepsilon}\|_{L^{p+1}(\R^{N})}
\rightarrow 0$ as $\varepsilon \rightarrow 0$.
Now, a direct calculation shows that 
\begin{align}
\|\xi_{\varepsilon}\|_{H^{s}_{\varepsilon}}^{2} &= \langle v_{\varepsilon}- \psi_{\varepsilon} \sum_{k=1}^{n} 
\omega^{k} (\cdot-y_{\varepsilon}^{k}), \xi_{\varepsilon}\rangle_{H^{s}_{\varepsilon}} \nonumber \\
&=\langle v_{\varepsilon}, \xi_{\varepsilon} \rangle_{H_{\varepsilon}} - \sum_{k=1}^{n} \langle\psi_{\varepsilon}
\omega^{k} (\cdot-y_{\varepsilon}^{k}), \xi_{\varepsilon}\rangle_{H^{s}_{\varepsilon}}. \label{4.14}
\end{align}
We aim to prove that for all $k=1, \dots, n$
\begin{equation}\label{4.15}
\langle \psi_{\varepsilon}\omega^{k}(\cdot-y_{\varepsilon}^{k}), \xi_{\varepsilon}\rangle_{H^{s}_{\varepsilon}} = 
\langle \omega^{k}(\cdot-y_{\varepsilon}^{k}), \psi_{\varepsilon}\xi_{\varepsilon}\rangle_{x^{k}}+ o(1).
\end{equation}
Indeed
\begin{align*}
&\langle \psi_{\varepsilon}\omega^{k}(\cdot-y_{\varepsilon}^{k}), \xi_{\varepsilon}\rangle_{H^{s}_{\varepsilon}} -
\langle\omega^{k}(\cdot-y_{\varepsilon}^{k}), \psi_{\varepsilon}\xi_{\varepsilon}\rangle_{x^{k}} \\
&=\Bigl[\iint_{\R^{2N}}  \frac{(\psi_{\varepsilon}(x)-\psi_{\varepsilon}(y)) (\xi_{\varepsilon}(x)-\xi_{\e}(y)) \omega^{k}(x-y_{\varepsilon}^{k})}{|x-y|^{N+2s}} \, dx dy \\
&-\iint_{\R^{2N}}  \frac{(\psi_{\varepsilon}(x)-\psi_{\varepsilon}(y)) (\omega^{k}(x-y_{\varepsilon}^{k})-\omega^{k}(y-y_{\e}^{k})) \xi_{\e}(x)}{|x-y|^{N+2s}} \, dx dy\Bigr] \\
& +\int_{\R^{N}} (V(\varepsilon x + \varepsilon y_{\varepsilon}^{k})-V(x^{k})) 
\psi(\varepsilon x+\varepsilon y_{\varepsilon}^{k}) \omega^{k}(x) \xi_{\varepsilon}(x+y_{\varepsilon}^{k}) \,dx\\
&=:(I)+(II).
\end{align*}
We note that 
\begin{align*}
&\left|\iint_{\R^{2N}}  \frac{(\psi_{\varepsilon}(x)-\psi_{\varepsilon}(y)) (\xi_{\varepsilon}(x)-\xi_{\e}(y)) \omega^{k}(x-y_{\varepsilon}^{k})}{|x-y|^{N+2s}} \, dx dy \right| \\
&\quad \leq \left(\iint_{\R^{2N}}  \frac{|\xi_{\varepsilon}(x)-\xi_{\e}(y)|^{2}}{|x-y|^{N+2s}} \, dx dy\right)^{\frac{1}{2}} \times \\
&\quad \times \left(\iint_{\R^{2N}}  \frac{|\psi_{\varepsilon}(x)-\psi_{\varepsilon}(y)|^{2} (\omega^{k}(x-y_{\varepsilon}^{k}))^{2}}{|x-y|^{N+2s}} \, dx dy \right)^{\frac{1}{2}}
\end{align*}
and
\begin{align*}
&\left|\iint_{\R^{2N}}  \frac{(\psi_{\varepsilon}(x)-\psi_{\varepsilon}(y)) (\omega^{k}(x-y_{\varepsilon}^{k})-\omega^{k}(y-y_{\e}^{k})) \xi_{\e}(x)}{|x-y|^{N+2s}} \, dx dy \right| \\
&\leq \left(\iint_{\R^{2N}}  \frac{|\omega^{k}(x-y_{\varepsilon}^{k})-\omega^{k}(y-y_{\varepsilon}^{k})|^{2}}{|x-y|^{N+2s}} \, dx dy\right)^{\frac{1}{2}}\times \\
&\quad \times \left(\iint_{\R^{2N}}  \xi^{2}_{\e}(x) \frac{|\psi_{\varepsilon}(x)-\psi_{\varepsilon}(y)|^{2}}{|x-y|^{N+2s}} \, dx dy \right)^{\frac{1}{2}}
\end{align*}
so, using the fact that $\|\xi_{\varepsilon}\|_{H^{s}}\leq \overline{C}_{1}$ and $\|\omega^{k}\|_{H^{s}}\leq \overline{C}_{2}$, for some $\bar{C}_{1}, \bar{C}_{2} >0$, we can argue as in the proof of (\ref{terry}) to see that $(I)\rightarrow 0$. We note that $(V(\varepsilon x+\varepsilon y_{\varepsilon}^{k})-V(x^{k}))\psi(\varepsilon x+\varepsilon y_{\varepsilon}^{k})$ is bounded in $L^{\infty}(\mathbb{R}^{N})$. By %using 
(\ref{4.5}) and (\ref{4.11}) we can deduce that 
\begin{equation}\begin{split}\label{4.17}
\xi_{\varepsilon} (\cdot+y_{\varepsilon}^{k}) \rightharpoonup 0 \mbox{ in } H^{s}(\R^{N}) \\
\xi_{\varepsilon} (\cdot+y_{\varepsilon}^{k}) \rightarrow 0 \mbox{ in } L^{2}_{loc}(\R^{N}).  
\end{split}\end{equation}
Then $(II)\rightarrow 0$ and we can conclude that \eqref{4.15} holds.

Putting together (\ref{4.14}) and (\ref{4.15}) we find
\begin{align*}
\| \xi_{\varepsilon}\|_{H^{s}_{\varepsilon}}^{2} &= \langle v_{\varepsilon}, \xi_{\varepsilon}\rangle_{H^{s}_{\varepsilon}} -
\sum_{k=1}^{n} \langle\omega^{k}(\cdot-y_{\varepsilon}^{k}), \psi_{\varepsilon}\xi_{\varepsilon}\rangle_{x^{k}} +o(1) \\
&= \langle J'_{\varepsilon}(v_{\varepsilon}), \xi_{\varepsilon}\rangle + \int_{\R^{N}} g(\varepsilon x, v_{\varepsilon})
\xi_{\varepsilon} \, dx - \sum_{k=1}^{n} \Bigl(\langle \Phi'_{x^{k}} (\omega^{k}(\cdot-y_{\varepsilon}^{k})),\psi_{\varepsilon} 
\xi_{\varepsilon}\rangle \\
&+ \int_{\R^{N}} g(x^{k}, \omega^{k}(x-y_{\varepsilon}^{k})) \psi_{\varepsilon} 
\xi_{\varepsilon} \, dx \Bigr) + o(1) \\
&= \int_{\R^{N}} g(\varepsilon x, v_{\varepsilon}) \xi_{\varepsilon} \,dx - \sum_{k=1}^{n} \int_{\R^{N}} 
g(x^{k}, \omega^{k}(x-y_{\varepsilon}^{k}))\psi_{\varepsilon} \xi_{\varepsilon} \, dx + o(1) \\
&=(III)- \sum_{k=1}^{n} (IV) +o(1).
\end{align*} 
By Corollary \ref{corollario 2.1}-$(iii)$ we have  
\begin{align*}
|(III)| &\leq  \delta \int_{\R^{N}} |v_{\varepsilon} \xi_{\varepsilon}| \, dx+ 
C_{\delta} \int_{\R^{N}} |v_{\varepsilon}|^{p}|\xi_{\varepsilon}| \, dx\\
&\leq \delta \|v_{\varepsilon}\|_{L^{2}(\R^{N})} \|\xi_{\varepsilon}\|_{L^{2}(\R^{N})}+ 
C_{\delta} \| v_{\varepsilon}\|_{L^{p+1}(\R^{N})}^{p} \|\xi_{\varepsilon}\|_{L^{p+1}(\R^{N})}
\end{align*}
and using $\|\xi_{\varepsilon}\|_{L^{p+1}(\R^{N})} \rightarrow 0$ as $\varepsilon\rightarrow 0$, the boundedness of
$\|v_{\varepsilon}\|_{L^{2}(\R^{N})}$ and $\|\xi_{\varepsilon}\|_{L^{2}(\R^{N})}$, and the arbitrariness of $\delta$,
we get $(III)\rightarrow 0$.
In view of (\ref{4.17}) we can see that $(IV)\rightarrow 0$. Hence $\|\xi_{\varepsilon}\|_{H^{s}_{\varepsilon}}\rightarrow 0$ and (\ref{4.13}) holds.

\begin{description}
\item [Step $3$] Suppose that there exist $n\in \N\cup\{0\}, (y^{k}_{\varepsilon})\subset \R^{N},
x^{k}\in \Omega, \omega^{k}\in H^{s}(\R^{N})\setminus\{0\}$ $(k=1, \ldots, n)$ such that (\ref{4.5}),(\ref{4.6}),
(\ref{4.7}) and (\ref{4.11}) hold. We also assume that 
there exists $z_{\varepsilon}\in \R^{N}$ such that
\begin{equation}\label{4.18}
\int_{B_{1}(z_{\varepsilon})} \Bigl|v_{\varepsilon}- \psi_{\varepsilon} \sum_{k=1}^{n} \omega^{k} 
(x-y_{\varepsilon}^{k})\Bigr|^{2} \, dx \rightarrow c>0. 
\end{equation}
Then there exist $x^{k+1}\in \Omega \mbox{ and } \omega^{k+1}\in H^{s}(\R^{N})\setminus\{0\}$ such that
\begin{align}
& |z_{\varepsilon} -y_{\varepsilon}^{k}|\rightarrow \infty \quad \mbox{ for all } k=1, \ldots , n \label{4.19}, \\
& \varepsilon z_{\varepsilon} \rightarrow x^{k+1} \in \Omega \label{4.20}, \\
& v_{\varepsilon}(\cdot+z_{\varepsilon}) \rightharpoonup \omega^{k+1}\not\equiv 0 \mbox{ in } H^{s}(\R^{N}) \label{4.21}, \\
& \Phi '_{x^{k+1}} (\omega^{k+1})=0 \label{4.22}.
\end{align}
\end{description}
It is standard to prove that $z_{\varepsilon}$ satisfies (\ref{4.19}) and that there exists $\omega^{k+1}\in H^{s}(\R^{N})\setminus \{0\}$ satisfying \eqref{4.21}.

Now we show (\ref{4.20}). Firstly, we prove that 
$\limsup_{\varepsilon \rightarrow 0} |\varepsilon z_{\varepsilon}|<\infty$.
Assume by contradiction that $|\varepsilon z_{\varepsilon}| \rightarrow \infty$.
Let $\varphi \in C_{0}^{\infty}(\R^{N})$ be a cut-off function such that $\varphi\geq 0$, $\varphi(0)=1$ and let $\varphi_{R}(x)=\varphi(x/R)$. Since $(\varphi_{R}(\cdot-z_{\varepsilon})v_{\varepsilon})$ is bounded in $H^{s}_{\e}$, we obtain  
$$
\langle J'_{\varepsilon}(v_{\varepsilon}), \varphi_{R}(\cdot-z_{\varepsilon})v_{\varepsilon} \rangle \rightarrow 0 \quad \mbox{ as } \varepsilon \rightarrow 0,
$$
that is 
\begin{align}\label{4.23}
&\int_{\R^{N}} \!\!\!(-\Delta)^{\frac{s}{2}} v_{\varepsilon} (x+z_{\varepsilon}) (-\Delta)^{\frac{s}{2}} (\varphi_{R}(x)v_{\varepsilon}(x+z_{\varepsilon}))\!+\!V(\varepsilon x + \varepsilon z_{\varepsilon}) v^{2}_{\varepsilon}(x+z_{\varepsilon}) \varphi_{R}(x) dx \nonumber\\
&-\int_{\R^{N}} g(\varepsilon x + \varepsilon z_{\varepsilon}, v_{\varepsilon}(x+z_{\varepsilon}))
v_{\varepsilon}(x+z_{\varepsilon})\varphi_{R}(x) \,dx \rightarrow 0 .
\end{align}
Let us note that $|\varepsilon z_{\varepsilon}| \rightarrow \infty$ yields
\begin{align*}
g(\varepsilon x + \varepsilon z_{\varepsilon}, v_{\varepsilon}(x+z_{\varepsilon})) 
= \underline{f}(v_{\varepsilon}(x+z_{\varepsilon})) \mbox{ on } \supp \varphi_{R} 
\end{align*}
for any $\varepsilon$ small enough. Moreover, $\varphi_{R}(x)\rightarrow 1$ as $R\rightarrow \infty$ and 
$$
|\underline{f}(\omega^{k+1})\omega^{k+1} \varphi_{R}|\leq C_{1} |\omega^{k+1}|^{2}+C_{2} |\omega^{k+1}|^{p+1}\in L^{1}(\R^{N}).
$$
in view of Lemma \ref{lemma 2.2}-$(iii)$ and Lemma \ref{lemma 2.1}-$(i)$.
Hence, by invoking the dominated convergence theorem we infer that
\begin{align}\label{fbar}
\lim_{R\rightarrow \infty}\lim_{\e\rightarrow 0} &\int_{\R^{N}} g(\varepsilon x + \varepsilon z_{\varepsilon}, v_{\varepsilon}(x+z_{\varepsilon}))
v_{\varepsilon}(x+z_{\varepsilon})\varphi_{R}(x) \,dx \nonumber \\
&=\lim_{R\rightarrow \infty} \int_{\R^{N}} \underline{f}(\omega^{k+1})\omega^{k+1} \varphi_{R} \, dx \nonumber\\
&= \int_{\R^{N}} \underline{f}(\omega^{k+1})\omega^{k+1} \, dx.
\end{align}
On the other hand, using (\ref{4.21}), H\"older's inequality and Lemma \ref{funlemma} (with $\eta_{R}=1-\varphi_{R}$), we can see that
\begin{equation}\label{4.24}
\lim_{R\rightarrow \infty}\limsup_{\e\rightarrow 0}\iint_{\R^{N}} \frac{(v_{\e}(x+z_{\e})-v_{\e}(y+z_{\e}))(\varphi_{R}(x)-\varphi_{R}(y))}{|x-y|^{N+2s}} v_{\e}(y+z_{\e}) \, dx dy=0,
\end{equation}
and applying Fatou's Lemma and (\ref{4.21}), we get 
\begin{align}\label{4.25}
\lim_{R\rightarrow \infty}\liminf_{\varepsilon \rightarrow 0} &\iint_{\R^{2N}} \frac{|v_{\varepsilon}(x+z_{\varepsilon})-v_{\varepsilon}(y+z_{\varepsilon})|^{2}}{|x-y|^{N+2s}} \varphi_{R}(x) \, dx dy \nonumber \\
&\geq \iint_{\R^{2N}} \frac{|\omega^{k+1}(x)-\omega^{k+1}(y)|^{2}}{|x-y|^{N+2s}} \, dx dy. 
\end{align}
Taking into account (\ref{4.23}), (\ref{fbar}), (\ref{4.24}) and (\ref{4.25}), we deduce that
\begin{equation}\label{4.27}
\iint_{\R^{2N}} \frac{|\omega^{k+1}(x)-\omega^{k+1}(y)|^{2}}{|x-y|^{N+2s}} \, dx dy+ \int_{\R^{N}} V_{0}(\omega^{k+1})^{2}-\underline{f}(\omega^{k+1})\omega^{k+1} \, dx\leq 0.
\end{equation}
By Lemma \ref{lemma 2.2} $(i)$-$(ii)$ and (\ref{4.27}), we have 
$$
\iint_{\R^{2N}} \frac{|\omega^{k+1}(x)-\omega^{k+1}(y)|^{2}}{|x-y|^{N+2s}} \, dx dy+ \int_{\R^{N}} (V_{0}-\nu)(\omega^{k+1})^{2} \, dx\leq 0.
$$
Since $V_{0}>\nu$, we infer that  $\omega^{k+1}\equiv 0$, which contradicts (\ref{4.21}). 

Then, $\limsup_{\varepsilon \rightarrow 0} |\varepsilon z_{\varepsilon}|
<\infty$ and there exists $x^{k+1}\in \R^{N}$ such that $\varepsilon z_{\varepsilon} \rightarrow x^{k+1}$.
This and the fact that $\langle J'_{\varepsilon}(v_{\varepsilon}), \varphi(\cdot-z_{\varepsilon})\rangle \rightarrow 0$ for any $\varphi\in C^{\infty}_{0}(\R^{N})$, gives $\Phi'_{x^{k+1}}(\omega^{k+1})=0$. Since $\omega^{k+1} \not\equiv 0$, it follows that $x^{k+1} \in \Omega$ by Lemma \ref{lemma 4.1} $(i)$.  

\begin{description}
\item [Step $4$] Conclusion.
\end{description}

\noindent
Let us suppose that $v_{0}\not\equiv 0$. Then we set $y_{\varepsilon}^{1}=0$, $x^{1}=0$, $\omega^{1}=v_{0}$.

If $\|v_{\varepsilon}-\psi_{\varepsilon}\omega^{1}\|_{H^{s}_{\varepsilon}}\rightarrow 0$, then
(\ref{4.5})-(\ref{4.8}) are satisfied by $0\in \Omega$, $v_{0}\not\equiv 0$ and $\Phi'_{0}(v_{0})=0$.

If $\|v_{\varepsilon}-\psi_{\varepsilon}\omega^{1}\|_{H^{s}_{\varepsilon}}$ does not converge to $0$,  
then (\ref{4.12}) in Step $2$ does not occur, and there exists $(z_{\varepsilon})$ satisfying (\ref{4.18}) in Step $3$. In view of Step $3$, there exist $x^{2}, \omega^{2}$  satisfying (\ref{4.19})-(\ref{4.22}). 
Then we set $y_{\varepsilon}^{2}=z_{\varepsilon}$.
If $\|v_{\varepsilon} - \psi_{\varepsilon}(\omega^{1} + \omega^{2}(\cdot-y_{\varepsilon}^{2}))\|_{H^{s}_{\varepsilon}}
\rightarrow 0$ then (\ref{4.5})-(\ref{4.8}) hold because of
$|y_{\varepsilon}^{2} - y_{\varepsilon}^{1}| = |z_{\varepsilon}|\rightarrow \infty$, 
$\varepsilon y_{\varepsilon}^{2} \rightarrow x^{2}\in \Omega$ and 
$\Phi'_{x^{2}}(\omega^{2})=0$. Otherwise, we can use Step $2$ and $3$ to continue this procedure. 

Now we assume that $v_{0}\equiv0$. If $\|v_{\varepsilon}\|_{H^{s}_{\varepsilon}}\rightarrow 0$, 
we have done. 
Otherwise, condition (\ref{4.12}) in Step $2$ does not occur, and we can find $(z_{\varepsilon})$ satisfying (\ref{4.18}) in Step $3$.
Applying Step $3$, there exist $x^{1}$ and $\omega^{1}$ satisfying (\ref{4.19})-(\ref{4.22}). Thus, we set $y_{\varepsilon}^{1}=z_{\varepsilon}$.

At this point, we aim to show that this process ends after a finite numbers of steps.
Firstly, we show that under assumptions (\ref{4.5})-(\ref{4.7}) and (\ref{4.11})
\begin{equation}\label{4.28}
\lim_{\varepsilon \rightarrow 0} \left\|v_{\varepsilon} - \psi_{\varepsilon} 
\sum_{k=1}^{n} \omega^{k}(\cdot-y_{\varepsilon}^{k})\right\|_{H^{s}_{\varepsilon}}^{2} = 
\lim_{\varepsilon \rightarrow 0 } \|v_{\varepsilon}\|_{H^{s}_{\varepsilon}}^{2} - 
\sum_{k=1}^{n} \bigl\bracevert \omega^{k}\bigr\bracevert _{x^{k}}^{2}.
\end{equation}
Let us note that
\begin{align}\label{4.29}
&\left\|v_{\varepsilon} - \psi_{\varepsilon} \sum_{k=1}^{n} \omega^{k}(\cdot-y_{\varepsilon}^{k})\right\|_{H^{s}_{\varepsilon}}^{2} \nonumber \\
&=\|v_{\varepsilon}\|_{H^{s}_{\varepsilon}}^{2} - 2 \sum_{k=1}^{n} \langle v_{\varepsilon}, \psi_{\varepsilon} 
\omega^{k}(\cdot-y_{\varepsilon}^{k}) \rangle_{H^{s}_{\varepsilon}} 
+ \sum_{k, k'} \langle \psi_{\varepsilon} 
\omega^{k}(\cdot-y_{\varepsilon}^{k}), \psi_{\varepsilon} \omega^{k'}(\cdot-y_{\varepsilon}^{k'})\rangle_{H^{s}_{\varepsilon}}.
\end{align}
Now we show that
\begin{align}\label{4.30}
\langle v_{\varepsilon} &, \psi_{\varepsilon} \omega^{k}(\cdot-y_{\varepsilon}^{k}) \rangle_{H^{s}_{\varepsilon}}\rightarrow \int_{\R^{N}} |(-\Delta)^{\frac{s}{2}} \omega^{k}|^{2} + V(x^{k})(\omega^{k})^{2} \,dx = 
\bigl\bracevert \omega^{k}\bigr\bracevert _{x^{k}}^{2}. 
\end{align}
In fact
\begin{align*}
&\langle v_{\varepsilon}, \psi_{\varepsilon} \omega^{k}(\cdot-y_{\varepsilon}^{k}) \rangle_{H^{s}_{\varepsilon}}\\
&=\iint_{\R^{2N}} \frac{(v_{\e}(x+y^{k}_{\e})-v_{\e}(y+y^{k}_{\e}))(\psi_{\e}(x+y_{\e}^{k})-\psi_{\e}(y+y_{\e}^{k}))}{|x-y|^{N+2s}} \omega^{k}(x) \, dx dy \\
&\quad +\iint_{\R^{2N}} \frac{(v_{\e}(x+y^{k}_{\e})-v_{\e}(y+y^{k}_{\e}))(\omega^{k}(x)-\omega^{k}(y))}{|x-y|^{N+2s}} \psi_{\e}(y+y_{\e}^{k}) \, dx dy \\
&\quad+\int_{\R^{N}} V(\e x+\e y_{\e}^{k}) \psi_{\e}(x+y_{\e}^{k}) v_{\e}(x+y_{\e}^{k}) \omega^{k}(x) \, dx\nonumber \\
&=:(I)+(II)+(III).
\end{align*}
Using H\"older's inequality and the boundedness of $v_{\e}(\cdot+y^{k}_{\e})$ we can argue as in the proof of (\ref{terry}) to see that $(I)\rightarrow 0$.

Concerning $(II)$ we can observe that
\begin{align*}
&\iint_{\R^{2N}} \frac{(v_{\e}(x+y^{k}_{\e})-v_{\e}(y+y^{k}_{\e}))(\omega^{k}(x)-\omega^{k}(y))}{|x-y|^{N+2s}} \psi_{\e}(y+y_{\e}^{k}) \, dx dy \\
&=\iint_{\R^{2N}} \frac{[(v_{\e}(x+y^{k}_{\e})-v_{\e}(y+y^{k}_{\e}))(\omega^{k}(x)-\omega^{k}(y))]}{|x-y|^{N+2s}} \, dx dy \\
&\quad +\iint_{\R^{2N}} \frac{(\psi_{\e}(y+y_{\e}^{k})-1)(v_{\e}(x+y^{k}_{\e})-v_{\e}(y+y^{k}_{\e}))(\omega^{k}(x)-\omega^{k}(y))}{|x-y|^{N+2s}} \, dx dy \nonumber \\
&=:(II)_{1}+(II)_{2}.
\end{align*}
Due to the fact that $v_{\e}(\cdot+y^{k}_{\e})\rightharpoonup \omega^{k}$ in $H^{s}(\R^{N})$, we obtain that $(II)_{1}\rightarrow [\omega^{k}]^{2}$.
On the other hand, using H\"older's inequality and the fact that $v_{\e}(\cdot+y_{\e}^{k})$ is bounded in $H^{s}(\R^{N})$, we have
\begin{align*}
|(II)_{2}|\leq C\left(\iint_{\R^{2N}} \frac{|(\psi_{\e}(x+y^{k}_{\e})-1)(\omega^{k}(x)-\omega^{k}(y))|^{2}}{|x-y|^{N+2s}}  \, dx dy\right)^{\frac{1}{2}}\rightarrow 0
\end{align*}
in view of the dominated convergence theorem.
Since it is clear that $(III)\rightarrow \int_{\R^{N}} V(x^{k}) (\omega^{k})^{2} \, dx$, we deduce that (\ref{4.30}) holds.
In a similar fashion, we can obtain  
\begin{equation}\label{4.31}
\langle\psi_{\varepsilon} \omega^{k}(\cdot-y_{\varepsilon}^{k}), \psi_{\varepsilon} \omega^{k'}(\cdot-y_{\varepsilon}^{k'})\rangle
_{H^{s}_{\varepsilon}}\rightarrow
\left\{
\begin{array}{cc}
0 &\mbox{ if } k\neq k' \\
\bigl\bracevert \omega^{k}\bigr\bracevert _{x^{k}}^{2} &\mbox{ if } k=k'.
\end{array} 
\right.
\end{equation}
Putting together (\ref{4.29}), (\ref{4.30}) and (\ref{4.31}), we can infer that (\ref{4.28}) holds. 
Now, (\ref{4.28}) yields that 
$$
\sum_{k=1}^{n} \bigl\bracevert \omega^{k}\bigr\bracevert _{x^{k}}^{2} \leq 
\lim_{\varepsilon \rightarrow 0} \|v_{\varepsilon}\|_{H^{s}_{\varepsilon}}^{2},
$$
and using Lemma \ref{lemma 4.1}-$(ii)$ and (\ref{4.4}) we get
$$
\delta^{2}_{1} n \leq \lim_{\varepsilon \rightarrow 0} \|v_{\varepsilon}\|_{H^{s}_{\varepsilon}}^{2} \leq m^{2}.
$$
Therefore, the procedure to find $(y_{\varepsilon}^{k}), x^{k}, \omega^{k}$ can not be iterated infinitely many times. 
Hence there exist $l\in \N \cup \{0 \}$, $(y_{\varepsilon}^{k}), x^{k}, \omega^{k}$ such that (\ref{4.5})-(\ref{4.8}) hold. 
Clearly, (\ref{4.9}) follows in a standard way by (\ref{4.5})-(\ref{4.8}).
\end{proof}

\noindent
In the next lemma we investigate the behavior of $c_{\varepsilon}$ as $\varepsilon \rightarrow 0$. 

\begin{lem}\label{proposizione 6.1}
Let $(c_{\varepsilon})_{\varepsilon\in (0, \varepsilon_{1}]}$ be the mountain pass value of $J_{\e}$ defined in $(\ref{2.20})$-$(\ref{2.21})$. Then
$$
c_{\varepsilon}\rightarrow m(0)=\inf_{x\in \R^{N}} m(x) \mbox{ as } \varepsilon\rightarrow 0.
$$
\end{lem}
\begin{proof}
From Lemma \ref{prop 5.2} we can find  a path $\gamma \in C([0,1], H^{s}(\R^{N}))$
such that  $\gamma (0)=0$, $\Phi_{0}(\gamma(1))<0$, $\Phi_{0}(\gamma(t))\leq m(0)$ for all $t\in [0,1]$, and
$$
\max_{t\in [0,1]} \Phi_{0}(\gamma(t))=m(0).
$$
Take $\varphi\in C^{\infty}_{0}(\R^{N})$ such that $\varphi(0)=1$ and $\varphi\geq 0$, and we set 
$$
\gamma_{R}(t)(x)= \varphi\left(\frac{x}{R}\right)\gamma(t)(x).
$$
Thus, it is easy to check that $\gamma_{R}(t)\in C([0,1], H^{s}_{\varepsilon}(\R^{N}))$, $\gamma_{R}(0)=0$ and $\Phi_{0}(\gamma_{R}(1))<0$
for any $R>1$ sufficiently large. Therefore $\gamma_{R}(t)\in \Gamma_{\varepsilon}$. 
Now, fixed $R>0$, we can see that $\max_{t\in [0,1]} |J_{\varepsilon}(\gamma_{R}(t))-\Phi_{0}(\gamma_{R}(t))|\rightarrow 0$ as
$\varepsilon \rightarrow 0$. Hence, for any $R>1$ large enough, 
we get 
$$
c_{\varepsilon}\leq \max_{t\in [0,1]} J_{\varepsilon}(\gamma_{R}(t))\rightarrow
\max_{t\in [0,1]} \Phi_{0}(\gamma_{R}(t)) \mbox{ as } \varepsilon\rightarrow 0.
$$ 
On the other hand 
$$
\max_{t\in [0,1]} \Phi_{0}(\gamma_{R}(t)) \rightarrow m(0)\mbox{ as } R\rightarrow \infty,
$$
so we deduce that $\limsup_{\varepsilon\rightarrow 0} c_{\varepsilon} \leq m(0)$. 

In order to complete the proof, we prove that $\liminf_{\varepsilon\rightarrow 0} c_{\varepsilon}\geq m(0)$. Let $v_{\varepsilon} 
\in H^{s}_{\varepsilon}$ be a critical point of $J_{\varepsilon}(v)$ associated to $c_{\varepsilon}$.
From Lemma \ref{prop 4.1}, there exist $\varepsilon_{j}\rightarrow 0, l\in \N 
\cup \{0\}$, $(y_{\varepsilon_{j}}^{k})\subset \R^{N}$, $x^{k}\in \Omega$, $\omega^{k}\in H^{s}(\R^{N})\setminus \{0\}$ $(k=1,\ldots, l)$ satisfying $(\ref{4.5})$-$(\ref{4.9})$. 
If by contradiction $l=0$, then (\ref{4.9}) yields $c_{\varepsilon_{j}}= J_{\varepsilon_{j}}(v_{\varepsilon_{j}})
\rightarrow 0$ which contradicts Corollary \ref{corollario 2.2}. 
Consequently, $l\geq 1$ and using (\ref{4.9}) and Lemma \ref{prop 5.3} we have 
$$
\liminf_{j\rightarrow \infty} c_{\varepsilon_{j}} = \sum_{k=1}^{l} \Phi_{x^{k}}(\omega^{k}) 
\geq \sum_{k=1}^{l} m(x^{k})\geq lm(0) \geq m(0).
$$
\end{proof}

\noindent
From Lemma \ref{proposizione 6.1} we deduce the following result.
\begin{lem}\label{proposizione 6.2}
For any $\varepsilon \in (0,\varepsilon_{1}]$, let us denote by $v_{\varepsilon}$ a critical point of  $J_{\varepsilon}$ corresponding to $c_{\varepsilon}$. Then for any sequence $\varepsilon_{j}\rightarrow 0$ we can find a subsequence, still denoted by $\varepsilon_{j}$, and $y_{\varepsilon_{j}}, x^{1}, \omega^{1}$ such that 
\begin{align}
& \varepsilon_{j} y_{\varepsilon_{j}} \rightarrow x^{1}, \label{6.1}\\
& x^{1} \in \Lambda' : V(x^{1})= \inf_{x\in \Lambda} V(x), \label{6.2}\\
& \omega^{1}(x) \mbox{ is a least energy solution of } \Phi'_{x^{1}}(v)=0,\label{6.3} \\
&\|v_{\varepsilon_{j}} - \psi_{\varepsilon_{j}} w^{1}(\cdot-y_{\varepsilon_{j}})\|_{H^{s}_{\varepsilon_{j}}}\rightarrow 0, 
\label{6.4}\\
&J_{\varepsilon_{j}}(v_{\varepsilon_{j}})\rightarrow m(x^{1})=m(0). \label{6.5}
\end{align}
\end{lem}

\section{Proof of Theorem \ref{teorema principale}}
In this last section we provide the proof of Theorem \ref{teorema principale}. From Corollary \ref{corollario 3.1}, we can see that
there exists $\e_{1}\in (0, \e_{0}]$ such that 
for any $\e\in (0, \e_{1}]$, there exists a critical point $v_{\e}\in H^{s}_{\e}$ of $J_{\e}$ satisfying $J_{\e}(v_{\e})=c_{\e}$.
Then, by Lemma \ref{proposizione 6.2} we know that for any sequence $\e_{j}\rightarrow 0$, there exists a subsequence $\e_{j}$ and $(y_{\varepsilon_{j}})\subset \R^{N}$, $x^{1}\in \Lambda'$, $\omega^{1}\in H^{s}(\R^{N})\setminus \{0\}$ satisfying (\ref{6.1})-(\ref{6.5}). Moreover, by the maximum principle \cite{CabS} $v_{\e_{j}}>0$ in $\R^{N}$.
In view of (\ref{2.18}) and (\ref{6.4}) we obtain 
\begin{equation}\label{AM1}
\|v_{\varepsilon_{j}} - \psi_{\varepsilon_{j}} \omega^{1}(\cdot-y_{\varepsilon_{j}})\|_{H^{s}(\R^{N})}\rightarrow 0.
\end{equation}
We also note that (\ref{4.28}) and (\ref{AM1}) yield
\begin{equation}\label{vincenzo}
\lim_{j\rightarrow \infty} \|v_{\e_{j}}\|^{2}_{H^{s}_{\e_{j}}}=\bigl\bracevert\omega^{1}\bigr\bracevert^{2}_{x^{1}}\neq 0.
\end{equation}
Let $\tilde{v}_{\varepsilon_{j}}(x):=v_{\e_{j}}(x+y_{\e_{j}})$.
Arguing as in the proof of (\ref{terry}), and using $\psi(x^{1})=1$, (\ref{6.1}) and the dominated convergence theorem, we can see that 
\begin{align*}
&[\psi_{\e_{j}}(\cdot+y_{\e_{j}})\omega^{1}-\omega^{1}]^{2} \\
&\leq 2\iint_{\R^{2N}} \frac{|\psi_{\e_{j}}(x+y_{\e_{j}})-\psi_{\e_{j}}(y+y_{\e_{j}})|^{2}}{|x-y|^{N+2s}}(\omega^{1}(x))^{2} \, dx dy\\ &\quad +2 \iint_{\R^{2N}} \frac{|\psi_{\e_{j}}(y+y_{\e_{j}})-1|^{2}}{|x-y|^{N+2s}}|\omega^{1}(x)-\omega^{1}(y)|^{2} \, dx dy\rightarrow 0.
\end{align*}
Clearly
\begin{align*}
\int_{\R^{N}} |\psi_{\e_{j}}(x+y_{\e_{j}}) \omega^{1}-\omega^{1}|^{2} \, dx \rightarrow 0.
\end{align*}
These two facts, together with (\ref{AM1}), imply that 
\begin{equation}\label{AM2}
\|\tilde{v}_{\varepsilon_{j}} - \omega^{1}\|_{H^{s}(\R^{N})}\rightarrow 0.
\end{equation}
Now we prove the following lemma which will be fundamental to study the behavior of the maximum points of solutions of (\ref{P}).
\begin{lem}\label{Moser}
There exists $K>0$ such that
$$
\|\tilde{v}_{\varepsilon_{j}}\|_{L^{\infty}(\R^{N})}\leq K \mbox{ for all } j\in \N.
$$
\end{lem}
\begin{proof}
Let $\beta\geq 1$ and $T>0$, and we introduce the following function
\begin{equation*}
\varphi(t)=\varphi_{T, \beta}(t)=
\left\{
\begin{array}{ll}
0 & \mbox{ if } t\leq 0 \\
t^{\beta} & \mbox{ if } 0<t<T \\
\beta T^{\beta-1}(t-T)+T^{\beta} & \mbox{ if } t\geq T. \\
\end{array}
\right.
\end{equation*}
Since $\varphi$ is convex and Lipschitz, we can see that for any $u\in \mathcal{D}^{s, 2}(\R^{N})$
\begin{align*}
&\varphi(u)\in \mathcal{D}^{s, 2}(\R^{N}) \\
&(-\Delta)^{s}\varphi(u)\leq \varphi'(u)(-\Delta)^{s}u.
\end{align*}
Now, using Theorem \ref{Sembedding}, an integration by parts, $(V1)$, $\tilde{v}_{\e_{j}}\geq 0$, and the growth assumptions on $g$, we have
\begin{align*}
&\|\varphi(\tilde{v}_{\e_{j}})\|^{2}_{L^{2^{*}_{s}}(\R^{N})}\\
&\leq S_{*}^{-1} \int_{\R^{N}} |(-\Delta)^{\frac{s}{2}} \varphi(\tilde{v}_{\e_{j}})|^{2}\, dx \\
&=S_{*}^{-1} \int_{\R^{N}} \varphi(\tilde{v}_{\e_{j}})(-\Delta)^{s} \varphi(\tilde{v}_{\e_{j}})\, dx \\
&\leq S_{*}^{-1} \int_{\R^{N}} \varphi(\tilde{v}_{\e_{j}}) \varphi'(\tilde{v}_{\e_{j}})(-\Delta)^{s} \tilde{v}_{\e_{j}} \, dx \\
&\leq CS_{*}^{-1} \int_{\R^{N}} \varphi(\tilde{v}_{\e_{j}}) \varphi'(\tilde{v}_{\e_{j}})(1+\tilde{v}_{\e_{j}}^{2^{*}_{s}-1})\, dx  \\
&=CS_{*}^{-1} \left(\int_{\R^{N}} \varphi(\tilde{v}_{\e_{j}}) \varphi'(\tilde{v}_{\e_{j}}) \, dx +  \int_{\R^{N}} \varphi(\tilde{v}_{\e_{j}}) \varphi'(\tilde{v}_{\e_{j}})\tilde{v}_{\e_{j}}^{2^{*}_{s}-1} \, dx\right),
\end{align*}
where $C$ is a constant independent of $\beta$ and $j$.

In view of $\varphi(\tilde{v}_{\e_{j}})\varphi'(\tilde{v}_{\e_{j}})\leq \beta \tilde{v}_{\e_{j}}^{2\beta-1}$ and $\tilde{v}_{\e_{j}}\varphi'(\tilde{v}_{\e_{j}})\leq \beta \varphi(\tilde{v}_{\e_{j}})$, we get
\begin{align}\label{Me1}
\|\varphi(\tilde{v}_{\e_{j}})\|^{2}_{L^{2^{*}_{s}}(\R^{N})}\leq C\beta \left( \int_{\R^{N}} \tilde{v}_{\e_{j}}^{2\beta-1} \, dx+\int_{\R^{N}} (\varphi(\tilde{v}_{\e_{j}}))^{2}\tilde{v}_{\e_{j}}^{2^{*}_{s}-2} \, dx \right),
\end{align}
where $C$ is a constant independent of $\beta$ and $j$. We also point out that the last integral in (\ref{Me1}) is well defined for every $T>0$ in the definition of $\varphi$.
Now we take $\beta$ in (\ref{Me1}) such that $2\beta-1=2^{*}_{s}$, and we denote it by 
\begin{equation}\label{Me2}
\beta_{1}=\frac{2^{*}_{s}+1}{2}.
\end{equation}
Let $R>0$ to be fixed later. Applying the H\"older inequality in the last integral in (\ref{Me1}), we can see that
\begin{align}\label{Me3}
&\int_{\R^{N}} (\varphi(\tilde{v}_{\e_{j}}))^{2} \tilde{v}_{\e_{j}}^{2^{*}_{s}-2}\, dx \nonumber \\
&=  \int_{\{\tilde{v}_{\e_{j}}\leq R\}} (\varphi(\tilde{v}_{\e_{j}}))^{2} \tilde{v}_{\e_{j}}^{2^{*}_{s}-2}\, dx
+\int_{\{\tilde{v}_{\e_{j}}> R\}} (\varphi(\tilde{v}_{\e_{j}}))^{2} \tilde{v}_{\e_{j}}^{2^{*}_{s}-2}\, dx \nonumber\\
&\leq  \int_{\{\tilde{v}_{\e_{j}}\leq R\}} \frac{(\varphi(\tilde{v}_{\e_{j}}))^{2}}{\tilde{v}_{\e_{j}}} R^{2^{*}_{s}-1}\, dx\nonumber \\
&\quad +\left(\int_{\R^{N}} (\varphi(\tilde{v}_{\e_{j}}))^{2^{*}_{s}} \, dx\right)^{\frac{2}{2^{*}_{s}}} \left(\int_{\{\tilde{v}_{\e_{j}}>R\}} \tilde{v}_{\e_{j}}^{2^{*}_{s}} \, dx\right)^{\frac{2^{*}_{s}-2}{2^{*}_{s}}}.
\end{align}
Since $(\tilde{v}_{\e_{j}})$ is bounded in $H^{s}(\R^{N})$, we can take $R$ sufficiently large such that
\begin{align*}
\left(\int_{\{\tilde{v}_{\e_{j}}>R\}} \tilde{v}_{\e_{j}}^{2^{*}_{s}} \, dx\right)^{\frac{2^{*}_{s}-2}{2^{*}_{s}}}\leq \frac{1}{2C \beta_{1}}.
\end{align*}
This together with (\ref{Me1}), (\ref{Me2}) and (\ref{Me3}), yields
\begin{align}\label{Me4}
\|\varphi(\tilde{v}_{\e_{j}})\|^{2}_{L^{2^{*}_{s}}(\R^{N})}\leq2 C\beta_{1}\left(\int_{\R^{N}} \tilde{v}_{\e_{j}}^{2^{*}_{s}} \, dx+R^{2^{*}_{s}-1}\int_{\R^{N}} \frac{\varphi(\tilde{v}_{\e_{j}})^{2}}{\tilde{v}_{\e_{j}}} \, dx \right).
\end{align}
From $\varphi(\tilde{v}_{\e_{j}})\leq \tilde{v}_{\e_{j}}^{\beta_{1}}$ and (\ref{Me2}), and taking the limit as $T\rightarrow \infty$ in (\ref{Me4}), we have
\begin{align*}
\left(\int_{\R^{N}} \tilde{v}_{\e_{j}}^{2^{*}_{s}\beta_{1}} \, dx\right)^{\frac{2}{2^{*}_{s}}}
\leq2 C\beta_{1}\left(\int_{\R^{N}} \tilde{v}_{\e_{j}}^{2^{*}_{s}} \, dx+R^{2^{*}_{s}-1}\int_{\R^{N}}   \tilde{v}_{\e_{j}}^{2^{*}_{s}} \, dx\right)<\infty,
\end{align*} 
which gives 
\begin{align}\label{Me5}
\tilde{v}_{\e_{j}}\in L^{2^{*}_{s}\beta_{1}}(\R^{N}).
\end{align}
Now we assume that $\beta>\beta_{1}$. Thus, using $\varphi(\tilde{v}_{\e_{j}})\leq \tilde{v}_{\e_{j}}^{\beta}$ on the right hand side of (\ref{Me1}) and letting $T\rightarrow \infty$ we deduce that
\begin{align}\label{Me6}
&\left(\int_{\R^{N}} \tilde{v}_{\e_{j}}^{2^{*}_{s}\beta} \, dx\right)^{\frac{2}{2^{*}_{s}}} \leq C\beta \left(\int_{\R^{N}} \tilde{v}_{\e_{j}}^{2\beta-1} \, dx+\int_{\R^{N}} \tilde{v}_{\e_{j}}^{2\beta+2^{*}_{s}-2} \, dx\right).
\end{align} 
Set
$$
a:=\frac{2^{*}_{s}(2^{*}_{s}-1)}{2(\beta-1)} \mbox{ and } b:=2\beta-1-a.
$$
Applying Young's inequality with exponents $r=\frac{2^{*}_{s}}{a}$ and $r'=\frac{2^{*}_{s}}{2^{*}_{s}-a}$, we can see that
\begin{align}\label{Me7}
\int_{\R^{N}} \tilde{v}_{\e_{j}}^{2\beta-1} \, dx&\leq \frac{a}{2^{*}_{s}} \int_{\R^{N}} \tilde{v}_{\e_{j}}^{2^{*}_{s}}\, dx+\frac{2^{*}_{s}-a}{2^{*}_{s}} \int_{\R^{N}} \tilde{v}_{\e_{j}}^{\frac{2^{*}_{s} b}{2^{*}_{s}-a}} \, dx \nonumber\\
&\leq \int_{\R^{N}} \tilde{v}_{\e_{j}}^{2^{*}_{s}}\, dx+\int_{\R^{N}} \tilde{v}_{\e_{j}}^{2\beta+2^{*}_{s}-2}\, dx \nonumber\\
&\leq C\left(1+\int_{\R^{N}}  \tilde{v}_{\e_{j}}^{2\beta+2^{*}_{s}-2}\, dx \right).
\end{align}
Putting together (\ref{Me6}) and (\ref{Me7}), we obtain
\begin{align}\label{Me8}
\left(\int_{\R^{N}} \tilde{v}_{\e_{j}}^{2^{*}_{s}\beta} \, dx\right)^{\frac{2}{2^{*}_{s}}}\leq C\beta \left(1+\int_{\R^{N}}  \tilde{v}_{\e_{j}}^{2\beta+2^{*}_{s}-2}\, dx \right).
\end{align}
Consequently,
\begin{align}\label{Me9}
\left(1+\int_{\R^{N}} \tilde{v}_{\e_{j}}^{2^{*}_{s}\beta} \, dx\right)^{\frac{1}{2^{*}_{s}(\beta-1)}}\leq (C\beta)^{\frac{1}{2(\beta-1)}} \left(1+\int_{\R^{N}}  \tilde{v}_{\e_{j}}^{2\beta+2^{*}_{s}-2}\, dx \right)^{\frac{1}{2(\beta-1)}}.
\end{align}
For $m\geq 1$ we define $\beta_{k+1}$ inductively so that $2\beta_{k+1}+2^{*}_{s}-2=2^{*}_{s}\beta_{k}$, that is
$$
\beta_{k+1}=\left(\frac{2^{*}_{s}}{2}\right)^{k}(\beta_{1}-1)+1.
$$
Hence, from \eqref{Me9}, it follows that
\begin{align}\label{Me10}
&\left(1+\int_{\R^{N}} \tilde{v}_{\e_{j}}^{2^{*}_{s}\beta_{k+1}} \, dx\right)^{\frac{1}{2^{*}_{s}(\beta_{k+1}-1)}}\nonumber \\
&\quad \leq (C\beta_{k+1})^{\frac{1}{2(\beta_{k+1}-1)}} \left(1+\int_{\R^{N}}  \tilde{v}_{\e_{j}}^{2^{*}_{s}\beta_{k}}\, dx \right)^{\frac{1}{2^{*}_{s}(\beta_{k}-1)}}.
\end{align}
Let us define
$$
A_{k}:=\left(1+\int_{\R^{N}} \tilde{v}_{\e_{j}}^{2^{*}_{s}\beta_{k}} \, dx\right)^{\frac{1}{2^{*}_{s}(\beta_{k}-1)}}
$$
and
$$
C_{k+1}:=C\beta_{k+1}.
$$
Then we can find a constant $c_{0}>0$ independent of $k$ such that 
$$
A_{k+1}\leq \prod_{m=2}^{k+1} C_{k}^{\frac{1}{2(\beta_{m}-1)}} A_{1}\leq c_{0} A_{1}.
$$
Hence, we can deduce that
$$
\|\tilde{v}_{\e_{j}}\|_{L^{\infty}(\R^{N})}\leq c_{0} A_{1}<\infty,
$$
uniformly in $j\in \N$, thanks to (\ref{Me5}) and $\|\tilde{v}_{\e_{j}}\|_{L^{2^{*}_{s}}(\R^{N})}\leq C$.
This ends the proof of Lemma \ref{Moser}.
\end{proof}

\noindent
Using Lemma \ref{Moser} and the interpolation in $L^{q}$ spaces, we can see that 
\begin{align}\label{moser}
&\tilde{v}_{\e_{j}}\rightarrow \omega^{1} \mbox{ in } L^{q}(\R^{N}),  \mbox{ for any } q\in (2, \infty),\\
&h_{j}(x)=g(\e_{j} x+\e_{j} y_{\e_{j}}, \tilde{v}_{\e_{j}})\rightarrow f(\omega^{1}) \mbox{ in } L^{q}(\R^{N}),  \mbox{ for any } q\in (2, \infty).
\end{align}
Now we note that $\tilde{v}_{\e_{j}}$ satisfies 
$$
(-\Delta)^{s}\tilde{v}_{\e_{j}}+\tilde{v}_{\e_{j}}=\alpha_{j} \mbox{ in } \R^{N},
$$
where $\alpha_{j}(x)=\tilde{v}_{\e_{j}}(x)+h_{j}(x)-V(\e_{j}x+\e_{j}y_{\e_{j}})\tilde{v}_{\e_{j}}(x)$.\\
In view of \eqref{6.1} and (\ref{moser}), we can deduce that 
$$
\alpha_{j}\rightarrow \omega^{1}+f(\omega^{1})-V(x^{1})\omega^{1} \mbox{ in } L^{q}(\R^{N})
$$
for any $q\in [2, \infty)$, and we can find a constant $\kappa>0$ such that 
$$
\|\alpha_{j}\|_{L^{\infty}(\R^{N})}\leq \kappa \mbox{ for all } j\in \N.
$$
Taking into account some results obtained in \cite{FQT}, we know that
$$
\tilde{v}_{\e_{j}}(x)=(\mathcal{K}*\alpha_{j})(x)=\int_{\R^{N}} \mathcal{K}(x-y) \alpha_{j}(y) \, dy,
$$
where $\mathcal{K}$ is the Bessel kernel.
Then we can argue as in the proof of Lemma $2.6$ in \cite{AM} to infer that 
\begin{equation}\label{AM3}
\tilde{v}_{\varepsilon_{j}}(x)\rightarrow 0  \mbox{ as } |x|\rightarrow \infty  
\end{equation}
uniformly in $j\in \N$.
Now we prove that $\tilde{v}_{\e_{j}}$ is a solution to (\ref{P}) for small $\e_{j}>0$.

Using the fact that $\e_{j} y_{\e_{j}}\rightarrow x^{1}\in \Lambda'$, there exists $r>0$ such that for some subsequence, still denoted by itself, we have
$$
B_{r}(\e_{j} y_{\e_{j}})\subset  \Lambda' \mbox{ for all } j\in \N.
$$
By setting $\Lambda'_{\e}=\frac{\Lambda'}{\e}$, we can see that
$$
B_{\frac{r}{\e_{j}}}(y_{\e_{j}})\subset  \Lambda'_{\e_{j}} \mbox{ for all } j\in \N
$$
which yields
$$
\R^{N}\setminus \Lambda'_{\e_{j}} \subset \R^{N}\setminus  B_{\frac{r}{\e_{j}}}(y_{\e_{j}}) \mbox{ for all } j\in \N.
$$
From (\ref{AM3}), there exists $R>0$ such that 
$$
\tilde{v}_{\varepsilon_{j}}(x)<r_{\nu} \mbox{ for all } |x|\geq R, j\in \N
$$
so that
$$
v_{\varepsilon_{j}}(x)=\tilde{v}_{\varepsilon_{j}}(x-y_{\e_{j}})<r_{\nu} \mbox{ for all } x\in \R^{N}\setminus  B_{R}(y_{\e_{j}}), j\in \N.
$$
On the other hand, there exists $j_{0}\in \N$ such that 
$$
\R^{N}\setminus \Lambda'_{\e_{j}}\subset \R^{N}\setminus  B_{\frac{r}{\e_{j}}}(y_{\e_{j}})\subset \R^{N}\setminus  B_{R}(y_{\e_{j}}) \mbox{ for all } j\geq j_{0}.
$$
Hence 
\begin{equation}\label{HZ1T}
v_{\varepsilon_{j}}(x)<r_{\nu} \mbox{ for all } x\in \R^{N}\setminus \Lambda'_{\e_{j}}, j\geq j_{0}.
\end{equation}
Now, up to a subsequence, we may assume that 
\begin{equation}\label{HZ2T}
\|v_{\varepsilon_{j}}\|_{L^{\infty}(B_{R}(y_{\e_{j}}))}\geq r_{\nu} \mbox{ for all }  j\geq j_{0}.
\end{equation}
Otherwise, if this is not the case, we have $\|v_{\varepsilon_{j}}\|_{L^{\infty}(\R^{N})}<r_{\nu}$, and taking into account the definition of $g$ and our choice of $r_{\nu}$, we get
$$
g(\e_{j} x, v_{\e_{j}})v_{\e_{j}}=f(v_{\e_{j}})v_{\e_{j}}\leq \nu v^{2}_{\e_{j}}<\frac{V_{0}}{2} v^{2}_{\e_{j}}. 
$$
Then, by $\langle J_{\e_{j}}'(v_{\e_{j}}), v_{\e_{j}}\rangle=0$ we can deduce that
$$
\|v_{\e_{j}}\|_{H^{s}_{\e_{j}}}^{2}=\int_{\R^{N}} f(v_{\e_{j}})v_{\e_{j}}\, dx\leq \frac{V_{0}}{2} \int_{\R^{N}} v^{2}_{\e_{j}}\, dx
$$
which implies that $\lim_{j\rightarrow \infty} \|v_{\e_{j}}\|_{H^{s}_{\e_{j}}}^{2}=0$, which is a contradiction in view of (\ref{vincenzo}).
Therefore, putting together (\ref{HZ1T}) and (\ref{HZ2T}), we deduce that the maximum points $z_{\e_{j}}\in \R^{N}$ of $v_{\e_{j}}$ belong to $B_{R}(y_{\e_{j}})$. Hence $z_{\e_{j}}=y_{\e_{j}}+\bar{z}_{\e_{j}}$, for some $\bar{z}_{\e_{j}}\in B_{R}$. Recalling that the associated solution of our problem (\ref{P}) is of the form $u_{\e_{j}}(x)=v_{\e_{j}}(\frac{x}{\e_{j}})$, we can conclude that the maximum point $x_{\e_{j}}$ of $u_{\e_{j}}$ is $x_{\e_{j}}:=\e_{j} y_{\e_{j}}+\e_{j} \bar{z}_{\e_{j}}$.
Since $(\bar{z}_{\e_{j}})\subset B_{R}$ is bounded and $\e_{j} y_{\e_{j}}\rightarrow x^{1}\in \Lambda'$ we obtain
$$
\lim_{j\rightarrow \infty} V(x_{\e_{j}})=V(x^{1})=\inf_{x\in \Lambda} V(x).
$$
Therefore, we have proved that there exists $\e_{0}>0$ such that for any $\e\in (0, \e_{0}]$, (\ref{P}) admits a positive solution $u_{\e}(x)=v_{\e}(\frac{x}{\e})$ satisfying $(1)$ of Theorem \ref{teorema principale}. 
Finally, we prove that $(2)$ holds.
Using Lemma $4.3$ in \cite{FQT} we know that there exists a function $w$ such that 
\begin{align}\label{HZ1}
0<w(x)\leq \frac{C}{1+|x|^{N+2s}},
\end{align}
and
\begin{align}\label{HZ2}
(-\Delta)^{s} w+\frac{V_{0}}{2}w\geq 0 \mbox{ in } \R^{N}\setminus B_{R_{1}}, 
\end{align}
for some suitable $R_{1}>0$. In view of (\ref{AM3}), we know that $\tilde{v}_{\e_{j}}(x)\rightarrow 0$ as $|x|\rightarrow \infty$ uniformly in $j$.
This, $(f2)$ and the definition of $g$, implies that for some $R_{2}>0$ sufficiently large, we get
\begin{align}\label{HZ3}
(-\Delta)^{s} \tilde{v}_{\e_{j}}+\frac{V_{0}}{2} \tilde{v}_{\e_{j}} 
&=g(\e_{j} x+\e_{j}y_{\e_{j}}, \tilde{v}_{\e_{j}})-\left(V-\frac{V_{0}}{2}\right)\tilde{v}_{\e_{j}} \nonumber\\
&\leq g(\e_{j} x+\e_{j}y_{\e_{j}}, \tilde{v}_{\e_{j}})-\frac{V_{0}}{2}\tilde{v}_{\e_{j}}\leq 0 \mbox{ in } \R^{N}\setminus B_{R_{2}}. 
\end{align}
Choose $R_{3}=\max\{R_{1}, R_{2}\}$, and we set 
\begin{align}\label{HZ4}
a=\inf_{B_{R_{3}}} w>0 \mbox{ and } \tilde{w}_{\e_{j}}=(b+1)w-a\tilde{v}_{\e_{j}},
\end{align}
where $b=\sup_{j\in \N} \|\tilde{v}_{\e_{j}}\|_{L^{\infty}(\R^{N})}<\infty$. 
Now we prove that 
\begin{equation}\label{HZ5}
\tilde{w}_{\e_{j}}\geq 0 \mbox{ in } \R^{N}.
\end{equation}
We first note that \eqref{HZ2}, \eqref{HZ3} and \eqref{HZ4} yield
\begin{align}
&\tilde{w}_{\e_{j}}\geq ba+w-ba>0 \mbox{ in } B_{R_{3}} \label{HZ0},\\
&(-\Delta)^{s} \tilde{w}_{\e_{j}}+\frac{V_{0}}{2}\tilde{w}_{\e_{j}}\geq 0 \mbox{ in } \R^{N}\setminus B_{R_{3}} \label{HZ00}.
\end{align}
We argue by contradiction, and we assume that there exists a sequence $(\bar{x}_{j, n})\subset \R^{N}$ such that 
\begin{align}\label{HZ6}
\inf_{x\in \R^{N}} \tilde{w}_{\e_{j}}(x)=\lim_{n\rightarrow \infty} \tilde{w}_{\e_{j}}(\bar{x}_{j, n})<0. 
\end{align}
Using (\ref{AM3}), \eqref{HZ1} and the definition of $\tilde{w}_{\e_{j}}$, it is clear that $|\tilde{w}_{\e_{j}}(x)|\rightarrow 0$ as $|x|\rightarrow \infty$, uniformly in $j\in \N$. Thus we can deduce that $(\bar{x}_{j, n})$ is bounded, and, up to subsequence, we may assume that there exists $\bar{x}_{j}\in \R^{N}$ such that $\bar{x}_{j, n}\rightarrow \bar{x}_{j}$ as $n\rightarrow \infty$. 
Thus from (\ref{HZ6}), we get
\begin{align}\label{HZ7}
\inf_{x\in \R^{N}} \tilde{w}_{\e_{j}}(x)= \tilde{w}_{\e_{j}}(\bar{x}_{j})<0.
\end{align}
From the minimality of $\bar{x}_{j}$ and the representation formula for the fractional Laplacian \cite{DPV}, we can see that 
\begin{align}\label{HZ8}
(-\Delta)^{s}\tilde{w}_{\e_{j}}(\bar{x}_{j})=\frac{C(N, s)}{2} \int_{\R^{N}} \frac{2\tilde{w}_{\e_{j}}(\bar{x}_{j})-\tilde{w}_{\e_{j}}(\bar{x}_{j}+\xi)-\tilde{w}_{\e_{j}}(\bar{x}_{j}-\xi)}{|\xi|^{N+2s}} d\xi\leq 0.
\end{align}
Taking into account (\ref{HZ0}) and (\ref{HZ6}), we can infer that $\bar{x}_{j}\in \R^{N}\setminus B_{R_{3}}$.
This, together with (\ref{HZ7}) and (\ref{HZ8}), yields 
$$
(-\Delta)^{s} \tilde{w}_{\e_{j}}(\bar{x}_{j})+\frac{V_{0}}{2}\tilde{w}_{\e_{j}}(\bar{x}_{j})<0,
$$
which contradicts (\ref{HZ00}).
Thus (\ref{HZ5}) holds, and using (\ref{HZ1}) we get
\begin{align}\label{HZ9}
\tilde{v}_{\e_{j}}(x)\leq \frac{\tilde{C}}{1+|x|^{N+2s}} \mbox{ for all } j\in \N, x\in \R^{N},
\end{align}
for some $\tilde{C}>0$.
Since $u_{\e_{j}}(x)=v_{\e_{j}}(\frac{x}{\e_{j}})=\tilde{v}_{\e_{j}}(\frac{x}{\e_{j}}-y_{\e_{j}})$ and $x_{\e_{j}}=\e_{j}y_{\e_{j}}+\e_{j} \bar{z}_{\e_{j}}$, from (\ref{HZ9}) we obtain for any $x\in \R^{N}$
\begin{align*}
u_{\e_{j}}(x)&=v_{\e_{j}}\left(\frac{x}{\e_{j}}\right)=\tilde{v}_{\e_{j}}\left(\frac{x}{\e_{j}}-y_{\e_{j}}\right) \\
&\leq \frac{\tilde{C}}{1+|\frac{x}{\e_{j}}-y_{\e_{j}}|^{N+2s}} \\
&=\frac{\tilde{C} \e_{j}^{N+2s}}{\e_{j}^{N+2s}+|x- \e_{j} y_{\e_{j}}|^{N+2s}} \\
&\leq \frac{\tilde{C} \e_{j}^{N+2s}}{\e_{j}^{N+2s}+|x-x_{\e_{j}}|^{N+2s}}.
\end{align*}
This ends the proof of Theorem \ref{teorema principale}.

\smallskip

\noindent
{\bf Acknowledgements.} 
The author warmly thanks the anonymous referee for her/his useful and nice comments on the paper. 
The manuscript has been carried out under the auspices of the INDAM - Gnampa Project 2017 titled:{\it Teoria e modelli per problemi non locali}.

%%%%%%%%%%%%%%%%%%%%%%%%%%%%%%%%%
% Bibliography
%%%%%%%%%%%%%%%%%%%%%%%%%%%%%%%%%

% Example for an article published in a journal
%\bibitem[]{}
%\RMIauthor{}
%\RMIpaper{}
%\RMIjournal{} \textbf{vol} (year), no. [number], initial--final pp.

% For a book
%\bibitem[]{}
%\RMIauthor{}
%\RMIbook{}
%Editorial, [other data], pub. year.

% For a chapter in a book
%\bibitem[]{}
%\RMIauthor{}
%\RMIpaper{}
% In %\RMIbook{}, initial--final pp.
%Editorial, [other data], pub. year.

%%%%%%%%%%%%%%%%%%%%%%%%%%%%%%%%%
% Projects and funding
% Uncomment and fill in with the information if needed
%%%%%%%%%%%%%%%%%%%%%%%%%%%%%%%%%

\end{document}